\documentclass[reqno]{amsart}
\usepackage{tkz-euclide}
\usepackage{hyperref}
\usepackage{tikz-3dplot}
\usepackage{amsmath}  
\usepackage{amsfonts}
\usepackage{amssymb} 
\usepackage{mathrsfs} 
\usepackage{graphicx}  
\usepackage{bm}
\usepackage{tikz} 
\usetikzlibrary{calc} 
\usepackage{xcolor}  
\usetikzlibrary{arrows.meta} 
\usepackage{amsthm} 
\usepackage{float}
\usepackage{enumitem}
\usepackage[english]{babel}
\usepackage[font=footnotesize, labelfont=bf]{ caption}
\usepackage{ esint }
\usepackage{stackengine,scalerel,graphicx}
\stackMath

\definecolor{trueblue}{rgb}{0.0, 0.45, 0.81}
\definecolor{ForestGreen}{rgb}{0.01, 0.75, 0.24}

\newcommand{\ZZZ}{\color{blue}} 
 
\newcommand{\BBB}{\color{black}}

\newcommand{\EEE}{\color{black}} 
 
\newcommand{\OOO}{\color{orange}} 
\newcommand{\AAA}{\color{black}}

\newcommand{\eps}{\varepsilon}

\theoremstyle{plain}
\begingroup
\newtheorem{theorem}{Theorem}[section]
\newtheorem{lemma}[theorem]{Lemma}

\newtheorem{remark}[theorem]{Remark}
\newtheorem{proposition}[theorem]{Proposition}

\endgroup

\newenvironment{step}[1]{\underline{Step #1}.}{}

\theoremstyle{definition}
\begingroup
\newtheorem{definition}[theorem]{Definition}
\endgroup

\renewcommand{\tilde}{\widetilde}

\renewcommand{\d}{ \mathrm{d}}

\DeclareMathOperator{\dist}{dist}

\numberwithin{equation}{section}
\newcommand{\N}{\mathbb{N}}
\newcommand{\Z}{\mathbb{Z}}
\newcommand{\R}{\mathbb{R}}

\renewcommand{\S}{\mathbb{S}}
\renewcommand{\L}{\mathcal{L}}
\renewcommand{\H}{\mathcal{H}}

\newcommand{\Q}{\mathcal{Q}}

\newcommand{\M}{\mathfrak{M}}

\usepackage[normalem]{ulem}

\begin{document}

\title[From atomistic systems to continuum \EEE models for elastic materials with voids]{From atomistic systems to  linearized continuum \EEE models for elastic materials with voids}

\author[M. Friedrich]{Manuel Friedrich} 
\address[Manuel Friedrich]{Department of Mathematics, Friedrich-Alexander Universit\"at Erlangen-N\"urnberg. Cauerstr.~11,
D-91058 Erlangen, Germany, \& Mathematics M\"{u}nster,  
University of M\"{u}nster, Einsteinstr.~62, D-48149 M\"{u}nster, Germany}
\email{manuel.friedrich@fau.de}

\author{Leonard Kreutz}
\address[Leonard Kreutz]{Applied Mathematics M\"unster, University of M\"unster\\
Einsteinstrasse 62, 48149 M\"unster, Germany}
\email{lkreutz@uni-muenster.de}

\author{Konstantinos Zemas}
\address[Konstantinos Zemas]{Applied Mathematics M\"unster, University of M\"unster\\
Einsteinstrasse 62, 48149 M\"unster, Germany}
\email{konstantinos.zemas@uni-muenster.de}

\begin{abstract}
We study an atomistic model that describes the microscopic formation of material voids inside elastically stressed solids under  an additional curvature regularization at the discrete level. Using a discrete-to-continuum analysis, by means of 
a recent geometric rigidity result in variable domains \cite{KFZ:2021} \EEE and $\Gamma$-convergence tools, we rigorously derive effective linearized continuum models for elastically stressed solids with material voids in three-dimensional elasticity.
\end{abstract}

\BBB


 \EEE

\maketitle

\section{Introduction}\label{Introduction}

In the last years, the understanding of \emph{stress driven rearrangement instabilities} (SDRI) has attracted huge interest both from the physics and the mathematics community \cite{Bourdin-Francfort-Marigo:2008,Grin86,Grin93,KhoPio19,SieMikVoo04}. These morphological instabilities of interfaces are generated by the competition between elastic bulk and surface
energies, including many different phenomena such as brittle fracture, formation of material voids
inside elastically stressed solids, or hetero-epitaxial growth of elastic thin films.

From a mathematical point of view, the common feature of functionals describing SDRI is the
presence of both stored elastic bulk and surface energies. In the static setting, problems arise
concerning existence, regularity, and stability of equilibrium configurations obtained by energy
minimization. 
In the framework of linearized elasticity, these issues are by now mostly dealt with in dimension two \cite{FonFusLeoMil11, FonFusLeoMor07, KhoPio19,KhoPio20} and only recently in dimension three  \cite{Crismale,KFZ:2021}. In this work, we focus on the formation of material voids inside elastically stressed solids. At the continuum level and in the regime of linear elasticity\EEE, the variational formulation consists in considering functionals defined on pairs of function-set, namely 
\begin{align}\label{intro eq:F}
F(u,E) := \int_{\Omega\setminus E} \mathbb{C}\, e(u): e(u)\,\mathrm{d}x + \int_{\partial E \cap \Omega} \varphi(\nu_E)\,\mathrm{d}\mathcal{H}^{d-1}\,,
\end{align}
where $E \subset \Omega$ represents the (sufficiently regular) void set within an elastic body with reference configuration $ \Omega \subset \R^d \EEE$, \EEE and $u$ is the corresponding elastic displacement field. The first part of the functional represents the elastic energy depending on the linear strain $e(u) :=\frac{1}{2}(\nabla u+(\nabla u)^{ T}\EEE)$, 
where $\mathbb{C} \colon\R^{d\times d}\mapsto\R^{d\times d}$ denotes the fourth-order positive semi-definite tensor of elasticity coefficients, whose kernel contains 
\EEE the subspace of skew-symmetric matrices.  The surface energy depends on
a possibly anisotropic density $\varphi$ evaluated at the outer  unit normal $\nu_E$ to $ \partial E\cap\Omega$.

From an analytical point of view, in order to guarantee existence of equilibrium configurations for Dirichlet boundary value problems, it is necessary to consider an effective relaxation of the energy given in \eqref{intro eq:F}. Indeed, even though for fixed $E$ the functional  $F(\cdot, E)$ is weakly lower semicontinuous in $H^1(\Omega\setminus E;\R^d)$ and, for fixed $u$, $F(u, \cdot)$ can be regarded as a lower semicontinuous functional on sets of
finite perimeter with respect to the $L^1$-convergence of sets, the energy defined on pairs $(u, E)$ is not lower semicontinuous. This is due to the fact that, along a sequence, the voids may collapse in the limit into a discontinuity of the limiting displacement, which lies in the space of \textit{generalized functions of bounded deformation} ($GSBD$) \cite{DalMaso:13}. \EEE The relaxation has to take this phenomenon into account, in particular collapsed surfaces need to be
counted twice in the relaxed energy. Another  interesting and challenging  question is to establish a link between the relaxation of \eqref{intro eq:F} and the corresponding model of nonlinear  elasticity via a simultaneous linearization-relaxation procedure. In that situation, the linear elastic energy  is replaced by an integral functional of a nonlinear elastic energy density of the full deformation gradient $\nabla y$ (where $y\colon\Omega\setminus E\to \R^d$ represents the deformation), which is rotationally invariant and is minimized on $SO(d)$. This analysis has been carried out in the physically relevant dimensions $d=2,3$ under the addition of a 
curvature regularization term in \cite{KFZ:2021}, where the main novelty is a \emph{quantitative geometric rigidity result for variable domains} (see Theorem 2.1 therein). 

The scope of this paper is to derive the relaxation of the model in \eqref{intro eq:F} as an effective continuum limit of physically relevant atomistic models. We will present here only the case $d=3$, since the case $d=2$ is completely analogous, and some arguments along the proofs can actually even be simplified. The passage from atomistic to continuum models (see \cite{BraGel}) via $\Gamma$-convergence \cite{Braides:02,DalMaso:93} has been carried out in various contexts, including elasticity \EEE \cite{AliCic, BraKre, Braides-Solci-Vitali:07, BraSch, Schmidt:2009}, fracture \cite{AliFocGel, BraLewOrt, Crismale-Scilla-Solombrino, FriSch1,FriSch2,FriSch3, Neg, Scardia:2010}, or more 
general \EEE problems containing free discontinuities \cite{BraBacCic,KrePio19,Ruf}. 

The analysis in some of these contributions, see \cite{BraLewOrt, FriSch3, KrePio19}, \EEE shows that for standard mass-spring models (e.g., governed by Lennard-Jones-type interactions) the effective  bulk and surface energies are of the same order only in the case of infinitesimal elastic strains. In this case, the  discrete-to-continuum limits are coupled with a \EEE simultaneous passage from nonlinear to linear elasticity. We follow the approach of {\sc Schmidt}  \cite{Schmidt:2009}, which can be motivated by regrouping the interactions of a mass-spring model, cf.~for instance~\cite{Braides-Solci-Vitali:07}. More precisely, \EEE given the reference set $\Omega$, we denote by $\varepsilon >0$ the lattice spacing of the underlying reference lattice $\mathbb{Z}_\varepsilon(\Omega):=\varepsilon\mathbb{Z}^3\cap \Omega$ and  by  $\delta_\varepsilon>0$  the linearization parameter depending on the lattice spacing, which represents the typical order of the elastic strain and satisfies $\delta_\eps \to 0$ as $\eps \to 0$. \EEE  Here, we would like to point out that our analysis extends to general Bravais lattices. For expository reasons, however,  we only consider the case of $\mathbb{Z}^3$. To a given pair $(y,E)$, where the lattice points $E \subset \mathbb{Z}_\varepsilon(\Omega)$  represent  the presence of voids inside the material at the microscopic level and $y \colon \mathbb{Z}_\varepsilon(\Omega)\setminus E \to \mathbb{R}^3$ represents its deformation\EEE, we associate an elastic energy
\begin{align}\label{intro eq: discrete elastic energy}
F_\varepsilon^{\mathrm{el}}(y,E) := \delta_\varepsilon^{-2}\sum_{i \in \mathbb{Z}_\varepsilon(\Omega)\setminus E} \varepsilon^{3} W_{\varepsilon, \mathrm{cell}}^{\mathrm{el}}(i,y,E)\,,
\end{align} 
where the \emph{cell energy} $W_{\varepsilon, \mathrm{cell}}^{\mathrm{el}}(i,y,E)$ (specified more precisely in Subsection \ref{Assumptions on the elastic energy}) represents the elastic energy needed in order to locally deform the neighborhood of $i \in \mathbb{Z}_\varepsilon(\Omega) \setminus E$ through the deformation $y$. Additionally, the presence of voids induces an extra energy contribution of \emph{discrete-perimeter type} due to the missing bonds at the boundary of the configuration, namely, 
\begin{align}\label{introeq: discrete perimeter}
F_\varepsilon^{\mathrm{ per }}(E):= \sum_{i \in 
E}\ \underset{i+\varepsilon\xi \in \mathbb{Z}_\varepsilon(\Omega)}{\sum_{\xi \in V}}\varepsilon^{2} c_\xi \big(1-\chi_{E}(i+\varepsilon\xi)\big)\,,
\end{align}
where $V \subset \mathbb{Z}^3$ describes the set of neighbors of a lattice point (possibly containing also next-nearest and next-next-nearest neighbors) and $c_\xi >0$ is the energy cost ``per unit area'' in the reference lattice with respect to the neighboring point $i+\varepsilon \xi$. This energy keeps track of the local valence of each atom in $ E$, i.e., the number of atoms missing in order to complete its  neighborhood.

In \cite{KrePio19}, the discrete-to-continuum limit of the energy given \EEE by the sum of the two terms in \eqref{intro eq: discrete elastic energy} and \eqref{introeq: discrete perimeter}  has been derived in  the two-dimensional setting of elastic thin films, which from a modeling viewpoint, in the continuum setting, corresponds to the simplifying assumption that in \eqref{intro eq:F} the void $E$ is given by a supergraph of a function. This analysis also includes the reasoning why the elastic strain needs to be of order $\delta_\eps \sim \sqrt{\eps}$, as then \eqref{intro eq: discrete elastic energy} and \eqref{introeq: discrete perimeter} are of the same order in $\eps$, see also \cite{BraLewOrt, FriSch3}. The main ingredient for the proof in \cite{KrePio19} is a (by now classical) rigidity estimate, see \cite{FrieseckeJamesMueller:02, Schmidt:2009}, which allows  to infer that configurations $y$ with bounded elastic energy have to be close in a suitable quantitative  sense to a global rotation on the set $\mathbb{Z}_\varepsilon(\Omega)\setminus E$.  This fact then permits to linearize around the respective rotation and in this way to obtain a linearized elastic energy in the limit. However, the analogous rigidity result in dimension $d=3$, or for the more general problem of elastic materials with voids, is missing due to the possibly highly complex geometry  of the interfaces between bulk and void. In \cite{KFZ:2021}, this issue has been circumvented in a continuum  setting by considering an additional curvature regularization term in the energy fuctional, allowing for a novel, refined rigidity estimate (cf.~Theorem 2.1 therein). In order to employ this result in the present paper for our discrete-to-continuum analysis, we add an additional \emph{discrete curvature term} in the functional, of the form  
\begin{align}\label{intro eq: curvature energy}
F_\varepsilon^{\mathrm{curv}}( E) := \sum_{i \in \mathbb{Z}_{\varepsilon}( \Omega)\EEE} \varepsilon^3 W_{\varepsilon,\mathrm{cell}}^{\mathrm{curv}}(i, E)\,.
\end{align}
The specific form of the curvature term is introduced in Subsection \ref{subsec:surface_energy}. 
It \EEE can be understood as a suitable discretization of an integral of the second fundamental form of surfaces on a mesoscopic scale $\eta_\eps>0$, satisfying $\eps \ll \eta_\eps \ll 1$. Moreover, we justify the use of this term by presenting an example in the framework of the \emph{Embedded-Atom Model} (EAM), see Subsection \ref{example:EAM}. This is a semi-empirical, many-atom potential aiming at describing the atomic structure of metallic systems by including a nonlocal electronic correction \cite{Daw1, Daw2}. \EEE The total energy of a  pair $(y,E)$ is then given by
\begin{align*}
F_\varepsilon(y,E) := F_\varepsilon^{\mathrm{el}}(y,E)+ F_\varepsilon^{\mathrm{per}}(E) + F_\varepsilon^{\mathrm{curv}}( E)\,,
\end{align*}
where the  terms are introduced in \eqref{intro eq: discrete elastic energy}--\eqref{intro eq: curvature energy}, 
respectively. \EEE Under suitable scaling assumptions on the curvature 
energy and under prescribed Dirichlet boundary conditions, we show that the energies $(F_{\varepsilon})_{\varepsilon>0}$  $\Gamma$-converge with respect to an appropriate topology to the relaxation of \eqref{intro eq:F}. We refer to Subsection \ref{subsec: main result} for the detailed statements, and in particular to \eqref{def:F} for the definition of the relaxed energy. Let us highlight that, as an additional feature of possibly independent interest, our effective  continuum and  linearized limit shows the validity of the so-called \emph{Cauchy-Born rule} \cite{kirchheim, FT}, namely that, loosely speaking, individual atoms follow the macroscopic deformation gradient. Mathematically, this corresponds to the fact that suitably defined discrete gradients of the atomic displacements reduce to classical gradients in the continuum limit \cite{Schmidt:2009}.

The proof follows a semi-discrete approach. First of all, to each pair $(y, E)$ we associate a piecewise constant interpolation of $y$ on $\eps$-cubes and a continuum set $Q_\eps(E)$ being the union of $\eps$-cubes, i.e., $Q_\varepsilon(E)=\bigcup_{i \in E} (i+ \eps\EEE[-\frac{1}{2},\frac{1}{2})^3)$. In order to apply the results obtained in \cite{KFZ:2021}, it is necessary to  modify and smoothen the continuum void set $Q_\varepsilon(E)$, see  Lemmata \ref{lemma : eta lattice}--\ref{lemma:eta scale smoothening}. More precisely, we first replace \EEE
$Q_\varepsilon(E)$ by a set $E_{\eta_\eps}$, being a union of $\eta_\eps$-cubes, where $\eta_\eps>0$ is the mesoscopic scale related to the curvature term \eqref{intro eq: curvature energy}. Afterwards, we smoothen the set $E_{\eta_\eps}$ in a suitable way to pass to a smooth set $\hat{E}_\eps$. The core of  our proof lies in the fact that the specific form of \eqref{intro eq: curvature energy} ensures that our modifications on the $\eta_\eps$-scale are energetically convenient and lead to sharp estimates on the perimeter and curvature energy of $\hat{E}_\eps$. After having  associated to the pair $(y,E)$ an interpolation and a smooth continuum set, we are in the position to apply the compactness result for materials with voids in the continuum setting, see  \cite[Proposition 3.1]{KFZ:2021}. We point out that this result yields compactness for \emph{rescaled displacement fields} $u := \delta_\eps^{-1}(y-{\rm id})$  in the space $GSBD$ \cite{DalMaso:13}, and that its proof fundamentally relies on the rigidity estimate in \cite[Theorem 2.1]{KFZ:2021}. At this point, a further step  consists in proving the validity of the Cauchy-Born rule in the setting of $GSBD$ functions. As the analogous result for Sobolev functions \cite{Schmidt:2009} is not directly applicable, beforehand we need to perform a delicate approximation of $GSBD$ functions by Sobolev functions, using recent blow-up techniques for $GSBD$ functions \cite{Conti-Focardi-Iurlano:15, newvito, Crismale-Friedrich-Solombrino, FriSolPer} and generalizing them to a discrete setting. \EEE

With the compactness result and the Cauchy-Born rule in $GSBD$ at hand, the $\Gamma$-liminf inequality for the elastic energy follows from  standard arguments by employing the strategy devised in \cite{Schmidt:2009}. For the surface energy, it is essential to perform the above modifications carefully, in order to have a sharp estimate on the continuum (anisotropic) surface energy in terms of the discrete perimeter and discrete curvature energy. This, together with the lower bound for the elastic energy, allows to conclude the $\Gamma$-liminf inequality. The $\Gamma$-limsup is performed via a density argument, which allows to reduce the problem to constructing recovery sequences for smooth functions and sets. 

The paper is structured as follows. After fixing some notation at the end of the Introduction, we present our model and the main results in Section \ref{sec: 2}. In particular, we introduce the various terms of the discrete energy functional in Subsections \ref{subsection 2.1}--\ref{subsec:surface_energy}. Our main results are stated in Subsection~\ref{subsec: main result} and their proofs are the content of Sections \ref{sec: 3}--\ref{sec: 6}. In particular, in Section \ref{sec: 3} we prove the compactness result,  while in Section \ref{sec: 4} we address the Cauchy-Born rule for discrete symmetrized gradients. Since this result can be of independent interest when dealing with discrete-to-continuum problems for elastic materials with surface 
discontinuities, \EEE we formulate and prove it in any dimension. \EEE Sections \ref{sec: 5}--\ref{sec: 6} are devoted to the proofs of the $\Gamma$-liminf \EEE and the $\Gamma$-limsup \EEE inequality, respectively.

\subsection*{Notation}\label{sec:notation}
We close the Introduction with some basic notation.  Let $d\in \N$, $d\geq 2$, denote the dimension of the ambient Euclidean space. For the most part of the paper we will focus on the case $d=3$, except for Section \ref{sec: 4}, the content of which is for any dimension $d$. \EEE Given an open bounded Lipschitz set \EEE $\Omega \subset \R^{ d\EEE}$, we denote by $\mathfrak{M}(\Omega)$ the collection of all measurable subsets of $\Omega$. By $\mathcal{A}(\Omega)$ we indicate the collection of open subsets of $\Omega$ and by $\mathcal{A}_{\rm reg}(\Omega)$ the subfamily consisting of those open sets $E \subset \Omega$ such that $\partial E \cap \Omega $ is a  ${ (d-1)\EEE}$-dimensional $C^2$-submanifold of $\mathbb{R}^{ d\EEE}$. Manifolds and functions of $C^2$-regularity will be called \emph{regular} in the following. Given $A \in \mathfrak{M}(\Omega)$, we denote by ${\rm int}(A)$  its interior, by $\overline{A}$ its closure, and by  $A^{\rm c}:= \Omega\setminus A$  its  complement  in $\Omega$. The diameter of $A$  is denoted by ${\rm diam}(A)$.  Moreover, for $r>0$ we let
\begin{align}\label{eq: thick-def}
(A)_r := \lbrace x\in \R^{ d\EEE}\colon \, \dist(x,A) < r\rbrace\,.
\end{align}  
Given $A, B \in \mathfrak{M}(\Omega)$, we write $A \subset \subset B$ if $\overline{A} \subset B$. The Hausdorff distance of $A$ and $B$ is denoted by ${\rm dist}_\mathcal{H}(A,B)$ and we write $A\triangle B := (A\setminus B) \cup (B\setminus A)$ for the symmetric difference. Given $A \in \mathfrak{M}(\Omega)$, we denote by $\chi_{A}$ the characteristic function of the set $A$, defined by $\chi_{A} =1$ on $A$ and $\chi_{A}=0$ otherwise. By $\mathcal{L}^{ d\EEE}$ and $\mathcal{H}^{ d-1\EEE}$ we denote the ${ d\EEE}$-dimensional Lebesgue measure and the ${ (d-1)\EEE}$-dimensional Hausdorff measure, respectively.  Given $x_0\in\R^{ d\EEE}$ and $\rho>0$, we denote by $B_{\rho}(x_0)$ the open ball in $\R^{ d\EEE}$ centered at $x_0$ of radius $\rho$, and we set for brevity $B_1:=B_1(0)$. By ${\rm id}$ we denote the identity mapping on $\R^{ d\EEE}$ and by ${\rm Id} \in \R^{{ d\EEE}\times { d\EEE}}$  the identity matrix. For each $F \in \R^{{ d\EEE} \times { d\EEE}}$, we let ${\rm sym}(F) := \frac{1}{2}\left(F+F^T\right)$, and $SO({ d\EEE}) := \lbrace F\in\R^{{ d\EEE} \times { d\EEE}}\colon F^TF = {\rm Id}, \, \det F = 1\rbrace$. Moreover, we denote by $\R^{{ d\EEE} \times { d\EEE}}_{\rm sym}$ and $\R^{{ d\EEE} \times { d\EEE}}_{\rm skew}$ the set of symmetric and skew-symmetric matrices, respectively. We further write $\mathbb{S}^{{d-1}} := \lbrace \nu \in \R^{ d\EEE}\colon \, |\nu|=1\rbrace$. 

Let us also denote by $\mathbb{Z}^{ d\EEE}$ the integer lattice in ${ d\EEE}$ dimensions, i.e.,
$$\mathbb{Z}^{ d\EEE}:=\{\lambda_1e_1+ \dots
\EEE+\lambda_{ d\EEE}e_{ d\EEE}\colon \ \lambda_1,\dots\EEE
,\lambda_{ d\EEE}\in \mathbb{Z}\}\,,$$
where $\{e_1,\dots \EEE
,e_{ d\EEE}\}$ denotes the standard orthonormal basis in $\mathbb{R}^{ d\EEE}$. We also set 
\begin{equation}\label{Z_points}
Z:=(z_1,\dots,z_{{ 2^d\EEE}})\in \mathbb{R}^{{ d\EEE}\times { 2^d\EEE}}, \ \ \mathrm{where} \ \ \mathcal{Z}:=\{z_1,\dots,z_{{ 2^d\EEE}}\} =\{-\tfrac{1}{2},\tfrac{1}{2}\}^{ d\EEE}
\end{equation}
denotes the set of vertices of the unit cube of $\R^{ d\EEE}$ centered in $0$. We also set $\bar{SO}({ d\EEE}):=SO({ d\EEE})Z \subset \R^{{ d\EEE} \times { 2^d\EEE}}$. We denote by $\R^{{ d\EEE}\times { 2^d\EEE}}_{\mathrm{s,t}}$ the subspace of $\R^{{ d\EEE}\times { 2^d\EEE}}$ spanned by infinitesimal translations and rotations, namely
\begin{equation}\label{infinitesimal_translation_rotation}
\R^{{ d\EEE}\times { 2^d\EEE}}_{\mathrm{s,t}}:=\{SZ+(v,\dots,v)\colon S\in \R^{{ d\EEE}\times { d\EEE}}_{\mathrm{skew}}, \ v\in \R^{ d\EEE}\}\,.
\end{equation} 
By $P:\R^{{ d\EEE}\times { 2^d\EEE}}\to\R^{{ d\EEE}\times { 2^d\EEE}}$ 
we define \EEE the orthogonal projection onto the space orthogonal to $\R^{{ d\EEE}\times { 2^d\EEE}}_{\mathrm{s,t}}$. \EEE

Let $\varepsilon >0$ represent the interatomic distance between neighboring atoms in the lattice reference configuration.   For $ A\in \mathfrak{M}(\Omega)\EEE$ we set $\mathbb{Z}_\varepsilon(A) := \varepsilon\mathbb{Z}^{ d\EEE} \cap A$.  For any $i\in \varepsilon\mathbb{Z}^{ d\EEE}$, we 
define \EEE $\hat{i}:=i+\varepsilon(\tfrac{1}{2}, \dots \EEE
,\tfrac{1}{2})$. Moreover, for every $i\in \mathbb{Z}_\varepsilon(\Omega)$ and any set of lattice points $ E \subset \mathbb{Z}_\varepsilon(\Omega)$ (which in the sequel will represent the void set at the microscopic level), we set  
\begin{equation}\label{def_set_of_neighbours}
\mathcal{I}_\varepsilon(i,E):=\big\{l \in \{1,\dots, { 2^d\EEE}\}\colon \,  \hat i+\varepsilon z_l\in\mathbb{Z}_\varepsilon(\Omega)\setminus E\big\}\,.
\end{equation}
We also define \EEE 
\begin{align}\label{def:Aeps}
Q_\varepsilon(E) := \bigcup_{i \in E} Q_\varepsilon(i)\,,
\end{align}
where, for every $i\in \R^{ d\EEE}$ and $\rho>0$, $Q_{\rho}(i):=i+\rho[-\tfrac{1}{2},\tfrac{1}{2})^{ d\EEE}$ is the half-open cube in $\mathbb{R}^{ d\EEE}$ of sidelength $\rho$ \EEE centered at $i$.  Given also $\lambda>0$ and a cube $Q:=Q_{\rho}(i)$, we set $Q_{\lambda\rho}:=Q_{\lambda\rho}(i)$. When $\rho=1$, we omit the dependence on the sidelength and simply denote $Q(i):=Q_1(i)$. 
For a set of finite perimeter $A\subset \R^d$ \cite{Ambrosio-Fusco-Pallara:2000}\EEE, we also define 
the family of cubes with respect to the shifted lattice $\varepsilon\big(\mathbb{Z}^{ d\EEE}+(\frac{1}{2}, \dots \EEE
\frac{1}{2})\big)$ by \EEE  
\begin{equation}\label{cubic_set_of_A}
\hat\Q_{\eps,A}  :=\big\{Q_\eps(\hat i)\colon i\in \eps\Z^d, \mathcal{L}^d( 
  Q_\eps(\hat i) \EEE \cap A) > 0\big\}\,.
\end{equation}
For simplicity, we   write $\hat\Q_\eps:= \hat\Q_{\eps,{\R^d}}$.
\EEE We also introduce  cubes of sidelength  $2\varepsilon$, 
namely \EEE
\begin{align}\label{def: Qeps}
\hat{\mathcal{Q}}_{\eps,2} \EEE:= \big\{Q_{2\varepsilon}
(\hat{i})\colon \,  i \EEE \in \varepsilon\mathbb{Z}^{ d\EEE} \big\}\,.
\end{align} 
\EEE The following fact will be used several times in the sequel: for every $\eta>0$ and $i\in \eta\mathbb{Z}^{ d\EEE}$, 
it holds that \EEE
\begin{align}\label{ineq:Qintersectbound}
\#\{Q \EEE_\eta(j)  \colon \, \EEE  j\in \eta\mathbb{Z}^{ d\EEE}, Q_{4\eta}(i) \cap Q \EEE_{4\eta}(j) \neq \emptyset \} \leq  7^{ d\EEE}\,.
\end{align} 
Finally, unless  stated otherwise,  $0<c \le 1 \le C$ \EEE will denote generic  dimensional \EEE constants, whose values will be allowed to vary from line to line.\EEE

\section{Setting and main results}\label{sec: 2}

\subsection{Admissible deformations in the discrete setting and discrete gradients}\label{subsection 2.1}

Adopting the notation that was introduced in the previous subsection, we denote the void set at a microscopic level by  $E\subset \mathbb{Z}_\varepsilon(\Omega)$. We consider the associated collection of atoms $\mathbb{Z}_\varepsilon(\Omega)\setminus E$ and corresponding (discrete) deformations $y \colon \mathbb{Z}_\varepsilon(\Omega)\setminus E\to\mathbb{R}^{ d\EEE}$. We start by introducing \emph{admissible deformations} satisfying Dirichlet boundary values in a specific sense. \EEE

In order to define (thickened) Dirichlet boundary values for discrete mappings, we introduce some further terminology and notation. 
First of all, we make the following geometrical assumption on the Dirichlet boundary $\partial_D\Omega$: there exists a decomposition $\partial \Omega = \partial_D \Omega \cup \partial_N\Omega \cup N$ with 
\begin{align*}
\partial_D \Omega, \partial_N\Omega \text{ relatively open}, \ \  \  \mathcal{H}^{{ d-1\EEE}}(N) = 0\,, \ \ \  \partial_D\Omega \cap \partial_N \Omega = \emptyset\,, \ \ \  \partial (\partial_D \Omega) = \partial (\partial_N \Omega)\,,
\end{align*}
(the \EEE outermost boundary has to be understood in the relative sense with respect to $\partial \Omega$), and there exist $\bar{\sigma}>0$ small enough and $x_0 \in\R^{ d\EEE}$ such that for all $\sigma \in (0,\bar{\sigma})$ it holds that
\begin{align*}
O_{\sigma,x_0}(\partial_D   \Omega   ) \subset \Omega\,,
\end{align*}
where $O_{\sigma,x_0}(x) := x_0 + (1-\sigma)(x-x_0)$. (These assumptions are related to a suitable density result for displacements, see \cite[Lemma 5.7]{Crismale}.) In order to fix boundary values, we introduce an open set $U\supset \Omega$ such that $U$ and $U \setminus \overline{\Omega}$ are Lipschitz sets and $U \cap \partial \Omega = \partial_D \Omega$. 
\begin{definition}[Admissible deformations]\label{boundary values}
Given $y_0 \in W^{1,\infty}(\mathbb{R}^{ d\EEE};\mathbb{R}^{ d\EEE})$ and $E \subset \mathbb{Z}_\varepsilon(\Omega)$, a discrete deformation $y \colon \mathbb{Z}_\varepsilon(\R^{ d\EEE})\to \mathbb{R}^{ d\EEE}$  is said to be \emph{admissible}, written $y \in \mathcal{Y}_\varepsilon(y_0,\partial_D\Omega,E)$, if and only if
\begin{align}\label{def:Dirichlet_boundary_values}
\begin{split}
{\rm (i)} & \ \ y(i)=y_0(i) \ \ \mathrm{for\ \ every\ \ } i\in \mathbb{Z}_\varepsilon(U\setminus \Omega\EEE
)\, , \\ 
{\rm (ii)} & \ \ y(i)=i \ \ \mathrm{for\ \ every\ \ } i\in   \big( \mathbb{Z}_\varepsilon( \R^{ d\EEE}) \setminus  \mathbb{Z}_\varepsilon(U\EEE
) \big)  \cup E\,.  
\end{split}
\end{align}
\end{definition}
The definition in (ii) is for definiteness only. Given $y \in \mathcal{Y}_\varepsilon(y_0,\partial_D\Omega,E)$, recalling the notation in \eqref{Z_points}, we define the discrete gradient of $y$ at each  point   $i\in 
\mathbb{Z}_\varepsilon(\R^d)\EEE$ as 
\begin{equation}\label{discrete_gradient_def}
\overline{\nabla}y( i ):=\frac{1}{\varepsilon}\Big(y(\hat{i} + \eps z_1)-\bar y(\hat{i}),\dots,y(\hat{i} + \eps z_{ 2^d \EEE} )-\bar y(\hat{i})\Big), \ \mathrm{where} \ \ \bar y(\hat{i}):=\frac{1}{{ 2^d \EEE}}\sum_{l=1}^{{ 2^d \EEE}}y(\hat{i} + \eps z_l )\,.
\end{equation}
 Moreover, we set $\bar{e}(y)
(i) \EEE :=P(\overline{\nabla}y
(i) \EEE )$ for the orthogonal projection of $\overline{\nabla} y$ onto the orthogonal complement of $\R^{{ d \EEE}\times { 2^d \EEE}}_{\mathrm{s,t}}$, cf.\ \eqref{infinitesimal_translation_rotation}. \EEE In the following, when no confusion arises, we will also identify $\overline{\nabla}y$  and $\bar{e}(y)$ \EEE  with  their \EEE piecewise constant interpolation, being equal to $\overline{\nabla}y(i)$ and  $\bar{e}(y)(i)$ respectively, \EEE on each $Q_\varepsilon(\hat{i})$. We now  focus on the case $d=3$ and \EEE proceed by introducing the discrete energy. An example fitting to our assumptions is given at the end of the section in  Subsection \ref{example:EAM}. \EEE 

\subsection{The elastic cell energy}\label{Assumptions on the elastic energy}

Following \cite{Schmidt:2009}, the elastic energy of $y$ is defined in terms of a discrete elastic cell energy associated to each lattice point $i\in \mathbb{Z}_\varepsilon(\Omega)\setminus E$. In particular, given $E\subset \mathbb{Z}_{\varepsilon}(\Omega)$, $y \in \mathcal{Y}_\varepsilon(y_0,\partial_D\Omega,E)$ 
and $A\in \mathfrak{M}(\R^3)$, \EEE we define 
\begin{align}\label{def: discrete elastic energy}
F_\varepsilon^{\mathrm{el}}(y,E,A) := \delta_\varepsilon^{-2}\sum_{i \in\mathbb{Z}_\varepsilon(A)\setminus E} \varepsilon^{3} W_{\varepsilon, \mathrm{cell}}^{\mathrm{el}}(i,y,E)\,,
\end{align} 
where $W_{\varepsilon, \mathrm{cell}}^{\mathrm{el}}$ denotes the cell energy specified below in \eqref{definition_elastic_cell_energy} and $\delta_\varepsilon>0$ \EEE  describes the typical order of the elastic strain. The factor $\eps^3$ corresponds to a \emph{bulk scaling} of the elastic energy. We suppose that    $\delta_\varepsilon \to 0$ as $\eps \to 0$, i.e., in the discrete-to-continuum limit, we  simultaneously pass to the limit of  infinitesimal strains. The most natural choice for $\delta_\eps$ in discrete energies featuring bulk and surface terms is $\delta_\eps \sim \sqrt{\eps}$, see e.g.\ mass-spring models with interaction potentials of Lennard-Jones type as mentioned in the Introduction.

The main assumption is that the elastic cell energy can be split into a bulk and a surface part: for each  $i\in \eps\Z^3\EEE
\setminus E$ the cell energy is defined by \EEE
\begin{align}\label{definition_elastic_cell_energy}
W^{\mathrm{el}}_{\varepsilon,\mathrm{cell}}(i,y,E):=\begin{cases} W^{\mathrm{el}}_{\mathrm{bulk}}(\overline{\nabla} y(i)) &\text{if } \mathcal{I}_\varepsilon(i,E)=\{1,\dots,8\}\,,\\
W^{\mathrm{el}}_{\mathrm{surf}}(\mathcal{I}_\varepsilon(i,E),\overline \nabla y(i)\EEE) &\text{if } \mathcal{I}_\varepsilon(i,E)\neq\{1,\dots,8\}\,,
\end{cases} 
\end{align}
where $\mathcal{I}_\eps$ is defined in \eqref{def_set_of_neighbours}. \EEE 
If $A=\Omega$, we omit the dependence on $A$ and simply write $F_\varepsilon^{\mathrm{el}}(y,E)$. As in \cite[Assumption 2.1]{Schmidt:2009}, our general assumptions on the cell energy $W_{\varepsilon,\mathrm{cell}}^{\mathrm{el}}$ are the following\AAA.\EEE 
\begin{itemize}
\item[(i)] Each of the energies $W^{\mathrm{el}}_{\mathrm{bulk}}(\cdot)$ and $W^{\mathrm{el}}_{\mathrm{surf}}(I,\cdot)\colon\mathbb{R}^{3\times 8}\to[0,\infty]$, $I \subset \{1,\OOO \dots \EEE,8\}$, is translational and rotational invariant, i.e., for every $F\in \mathbb{R}^{3\times 8}, R\in SO(3)$, and $v\in \mathbb{R}^3$ we have
$$W^{\mathrm{el}}_{\mathrm{bulk}}(RF+ (v,\dots,v)) =W^{\mathrm{el}}_{\mathrm{bulk}}(F), \  \ \ \ \ W^{\mathrm{el}}_{\mathrm{surf}}(I,RF+ (v,\dots,v)) =W^{\mathrm{el}}_{\mathrm{surf}}(I,F)\,;$$
\item [(ii)] $W^{\mathrm{el}}_{\mathrm{bulk}}(F)$ for $F \in \R^{3 \times 8}$  is minimal ($= 0$) if and only if there exists $R\in SO(3)$ and $v\in \mathbb{R}^3$ such that $F_l=Rz_l+v \quad \forall\ l= 1, \dots,8$; 
\item[(iii)] $W^{\mathrm{el}}_{\mathrm{bulk}}$  is $C^{3}$-regular in a neighborhood of $\bar{SO}(3)$ and the Hessian $Q_{\mathrm{ bulk}}:=D^2W^{\mathrm{el}}_{\mathrm{bulk}}(Z)$ at the discrete identity  $Z$ is 
positive-definite \EEE on the orthogonal complement of the subspace $\R^{3 \times 8}_{\mathrm{s,t}}$, see \eqref{infinitesimal_translation_rotation};
\item[(iv)] $W^{\mathrm{el}}_{\mathrm{bulk}}$ grows at infinity at least quadratically on the orthogonal complement of the subspace spanned by infinitesimal translations, i.e.,
$$ \liminf_{\underset{F\in K}{|F|\to\infty}} \frac{W^{\mathrm{el}}_{\mathrm{bulk}}(F)}{|F|^2}>0\,,$$
where $K:=\{F\in \mathbb{R}^{3\times 8}:F_1+...+F_{8}= 0\}$; 
\item[(v)]  $W^{\mathrm{el}}_{\mathrm{surf}}$ depends only on the positions of atoms related to $\mathcal{I}_{\varepsilon}(i,E)$ in the reference configuration, i.e., $W^{\mathrm{el}}_{\mathrm{surf}}(\mathcal{I}_\varepsilon(i,E), F)$,  for $F\in \R^{3 \times 8}$, depends on the second variable only through the columns $F_l$, $l\in\mathcal{I}_{\varepsilon}(i,E)$;
\item[(vi)] The discrete identity $Z$ \EEE is an equilibrium configuration also for $W^{\mathrm{el}}_{\mathrm{surf}}$. In particular, we assume that there exists a constant $C>0$ such that for every $I\subset\{1,\dots,8\}$ it holds that
\begin{equation*}
W^{\mathrm{el}}_{\mathrm{surf}}(I,\cdot)\leq CW^{\mathrm{el}}_{\mathrm{bulk}}(\cdot)
\end{equation*}
in a neighborhood of $\bar{SO}(3)$. 
\end{itemize}
Note that the assumptions (i) and (iii) imply that $Q_{\mathrm{bulk}}(v,\OOO \dots \EEE,v) = 0$   and $Q_{\mathrm{bulk}}(Az_1,\dots\EEE,Az_{8}) = 0$ for all $v\in \mathbb{R}^3$ and $A\in \mathbb{R}^{3\times 3}_{\rm skew}$. Assumption (v) ensures that the atomic positions at lattice points in $( \mathbb{Z}_\varepsilon( \R^3 ) \setminus  \mathbb{Z}_\varepsilon(U\EEE) )  \cup E$  do not affect the elastic energy, i.e., \eqref{def:Dirichlet_boundary_values}(ii) is indeed for definiteness only. The last condition  in (vi) was referred to as  \emph{compatibility condition} between the surface and the bulk elastic energy density in \cite[Definition 2.2]{Schmidt:2009}.

\subsection{\EEE The surface energy}\label{subsec:surface_energy}
  Let $V\subset \mathbb{Z}^3  \setminus \{0\}$ be a finite set of vectors describing the neighbors of a point on the lattice.  We assume that 
\begin{itemize}
\item[(i)] $\|\xi\|_\infty:=\max_{l=1,2,3}|\xi_l|\EEE \leq 1$ for $\xi \in V$\,;
\item[(ii)] $ \{\pm e_1, \pm e_2,\pm e_3\}\subset V$\,;
\item[(iii)] $\xi \in V$ if and only if $-\xi \in V$.
\end{itemize}
An example of such a set of neighbors is $V=\big\{\sum_{l=1}^3\mu_le_l\colon \mu_l\in \{-1,0,1\}\big\}$. This would correspond to a model taking into account nearest, next-nearest, and next-next-nearest neighbor interactions.

For any $E \subset \mathbb{Z}_\varepsilon(\Omega)\EEE$ and $A \in \mathfrak{M}(\R^3)$, we consider an \textit{anisotropic discrete perimeter energy}, consisting of pair-interactions among neighboring atoms subordinate to the set $V$. More precisely, we define 
\begin{align}\label{def: discrete perimeter}
F_\varepsilon^{\mathrm{per}}(E,A):= \sum_{i \in \mathbb{Z}_\varepsilon(A)\cap E}\ \underset{i+\varepsilon\xi \in \mathbb{Z}_\varepsilon(U)\EEE}{\sum_{\xi \in V}}\varepsilon^{2} c_\xi \big(1-\chi_{E}(i+\varepsilon\xi)\big)\,,
\end{align}
where $U$ is the auxiliary set introduced in Definition \hyperref[boundary values]{2.1} in order to incorporate Dirichlet boundary conditions for the discrete deformations. Here, $c_\xi >0$ and  $c_\xi=c_{-\xi}$ for all $\xi \in V$. 
Since $E\subset \Z_\eps(\Omega)$, it is clear by the definition that $F_\eps^{\mathrm{per}}(E,U)=F_\eps^{\mathrm{per}}(E,\Omega)$. Thus, when \EEE $A=U\EEE$, we omit the dependence on $U\EEE$ in the  second variable.  Note that positive contributions correspond to the situation that an atom at position $i + \eps \xi \in\Z_\eps(U) \setminus E$ in the reference configuration misses a neighbor $i \in E$ lying in the void set $E$. \EEE The factor $\eps^2$ corresponds to the \emph{surface scaling}.

We complement the discrete perimeter energy with an additional \textit{discrete curvature energy}, associated to each pair $(E,A)$. This should be thought of as a discrete version of the continuum curvature energy considered in \cite[Equation (2.1)]{KFZ:2021}. It will allow us to implement the compactness result   \cite[Proposition 3.1]{KFZ:2021}, obtained through the rigidity result in Theorem 2.1 therein. The  discrete curvature energy will be considered at a mesoscale $\eta_{\varepsilon}>0$, satisfying
\begin{align}\label{eq: ratio}
\eta_\varepsilon / \varepsilon \in \mathbb{N}, \quad \eta_\eps \to 0,  \quad \text{and} \quad \eta_{\eps} / \eps \to +\infty \quad \text{as $\varepsilon\to 0$}\,.
\end{align}
More details why $\eta_\eps$ indeed needs to be chosen as a mesoscale are given below \eqref{eq:rate gammadelta}. \EEE 
\EEE Before introducing the specific form of the energy, we will first discuss which configurations of atoms will be considered as \emph{flat}, in the sense that they will not pay any curvature energy.
\begin{definition}[Local flatness]\label{locally_flat_configurations}
Let $i \in \varepsilon\mathbb{Z}^3$ and $\eta\in [\eta_\varepsilon,+\infty)$\EEE. We say that $E\subset\mathbb{Z}_\varepsilon(\Omega)$ is \textit{locally flat}  in $Q_{\eta
}(i)$ (with respect to $U\EEE$) if and only if for  all $Q := Q_{2\varepsilon}(\hat{k}) \in \hat{\mathcal{Q}}_{\eps,2}$, $k\in \varepsilon\mathbb{Z}^3$, such that $Q_{\eta}(i) \cap Q \neq \emptyset$, one of the following is true (see Figure \hyperref[flat configurations]{1}):
\begin{itemize}
\item[(i)] (Complete void) $E\EEE\cap Q= \mathbb{Z}_\varepsilon(Q) \cap U\EEE$\,;
\item[(ii)] (No void) $E\cap Q= \emptyset$\,;
\item[(iii)] (Half-cube void) $E\cap Q = \{k, k+\varepsilon e_1, k+\varepsilon e_2,k+\varepsilon (e_1+e_2)\} \cap U\EEE$, up to a  rotation of a multiple of $\pi/2$ with respect to an axis passing through $\hat k$ in the direction of one of the coordinate vectors.
\end{itemize}
\end{definition}

\EEE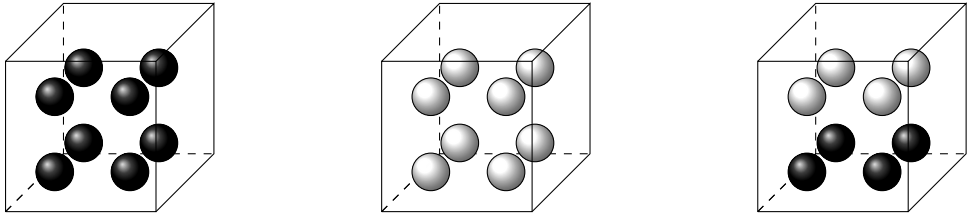
\begin{figure}[H]\label{flat configurations}
\begin{tikzpicture}

\begin{scope}[shift={(-10,0,0)}]
\draw[dashed](0,0,0)--++(0,0,-2)--++(0,2,0);
\draw[dashed](0,0,0)--++(0,0,-2)--++(2,0,0);
\shadedraw [ball color=black] (.75,.625,.25) circle (0.25);
\shadedraw [ball color=black] (1.75,.625,.25) circle (0.25);
\shadedraw [ball color=black] (.75,.625,-.75) circle (0.25);
\shadedraw [ball color=black] (1.75,.625,-.75) circle (0.25);

\shadedraw [ball color=black] (.75,1.625,.25) circle (0.25);
\shadedraw [ball color=black] (1.75,1.625,.25) circle (0.25);
\shadedraw [ball color=black] (.75,1.625,-.75) circle (0.25);
\shadedraw [ball color=black] (1.75,1.625,-.75) circle (0.25);

\draw(0,0,0)--++(2,0,0)--++(0,2,0)--++(-2,0,0)--++(0,-2,0);
\draw(2,0,0)--++(0,0,-2)--++(0,2,0)--++(0,0,2);
\draw(0,2,0)--++(0,0,-2)--++(2,0,0);

\end{scope}

\begin{scope}[shift={(-5,0,0)}]
\draw[dashed](0,0,0)--++(0,0,-2)--++(0,2,0);
\draw[dashed](0,0,0)--++(0,0,-2)--++(2,0,0);
\shadedraw [ball color=white] (.75,.625,.25) circle (0.25);
\shadedraw [ball color=white] (1.75,.625,.25) circle (0.25);
\shadedraw [ball color=white] (.75,.625,-.75) circle (0.25);
\shadedraw [ball color=white] (1.75,.625,-.75) circle (0.25);

\shadedraw [ball color=white] (.75,1.625,.25) circle (0.25);
\shadedraw [ball color=white] (1.75,1.625,.25) circle (0.25);
\shadedraw [ball color=white] (.75,1.625,-.75) circle (0.25);
\shadedraw [ball color=white] (1.75,1.625,-.75) circle (0.25);

\draw(0,0,0)--++(2,0,0)--++(0,2,0)--++(-2,0,0)--++(0,-2,0);
\draw(2,0,0)--++(0,0,-2)--++(0,2,0)--++(0,0,2);
\draw(0,2,0)--++(0,0,-2)--++(2,0,0);

\end{scope}

\draw[dashed](0,0,0)--++(0,0,-2)--++(0,2,0);
\draw[dashed](0,0,0)--++(0,0,-2)--++(2,0,0);
\shadedraw [ball color=black] (.75,.625,.25) circle (0.25);
\shadedraw [ball color=black] (1.75,.625,.25) circle (0.25);
\shadedraw [ball color=black] (.75,.625,-.75) circle (0.25);
\shadedraw [ball color=black] (1.75,.625,-.75) circle (0.25);

\shadedraw [ball color=white] (.75,1.625,.25) circle (0.25);
\shadedraw [ball color=white] (1.75,1.625,.25) circle (0.25);
\shadedraw [ball color=white] (.75,1.625,-.75) circle (0.25);
\shadedraw [ball color=white] (1.75,1.625,-.75) circle (0.25);
\draw(0,0,0)--++(2,0,0)--++(0,2,0)--++(-2,0,0)--++(0,-2,0);
\draw(2,0,0)--++(0,0,-2)--++(0,2,0)--++(0,0,2);
\draw(0,2,0)--++(0,0,-2)--++(2,0,0);    
  
\end{tikzpicture}
\caption{The (up to rotation) three possibilities of locally flat configurations in the cubes $ Q\in \hat{\mathcal{Q}}_{\eps,2}$. Here, the black spheres denote the lattice points in $E$ and the white spheres denote the lattice points in the complement of $E$.}
\label{fig:flat configuration}
\end{figure}

We point out that locally flat sets in a cube $Q_{\eta}(i)$, $\eta\in [\eta_\varepsilon,+\infty)$, $i\in \varepsilon\mathbb{Z}^3$, are locally \textit{simple coordinate laminates}. We refer to Figure \ref{laminate_picture} for an illustration, Definition \ref{def: lam} below for an exact definition, and Lemma \ref{lemma : laminate} for a proof of this fact. 
Due to the fact that $\eta/\eps \ge \eta_\eps/\eps \to +\infty$ as  $\varepsilon\to 0$, \EEE the number of these laminates inside the cube $Q_{\eta}(i)$ can possibly tend to infinity, and hence there can be an arbitrarily large number of possibilities for the geometry of the locally flat set $E$ inside $Q_{\eta}(i)$. In view of this fact, we will introduce sets that have a bounded number of possibilities for their geometry inside $Q_{\eta}(i)$, independently of $\varepsilon$. \EEE

\begin{figure}[H]
\begin{tikzpicture}
\draw[dashed](0,0,0)--++(0,0,-4)--++(0,4,0);
\draw[dashed](0,0,0)--++(0,0,-4)--++(4,0,0);
	
\draw(0,0,0)--++(4,0,0)--++(0,4,0)--++(-4,0,0)--++(0,-4,0);
\draw(4,0,0)--++(0,0,-4)--++(0,4,0)--++(0,0,4);
\draw(0,4,0)--++(0,0,-4)--++(4,0,0);
	
\draw(0,1,0)--++(4,0,0)--++(0,0,-4);
\draw[dashed](0,1,0)++(4,0,0)++(0,0,-4)--++(-4,0,0)--++(0,0,4);
	
\draw(0,.6,0)--++(4,0,0)--++(0,0,-4);
\draw[dashed](0,.6,0)++(4,0,0)++(0,0,-4)--++(-4,0,0)--++(0,0,4);

\draw[fill=gray,opacity=.3](0,0,0)++(0,1,0)++(4,0,0)--++(0,-.4,0)--++(0,0,-4)--++(0,.4,0);
	
\draw[fill=gray,opacity=.3](0,0,0)++(0,1,0)--++(4,0,0)--++(0,-.4,0)--++(-4,0,0)--++(0,.4,0);
	
\draw[fill=gray,opacity=.3](0,0,0)++(0,1,0)--++(4,0,0)--++(0,0,-4)--++(-4,0,0)--++(0,0,4);

	
\draw(0,2.2,0)--++(4,0,0)--++(0,0,-4);
\draw[dashed](0,2.2,0)++(4,0,0)++(0,0,-4)--++(-4,0,0)--++(0,0,4);
	
\draw(0,1.6,0)--++(4,0,0)--++(0,0,-4);
\draw[dashed](0,.6,0)++(4,0,0)++(0,0,-4)--++(-4,0,0)--++(0,0,4);

\draw[fill=gray,opacity=.3](0,0,0)++(0,2.2,0)++(4,0,0)--++(0,-.6,0)--++(0,0,-4)--++(0,.6,0);
	
\draw[fill=gray,opacity=.3](0,0,0)++(0,2.2,0)--++(4,0,0)--++(0,-.6,0)--++(-4,0,0)-++(0,.6,0);
	
\draw[fill=gray,opacity=.3](0,0,0)++(0,2.2,0)--++(4,0,0)--++(0,0,-4)--++(-4,0,0)--+(0,0,4);
	
		
\draw(0,4,0)--++(4,0,0)--++(0,0,-4);
\draw[dashed](0,4,0)++(4,0,0)++(0,0,-4)--++(-4,0,0)--++(0,0,4);
	
\draw(0,2.8,0)--++(4,0,0)--++(0,0,-4);
\draw[dashed](0,2.8,0)++(4,0,0)++(0,0,-4)--++(-4,0,0)--++(0,0,4);
		
\draw[fill=gray,opacity=.3](0,0,0)++(0,4,0)++(4,0,0)--++(0,-1.2,0)--++(0,0,-4)--++(0,1.2,0);
	
\draw[fill=gray,opacity=.3](0,0,0)++(0,4,0)--++(4,0,0)--++(0,-1.2,0)--++(-4,0,0)--++(0,1.2,0);
	
\draw[fill=gray,opacity=.3](0,0,0)++(0,4,0)--++(4,0,0)--++(0,0,-4)--++(-4,0,0)--++(0,0,4);

\end{tikzpicture}
\caption{A laminate-configuration depicted schematically. The lattice points
in $\varepsilon\mathbb{Z}^3$ that lie in the gray set should be thought of as the set $E$.}
\label{laminate_picture}
\end{figure}
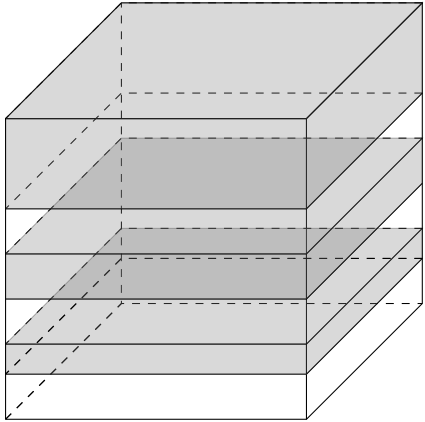

\begin{definition}[Cubic sets]\label{corners}  Let $i\in \varepsilon\mathbb{Z}^3$ and $\eta\in [\eta_\varepsilon,+\infty)$. We say that $ E \subset \mathbb{Z}_\varepsilon(\Omega)$ is \textit{cubic} in $Q_{\eta
}(i)$ (with respect to $U$) if and only if there exist $i_0 \in \mathbb{Z}_{\varepsilon}(Q_{\eta_\varepsilon
})$ and $\{k_1,\ldots,k_n\} \subset i_0 + \eta_\varepsilon \mathbb{Z}^3$ 
(possibly empty) such that
\begin{align*}
E\cap Q_{\eta
}(i) = \mathbb{Z}_\varepsilon\left(\bigcup_{m=1}^n Q_{\eta_\varepsilon
}(k_m)\right)\cap Q_{\eta
}(i)\cap U\EEE\,.
\end{align*}
\end{definition}

For an illustration of cubic sets we refer to Figure~\ref{fig:corner}. Note that cubic sets can be locally flat but might also not. 

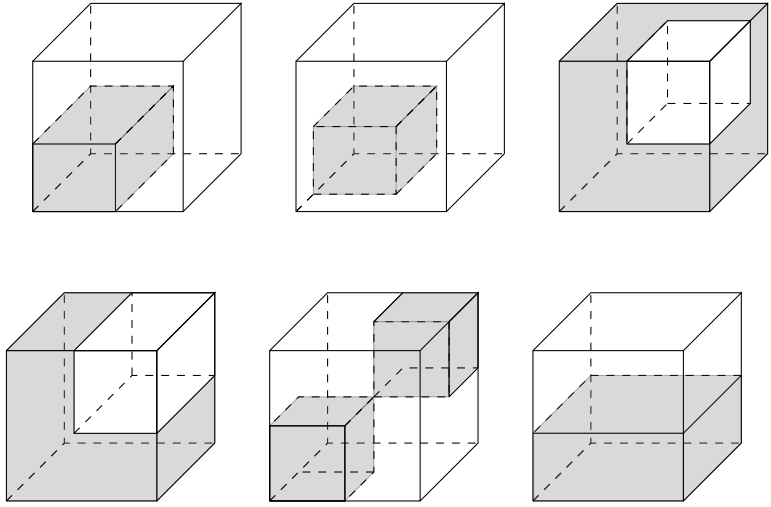
\begin{figure}[H]\label{locally flat configurations}
\EEE
\begin{tikzpicture}
		
\draw[dashed](0,0,0)--++(0,0,-2)--++(0,2,0);
\draw[dashed](0,0,0)--++(0,0,-2)--++(2,0,0);

\draw[fill=gray,opacity=.3](0,0,0)--++(1.1,0,0)--++(0,.9,0)--++(-1.1,0,0)--++(0,-.9,0);
\draw(0,0,0)--++(1.1,0,0)--++(0,.9,0)--++(-1.1,0,0)--++(0,-.9,0);
		
		
\draw[fill=gray,opacity=.3](0,0,0)++(1.1,0,0)++(0,.9,0)--++(0,0,-2)--++(0,-.9,0)--++(0,0,2);

\draw[fill=gray,opacity=.3](0,0,0)++(1.1,0,0)++(0,.9,0)--++(0,0,-2)--++(-1.1,0,0)--++(0,0,2);
		
\draw[dashed](0,0,0)++(1.1,0,0)++(0,.9,0)--++(0,0,-2)--++(0,-.9,0)--++(0,0,2);
\draw[dashed](0,0,0)++(1.1,0,0)++(0,.9,0)--++(0,0,-2)--++(-1.1,0,0)--++(0,0,2);

\draw(0,0,0)--++(2,0,0)--++(0,2,0)--++(-2,0,0)--++(0,-2,0);
\draw(2,0,0)--++(0,0,-2)--++(0,2,0)--++(0,0,2);
\draw(0,2,0)--++(0,0,-2)--++(2,0,0);
		
\begin{scope}[shift={(3.5,0,0)}]
\draw[dashed](0,0,0)--++(0,0,-2)--++(0,2,0);
\draw[dashed](0,0,0)--++(0,0,-2)--++(2,0,0);		
		
\draw[fill=gray,opacity=.3](0,0,-.6)--++(1.1,0,0)--++(0,.9,0)--++(-1.1,0,0)--++(0,-.9,0);
		
\draw[fill=gray,opacity=.3](0,0,-.6)++(1.1,0,0)++(0,.9,0)--++(0,0,-1.4)--++(0,-.9,0)--++(0,0,1.4);
		
\draw[fill=gray,opacity=.3](0,0,-.6)++(1.1,0,0)++(0,.9,0)--++(0,0,-1.4)--++(-1.1,0,0)--++(0,0,1.4);

\draw[dashed](0,0,-.6)--++(1.1,0,0)--++(0,.9,0)--++(-1.1,0,0)--++(0,-.9,0);
		
\draw[dashed](0,0,-.6)++(1.1,0,0)++(0,.9,0)--++(0,0,-1.4)--++(0,-.9,0)--++(0,0,1.4);
		
\draw[dashed](0,0,-.6)++(1.1,0,0)++(0,.9,0)--++(0,0,-1.4)--++(-1.1,0,0)--++(0,0,1.4);
		
\draw(0,0,0)--++(2,0,0)--++(0,2,0)--++(-2,0,0)--++(0,-2,0);
\draw(2,0,0)--++(0,0,-2)--++(0,2,0)--++(0,0,2);
\draw(0,2,0)--++(0,0,-2)--++(2,0,0);
\end{scope}
				
\begin{scope}[shift={(7,0,0)}]
\draw[dashed](0,0,0)--++(0,0,-2)--++(0,2,0);
\draw[dashed](0,0,0)--++(0,0,-2)--++(2,0,0);
				
\draw[fill=gray,opacity=.3](0,0,0)--++(2,0,0)--++(0,.9,0)--++(-1.1,0,0)--++(0,1.1,0)--++(-.9,0,0);
	
\draw[fill=gray,opacity=.3](0,0,0)++(2,0,0)++(0,.9,0)++(-1.1,0,0)--++(0,0,-1.4)--++(0,1.1,0)--++(0,0,1.4);
\draw[fill=gray,opacity=.3](0,0,0)++(2,0,0)++(0,.9,0)++(-1.1,0,0)++(0,0,-1.4)--++(1.1,0,0)--++(0,1.1,0)--++(-1.1,0,0);
		
\draw[fill=gray,opacity=.3](0,0,0)++(2,0,0)++(0,.9,0)++(-1.1,0,0)++(0,0,-1.4)--++(1.1,0,0)--++(0,0,1.4)--++(-1.1,0,0);
		
\draw[fill=gray,opacity=.3](0,0,0)++(2,0,0)--++(0,.9,0)--++(0,0,-1.4)--++(0,1.1,0)--++(0,0,-.6)--++(0,-2,0)--++(0,0,2);
		
\draw[fill=gray,opacity=.3](0,2,0)--++(.9,0,0)--++(0,0,-1.4)--++(1.1,0,0)--++(0,0,-.6)--++(-2,0,0)--++(0,0,2);
		
\draw[fill=white](0,0,0)++(2,0,0)++(0,.9,0)--++(-1.1,0,0)--++(0,1.1,0)--++(1.1,0,0)--++(0,-1.1,0);
		
\draw[fill=white](0,0,0)++(2,0,0)++(0,.9,0)--++(0,0,-1.4)--++(0,1.1,0)--++(0,0,1.4)--++(0,-1.1,0);
				
\draw[fill=white](0,0,0)++(2,0,0)++(0,2,0)--++(0,0,-1.4)--++(-1.1,0,0)--++(0,0,1.4)--++(1.1,0,0);
		
\draw[dashed](0,0,0)++(2,0,0)++(0,.9,0)++(-1.1,0,0)--++(0,0,-1.4)--++(0,1.1,0)--++(0,0,1.4)--++(0,-1.1,0);
		
\draw[dashed](0,0,0)++(2,0,0)++(0,.9,0)++(-1.1,0,0)++(0,0,-1.4)--++(1.1,0,0);
		
%
%
%
%
		
\draw(0,0,0)--++(2,0,0)--++(0,2,0)--++(-2,0,0)--++(0,-2,0);
\draw(2,0,0)--++(0,0,-2)--++(0,2,0)--++(0,0,2);
\draw(0,2,0)--++(0,0,-2)--++(2,0,0);
\end{scope}

\begin{scope}[shift={(3.5,0,10)}]
\draw[dashed](0,0,0)--++(0,0,-2)--++(0,2,0);
\draw[dashed](0,0,0)--++(0,0,-2)--++(2,0,0);
		
\draw[fill=gray,opacity=.3](0,0,0)--++(2,0,0)--++(0,.9,0)--++(-1.1,0,0)--++(0,1.1,0)--++(-.9,0,0);
		
\draw[fill=gray,opacity=.3](0,0,0)++(2,0,0)++(0,.9,0)++(-1.1,0,0)--++(0,0,-2)--++(0,1.1,0)--++(0,0,2);
		
\draw[fill=gray,opacity=.3](0,0,0)++(2,0,0)++(0,.9,0)++(-1.1,0,0)--++(0,0,-2)--++(1.1,0,0)--++(0,0,2);
		
\draw[fill=gray,opacity=.3](0,0,0)++(2,0,0)--++(0,.9,0)--++(0,0,-2)--++(0,-.9,0)--++(0,0,2);
				
\draw[fill=gray,opacity=.3](0,0,0)++(0,2,0)--++(.9,0,0)--++(0,0,-2)--++(-.9,0,0)--++(0,0,2);
		
\draw[fill=white](0,0,0)++(2,0,0)++(0,.9,0)--++(-1.1,0,0)--++(0,1.1,0)--++(1.1,0,0)--++(0,-1.1,0);
		
\draw[fill=white](0,0,0)++(2,0,0)++(0,.9,0)--++(0,1.1,0)--++(0,0,-2)--++(0,-1.1,0)--++(0,0,2);
		
\draw[fill=white](0,0,0)++(2,0,0)++(0,2,0)--++(-1.1,0,0)--++(0,0,-2)--++(1.1,0,0)--++(0,0,2);
		
\draw[dashed](0,0,0)++(2,0,0)++(0,.9,0)++(-1.1,0,0)--++(0,0,-2)--++(1.1,0,0);
		
\draw[dashed](0,0,0)++(2,0,0)++(0,.9,0)++(-1.1,0,0)++(0,0,-2)--++(0,1.1,0);
		
\draw(0,0,0)--++(2,0,0)--++(0,2,0)--++(-2,0,0)--++(0,-2,0);
\draw(2,0,0)--++(0,0,-2)--++(0,2,0)--++(0,0,2);
\draw(0,2,0)--++(0,0,-2)--++(2,0,0);
\end{scope}

\begin{scope}[shift={(7,0,10)}]
\draw[dashed](0,0,0)--++(0,0,-2)--++(0,2,0);
\draw[dashed](0,0,0)--++(0,0,-2)--++(2,0,0);
		
\draw[fill=gray,opacity=.3](0,0,0)--++(1,0,0)--++(0,1,0)--++(-1,0,0)--++(0,-1,0);
		
\draw(0,0,0)--++(1,0,0)--++(0,1,0)--++(-1,0,0)--++(0,-1,0);
				
\draw[fill=gray,opacity=.3](0,0,0)++(1,0,0)++(0,1,0)--++(0,0,-1)--++(0,-1,0)--+(0,0,1);
		
\draw[dashed](0,0,0)++(1,0,0)++(0,1,0)--++(0,0,-1)--++(0,-1,0)--++(0,0,1);
				
\draw[fill=gray,opacity=.3](0,0,0)++(1,0,0)++(0,1,0)--++(0,0,-1)--++(-1,0,0)--++(0,0,1);
		
\draw[dashed](0,0,0)++(1,0,0)++(0,1,0)--++(0,0,-1)--++(-1,0,0)--++(0,0,1);		

\draw[dashed](0,0,0)++(1,0,0)++(0,0,-1)--++(-1,0,0);
		
\draw(0,0,0)--++(1,0,0)--++(0,1,0)--++(-1,0,0)--++(0,-1,0);
		
\draw[fill=gray,opacity=.3](0,0,0)++(1,0,0)++(0,1,0)++(0,0,-1)--++(1,0,0)--++(0,1,0)--++(-1,0,0)--++(0,-1,0);
\draw[dashed](0,0,0)++(1,0,0)++(0,1,0)++(0,0,-1)--++(1,0,0)--++(0,1,0)--++(-1,0,0)--++(0,-1,0);
				
\draw[fill=gray,opacity=.3](0,0,0)++(2,0,0)++(0,1,0)++(0,0,-1)--++(0,1,0)--++(0,0,-1)--++(0,-1,0);
		
\draw(0,0,0)++(2,0,0)++(0,1,0)++(0,0,-1)--++(0,1,0)--++(0,0,-1)--++(0,-1,0)--++(0,0,1);
		
\draw(0,0,0)++(1,0,0)++(0,2,0)++(0,0,-1)--++(0,0,-1);
		
\draw(0,0,0)++(1,0,0)++(0,2,0)++(0,0,-1)--++(0,0,-1);
		
\draw[dashed](0,0,0)++(1,0,0)++(0,1,0)++(0,0,-1)--++(0,0,-1)--++(1,0,0);
		
\draw[fill=gray,opacity=.3](0,0,0)++(1,0,0)++(0,2,0)++(0,0,-1)--++(0,0,-1)--++(1,0,0)--++(0,0,1);
				
\draw(0,0,0)--++(2,0,0)--++(0,2,0)--++(-2,0,0)--++(0,-2,0);
\draw(2,0,0)--++(0,0,-2)--++(0,2,0)--++(0,0,2);
\draw(0,2,0)--++(0,0,-2)--++(2,0,0);
\end{scope}

\begin{scope}[shift={(10.5,0,10)}]
\draw[dashed](0,0,0)--++(0,0,-2)--++(0,2,0);
\draw[dashed](0,0,0)--++(0,0,-2)--++(2,0,0);
			
\draw(0,.9,0)--++(2,0,0)--++(0,0,-2);
		
\draw[dashed](0,.9,0)++(2,0,0)++(0,0,-2)--++(-2,0,0)--++(0,0,2);
		
\draw[fill=gray,opacity=.3](0,.9,0)--++(2,0,0)--++(0,0,-2)--++(-2,0,0)--++(0,0,2);
		
\draw[fill=gray,opacity=.3](0,0,0)--++(0,.9,0)--++(2,0,0)--++(0,-.9,0)--++(-2,0,0);
		
\draw[fill=gray,opacity=.3](0,0,0)++(0,.9,0)++(2,0,0)--++(0,-.9,0)--++(0,0,-2)--++(0,.9,0);
		
\draw(0,0,0)--++(2,0,0)--++(0,2,0)--++(-2,0,0)--++(0,-2,0);
\draw(2,0,0)--++(0,0,-2)--++(0,2,0)--++(0,0,2);
\draw(0,2,0)--++(0,0,-2)--++(2,0,0);
\end{scope}
		
\end{tikzpicture}
\caption{Some of the possible configurations for cubic sets depicted schematically. The lattice points in $\varepsilon\mathbb{Z}^3$ that lie in the gray set should be thought of as the set $E$.  Note that for the various pictures $i_0 \in \mathbb{Z}_\varepsilon(Q_{\eta_\varepsilon})$ may be chosen differently.}
\label{fig:corner}
\end{figure}

\begin{definition}[Flatness]\label{flatness-new}  Let $i\in \varepsilon\mathbb{Z}^3$ and $\eta\in [\eta_\varepsilon,+\infty)$. We say that $ E \subset \mathbb{Z}_\varepsilon(\Omega)$ is \textit{flat} in $Q_{\eta
}(i)$ (with respect to $U\EEE$) if and only if it is cubic and locally flat in $Q_{\eta}(i)$.  
\end{definition}

It is elementary to check that 
flat \EEE sets in $Q_{\eta
}(i)$ are a halfspace intersected with  $\mathbb{Z}_\varepsilon(Q_{\eta}(i))$, see the last example in Figure \ref{fig:corner}, i.e., in such a case, there exists $x_0\in Q_{\eta}(i)$ and $\nu\in\{\pm e_1,\pm e_2,\pm e_3\}$, such that 
\begin{equation*}
E\cap Q_{\eta}(i)=\mathbb{Z}_\varepsilon(Q_{\eta}(i))\cap\{y\in \R^3\colon (y-x_0) \cdot \nu  \geq 0\}\cap U\EEE.
\end{equation*}
We now introduce a discrete  version of the curvature regularization energy that we used in \cite{KFZ:2021}, expressed through a curvature cell energy at the mesoscale $\eta_\varepsilon$. In particular, let $q\in [2,+\infty)$ be fixed from now on. Given any $ E \subset \mathbb{Z}_\varepsilon(\Omega)$ and $A\in \mathfrak{M}(\R^3)$, we define
\begin{align}\label{def: curvature energy}
F_\varepsilon^{\mathrm{curv}}(E, A) := \sum_{i \in \mathbb{Z}_{\varepsilon}(A
)} \varepsilon^3 \, W_{\varepsilon,\mathrm{cell}}^{\mathrm{curv}}(i,E)\,.
\end{align}
We  again omit the dependence on the set in the second variable 
for \EEE $A=U\EEE$. Our structural assumptions on the curvature cell energy are the following: 
\begin{itemize}
\item[(i)] (Flatness) $W_{\varepsilon,\mathrm{cell}}^\mathrm{curv}(i, E) = 0$ \quad if $E$ is flat in $Q_{\eta_{ \varepsilon}}(i)$\,;
\item[(ii)] (Lower bound for local non-flatness) There exist $c>0$ and $\gamma_\varepsilon>0$, $\gamma_{\varepsilon}\to 0$ as $\varepsilon\to 0$, such that $W_{\varepsilon,\mathrm{cell}}^\mathrm{curv}(i, E) \geq c \gamma_{\varepsilon}\eta_{\varepsilon}^{-1-q}$ if $E$ is not locally flat in $Q_{\eta_{\varepsilon}}(i)$\,;
\item[(iii)] (Upper bound for cubic sets) There exists $C>0$ such that $W_{\varepsilon,\mathrm{cell}}^\mathrm{curv}(i,E\EEE) \leq C\gamma_{\varepsilon} \eta_{\varepsilon}^{-1-q}$ \ if $ E$ is cubic in $Q_{\eta_{\varepsilon}}(i)$\,\OOO.\EEE
\end{itemize}
Let us comment on the scaling of the energy in (ii) and (iii): roughly speaking, local non-flatness corresponds to an energy per atom of order $\gamma_{\varepsilon}\eta_{\varepsilon}^{-1-q}$, where $\gamma_\eps$ represents a curvature regularization parameter. In this case, in a cube $Q_{\eta_\eps}(i)$, $i\in\varepsilon\mathbb{Z}^3$, \EEE the overall energy  $F_\varepsilon^{\mathrm{curv}}(E, Q_{\eta_\eps}(i))$ is typically of the order $\gamma_{\varepsilon}\eta_{\varepsilon}^{2-q}$, since $\# \mathbb{Z}_{\varepsilon}(Q_{\eta_\eps}(i))=\eta_\eps^3/\eps^3$. Up to the prefactor $\gamma_\eps$, this corresponds to  the integral of  $q$-th power of the second fundamental form of a 
round sphere in $\R^3$ \EEE with radius $\eta_\eps$. This effectively relates our choice of discrete curvature to classical continuum curvature notions. \EEE 

\begin{remark}\label{U_independence}
{\normalfont
Note that, since $E\subset Z_\eps(\Omega)$, the Definitions \ref{locally_flat_configurations}--\ref{flatness-new} as well as the curvature energy $F_\eps^{\rm{curv}}(E,U)$ are independent of the choice of the auxiliary set $U$ for $\eps>0$ sufficiently small (and thus $\eta_\eps$ small accordingly, see \eqref{eq: ratio}).}
\end{remark}
\EEE
\begin{remark}[Different assumptions]\label{rem: diffassu}
{\normalfont
For later purposes in the example of Subsection \ref{example:EAM}, let us comment that our results are still valid if  in (ii) of the above assumptions, the cubes $Q_{\eta_{\varepsilon}}(i)$ are replaced by smaller cubes $Q_{\eta}(i)$ for $\eta \in [\eta_\eps/2, \eta_\eps]$ with $\eta /\eps \in \N$. Our choice of the exact sidelength $\eta_{\varepsilon}$  is  only for expository reasons, in order to formulate  the above assumptions in a simpler way. 
}
\end{remark}

We also assume that the linearization parameter $\delta_\varepsilon$, the mesoscale $\eta_\varepsilon$,  and the curvature regularization parameter $\gamma_{\varepsilon}$  are related to each other via the rates 
\begin{align} \label{eq:rate gammadelta}
\gamma_{\varepsilon}\eta^{-q}_\varepsilon \to +\infty, \quad  \gamma_{\varepsilon}\eta_\varepsilon^{1-q} \to 0 \quad \text{and} \quad \gamma_{\varepsilon}\delta_{\varepsilon}^{-q/9}\to +\infty \EEE\quad \text{as} \quad \varepsilon\to 0.
\end{align}
The last condition is \cite[Equation (3.4)]{KFZ:2021}, and is required in order to apply the compactness result \cite[Proposition 3.1]{KFZ:2021}. \EEE The first two conditions, which relate the mesoscale on which the curvature energy is defined and the curvature regularization parameter, are necessary in order to suitably modify the void sets at $\varepsilon$-scale to geometrically and energetically more convenient sets at the $\eta_\varepsilon$-scale. More precisely, the first condition allows for sharp bounds on the perimeter and curvature energy, as well as the cardinality of the modified sets in terms of the energies and cardinality of the original ones, see Lemmata \ref{lemma : eta lattice} and \ref{lemma:eta scale smoothening}. The second condition is needed to ensure that the curvature energy for recovery sequences is negligible.

We remark that requiring $\delta_\eps^{1/9} \eta_\eps^{-1} \to + \infty$ and $\delta_\eps^{1/9} \eta_\eps^{-1 + \frac{1}{q}} \to 0$ as $\varepsilon\to 0$ are sufficient conditions to find $(\gamma_\eps)_{\varepsilon>0}$ such that \eqref{eq:rate gammadelta} holds true. In view of \eqref{eq: ratio}, this forces the choice of the sequence $(\delta_\eps)_{\eps>0}$ to satisfy $ \eps^{9} \delta_\eps^{-1} \to 0$ as $\varepsilon\to 0$. In particular, the most relevant scaling $\delta_\eps \sim \sqrt{\eps}$ is admissible, and in this case we have $ (\eps^{q/(q-1)})^{1/18} \ll \eta_\eps \ll \eps^{1/18}$, i.e., $\eta_\eps$ represents indeed a mesoscale. 

\subsection{Effective continuum $\Gamma$-limit of atomistic systems}\label{subsec: main result}

We suppose that the boundary data are imposed  on a relatively open subset  $\partial_D \Omega \subset \partial \Omega$ and are close to the identity. More precisely, as mentioned in Definition \ref{boundary values}, we impose the boundary data on $\Z_\eps(U\setminus\Omega)$, where $U\supset \Omega$ is an artificially introduced open Lipschitz set with $U\cap\partial\Omega=\partial_D\Omega$. \EEE To this end, let $u_0 \in W^{1,\infty}(\R^3;\R^3)$ and, given $\varepsilon>0$ and the strain parameter \EEE $\delta_\varepsilon>0$, we define $y_0^\varepsilon := {\rm id} + \delta_\varepsilon u_0$. Recalling \eqref{def: discrete elastic energy}, \eqref{def: discrete perimeter}, \eqref{def: curvature energy}, as well as Definition \hyperref[boundary values]{2.1}, we let $F_\varepsilon$ be the functional defined by
\begin{align*}
F_\varepsilon(y,E) :=\begin{cases} F_\varepsilon^\mathrm{el}(y,E) + F_\varepsilon^\mathrm{per}(E) +F_\varepsilon^\mathrm{curv}(E)\, \quad  \mathrm{if} \ E\subset \mathbb{Z}_{\varepsilon}(\Omega) \text{ and }  y\in \mathcal{Y}_\varepsilon(y_0^{\varepsilon},\partial_D\Omega,E)\,,\\
+\infty \quad \mathrm{otherwise}.
\end{cases}
\end{align*}
From now on, given $y\in \mathcal{Y}_\varepsilon(y_0^{\varepsilon},\partial_D\Omega,E)$, we define the corresponding   \emph{discrete rescaled displacement}  $u \in \mathcal{U}_\varepsilon(u_0,\partial_D\Omega,E) $ by $u := \delta_\varepsilon^{-1}(y - {\rm id})$, and for notational convenience \EEE we set
\begin{equation}\label{def: energy_u}
\mathcal{F}_\varepsilon(u,E) := F_\varepsilon({\rm id} + \delta_\varepsilon u, E).
\end{equation}
Our goal is to identify the effective continuum limit of the above discrete energies as $\eps \to 0$. It turns out that the limiting energy functional coincides with the relaxed,  linearized model for material voids in elastically stressed solids  studied in \cite{Crismale}. First of all, we introduce the continuum surface energy density $\varphi \colon \mathbb{R}^{3} \to [0,+\infty)$  by 
\begin{align}\label{def:varphi}
\varphi(\nu) := \sum_{k=1}^3 c_k |\nu_k|\,, \quad \text{where } c_k :=  \sum_{\xi \in V,\, \xi_k >0} c_\xi\,
\end{align}
for $\nu \in \R^3$, where $V$ was defined before \eqref{def: discrete perimeter}. 

Given $u \in GSBD^2(\Omega)$ (see \cite{DalMaso:13}  for the definition and details on this function space) and a set of finite perimeter $E\in\M(\Omega)$, we first define the \emph{boundary energy term} by
\begin{align}\label{eq: bdaypart}
\mathcal{F}^{\rm \partial_D}(u,E) :=   \int_{\partial^* E\cap\partial_D\Omega}  \varphi  (\nu_E) \,{\rm d}\mathcal{H}^{2}  +  \int\limits_{ \{ \mathrm{tr}(u)  \neq \mathrm{tr}(u_0)  \} \cap  (\partial_D \Omega \setminus \partial^* E) }  \hspace{-0.5cm} 2 \, \varphi(  \nu_\Omega  ) \, {\rm d}\mathcal{H}^{2}, 
\end{align}
which is non-trivial if the void goes up to the Dirichlet part of the boundary, or the  mapping   $u$ does not satisfy the imposed boundary conditions. Here, $\nu_E$ denotes the measure-theoretical outer unit normal to $\partial^{\AAA * \EEE} E$, \EEE $\nu_\Omega$ denotes the outer unit normal to $\partial \Omega$, and ${\rm tr}(u)$ indicates the trace of $u$ at $\partial \Omega$, which  is well-defined for functions in $GSBD^2(\Omega)$, see \cite[Section 5]{DalMaso:13}.\\ The limiting linearized elastic energy for pairs $(u,E)$ as above is defined by
\begin{align}\label{def:contel}
\mathcal{F}^{\mathrm{el}}(u,E) := \frac{1}{2}\int_{\Omega\setminus E} Q_\mathrm{\mathrm{bulk}}(e(u)Z)\,\mathrm{d}x\,, 
\end{align}
where $e(u) := \mathrm{sym}(\nabla u)$ is the approximate symmetric gradient of $u$ and $Q_\mathrm{bulk}:= D^2W^{\mathrm{el}}_{\mathrm{bulk}}(Z)$, see Assumption (iii) in Subsection \ref{Assumptions on the elastic energy}. The limiting surface energy is given by
\begin{align}\label{def:contsurf}
\mathcal{F}^\mathrm{surf}(u,E) := \int_{\partial^* E\cap \Omega} \varphi(\nu_E)\,\mathrm{d}\mathcal{H}^2+\int_{J_u \setminus \partial^* E} 2\varphi(\nu_u)\,\mathrm{d}\mathcal{H}^2\,,
\end{align}
where $\nu_{u}$ is the approximate unit normal vector to $J_u$. The total limiting energy is defined by
\begin{equation}\label{def:F}
\mathcal{F}_0\EEE(u,E) := \mathcal{F}\EEE^\mathrm{el}(u,E) +\mathcal{F}\EEE^\mathrm{surf}(u, E)+\mathcal{F}^{\mathrm{\partial_D}}(u,E)\,,
\end{equation}
if $\mathcal{H}^2(\partial^{*}E\cap \Omega)<+\infty,u = \chi_{\Omega \setminus E} u \in GSBD^2(\Omega)$, and $\mathcal{F}_0(u,E):=+\infty$ otherwise.

In order to formulate the main result of this paper, we proceed with the definition of convergence for sequences of discrete displacements and void subsets of lattice points. To that end, given $E\subset \mathbb{Z}_{\varepsilon}(\Omega)$ and a discrete displacement $u \in \mathcal{U}_\varepsilon(u_0,\partial_D\Omega,E)$,  we denote by $\bar u$ 
the piecewise constant interpolation of $u$, being equal to $\delta_\eps^{-1}(\bar y(\hat{i}) - {\rm id})$  on  $Q_\varepsilon(\hat{i})$ for each $i \in \eps\Z^3$, where $\bar y(\hat{i})$ is defined in \eqref{discrete_gradient_def}. 
For the next definition, we denote by $L^0(\Omega;\R^3)$ the space of measurable maps in $\Omega$ (equipped with the topology of convergence in measure) and we \EEE recall also the notation introduced in \eqref{def:Aeps}.

%
	
%

\begin{definition}(Discrete-to-continuum convergence)\label{def:convergence} We say that a sequence 
$(u_\varepsilon,E_\varepsilon)_{\varepsilon>0}$ with $E_\varepsilon \subset\mathbb{Z}_\varepsilon(\Omega)$ and $u_\varepsilon \in \mathcal{U}_\varepsilon(u_0,\partial_D\Omega,E_\eps)$ converges to a pair $(u,E)\in L^0(\Omega;\mathbb{R}^{3}) \times \mathfrak{M}(\Omega)$ 
\EEE as $\varepsilon\to 0$ in the d-sense, and we write $(u_\varepsilon,E_\varepsilon) \overset{\d}{\to} (u,E)$, if there exists a set of finite perimeter $\omega_u\in \mathfrak{M}(\Omega)$ such that  $\chi_{Q_\varepsilon(E_\varepsilon)\cap \Omega} \to \chi_E \text{ in }L^1(\Omega)$, $\bar u_\varepsilon \to u$ in measure on $\Omega \setminus \omega_u$, and $u\equiv 0$ on $E \cup \omega_u$.
\end{definition}

Note that this convergence of $(\chi_{Q_\varepsilon(E_\varepsilon)\cap \Omega},\bar u_\varepsilon)_{\varepsilon>0}$ is referred to as $\tau$-convergence in \cite{KFZ:2021}, see the paragraph below \cite[Proposition 3.1]{KFZ:2021}. There, also the necessity of the set $\omega_u$ is discussed, which here is related to the fact that $\Z_\eps(U\EEE
)$ might be disconnected into connected components by $E_\eps$ (in a  graph-theoretical  sense), and on the 
components \EEE not intersecting the Dirichlet boundary, the behavior of the displacement cannot be controlled. \EEE 

After these preparations, we now present the main results of the paper.

\begin{theorem}[Compactness] \label{theorem:compactness} Let $(u_\varepsilon,E_\varepsilon)_{\varepsilon>0}$, where $E_\varepsilon\subset\mathbb{Z}_{\varepsilon}(\Omega)$ and $u_\varepsilon\in \mathcal{U}_\varepsilon(u_0,\partial_D\Omega,E_\eps)$, be such that
\begin{align*}
\sup_{\varepsilon >0} \mathcal{F}_\varepsilon(u_\varepsilon\EEE,E_\varepsilon) <+\infty\,.
\end{align*} 
(i) Then, there exists a subsequence (not relabeled) converging to a pair $(u,E)$ in the 
${\rm d}$-sense, \EEE where $u\in GSBD^2(\Omega)$ and $E\in \mathfrak{M}(\Omega)$ is a set of finite perimeter. \\
(ii) Let $\omega_u$ be the set of finite perimeter in Definition~\ref{def:convergence}.  \EEE There exist sets of finite perimeter $(V_\eps)_{\eps>0}\subset \mathfrak{M}({\R}^3)$ 
with $Q_\eps(E_\varepsilon)\cap \Omega \subset V_\eps
$, such that\\[-10pt] 
\begin{equation}\label{ineq:small difference in measures}
\lim_{\varepsilon \to 0 }\mathcal{L}^3(V_\eps\triangle (E\cup \omega_u))=0\,,
\end{equation}
and 
\begin{equation}\label{symmetric_Cauchy-Born}
\chi_{\Omega \setminus V_\eps}
\bar{e}(u_\varepsilon) \rightharpoonup \chi_{\Omega \setminus (E \cup \omega_u)} e(u)Z  \ \ \text{weakly in $L^2_{\rm loc}(\Omega;\R^{3\times 8})$}\,.
\end{equation}
\end{theorem}

The property stated in (ii) above \EEE can be understood as a manifestation of the \emph{Cauchy-Born rule for symmetrized gradients}: for a given linear macroscopic deformation gradient $F$, the corresponding symmetrized discrete gradient is given by ${\rm sym}(F)Z$.  Note that this rule does not 
enter as an assumption, but is rather  a consequence of our analysis.

\begin{theorem}[Discrete-to-continuum $\Gamma$-convergence]\label{Main G-convergence theorem}
Under the above assumptions, the sequence of functionals $(\mathcal{F}_\varepsilon)_{\varepsilon>0}$  $\Gamma$-converges to $\mathcal{F}_0$ as $\varepsilon\to 0$, with respect to the convergence $\mathrm{d}$.   
\end{theorem}

The \AAA proofs \EEE of the above theorems will be given 
in the next sections. Concerning the elastic energy, we will follow the strategy devised in \cite{Schmidt:2009}, combined with the compactness result in \cite[Proposition 3.1]{KFZ:2021}, which uses the novel rigidity estimate in \cite[Theorem 2.1]{KFZ:2021}. 
A further ingredient is the Cauchy-Born rule for symmetrized discrete gradients, see Section \ref{sec: 4}. \EEE For the surface part, our idea relies on replacing the discrete void set by a continuum representative, and then using the result in the continuum version, i.e., \cite[Theorem 3.2]{KFZ:2021}. 

Finally, we want to stress that the form \eqref{def:varphi} of the surface energy density is justified 
by \EEE the curvature energy term: the latter ensures that energetically convenient configurations will locally be only along coordinate directions. Thus, no microscopic relaxation takes place and it suffices to calculate the energy density for half-spaces with coordinate vectors as outer unit normal, 
see the following figure.\EEE 

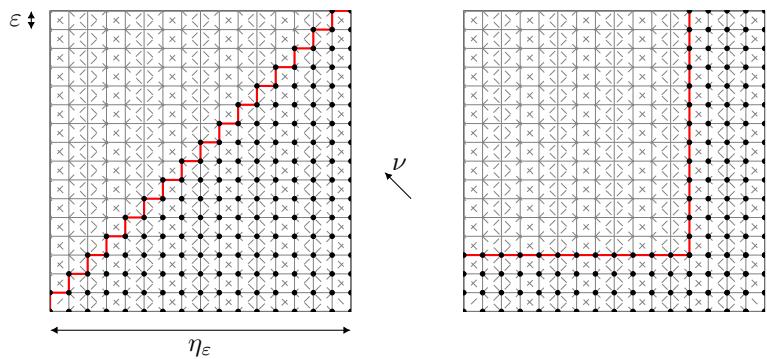
\begin{figure}[H]\label{lastfigure}
\begin{tikzpicture}[scale=.50]

\tikzset{>={Latex[width=1mm,length=1mm]}};
	
\draw[->](9.6,3)--++(135:1);

\draw(10,2.8)++(135:1) node[anchor=south]{$\nu$};
	
\draw[<->](-.5,7.5)--++(0,.5);

\draw(-.5,7.5)++(0,.25 ) node[anchor=east]{$\varepsilon$};

\draw(0,0)--++(8,0)--++(0,8)--++(-8,0)--++(0,-8);
	
\draw[<->](0,-.5)--++(8,0);
	
\draw(4,-.5) node[anchor=north]{$\eta_\varepsilon$};
	
\begin{scope}
	
\clip(0,0)--++(8,0)--++(0,8)--++(-8,0)--++(0,-8);
	
\foreach \j in {-16,...,16}{
\draw[gray,ultra thin,dashed] (-1,.5*\j)--++(18,18);
		
\draw[gray,ultra thin,dashed] (9,.5*\j)--++(-18,18);
}
	
\foreach \j in {0,...,16}{
		
\draw[gray,ultra thin] (-1,.5*\j)--++(9,0);
		
\draw[gray,ultra thin] (.5*\j,.-2)--++(0,11);
}
	
\foreach \j in {-1,...,17}{
\draw[red, thick](.5*\j,.5*\j+.5) --++(.5,0)--++(0,.5);
}
	
\foreach \j in {-1,...,17}{
		
\foreach \i in {-1,...,17}{
\draw[fill=black](.5*\j+.5*\i,.5*\j) circle(.06);
}
		
}
\end{scope}

\begin{scope}[shift={(11,0)}]


%
%
	
\draw(0,0)--++(8,0)--++(0,8)--++(-8,0)--++(0,-8);

\begin{scope}
	
\clip(0,0)--++(8,0)--++(0,8)--++(-8,0)--++(0,-8);
	
\foreach \j in {-16,...,16}{
\draw[gray,ultra thin,dashed] (-1,.5*\j)--++(18,18);
		
\draw[gray,ultra thin,dashed] (9,.5*\j)--++(-18,18);
}
	
\foreach \j in {0,...,16}{
		
\draw[gray,ultra thin] (-1,.5*\j)--++(9,0);
		
\draw[gray,ultra thin] (.5*\j,.-2)--++(0,11);
}
	
\draw[red, thick](0,1.5)--++(6,0)--++(0,9);
	
\foreach \j in {-1,...,3}{
		
\foreach \i in {-1,...,17}{
\draw[fill=black](.5*\i,.5*\j) circle(.06);
\draw[fill=black](6.5+.5*\j,.5*\i) circle(.06);
}
		
}
\end{scope}
	
\end{scope}
	
\end{tikzpicture}
\caption{Two microscopic approximations of an interface with normal $\nu=(-1/\sqrt{2},1/\sqrt{2},0)$ to illustrate that no microscopic relaxation takes place. The one on the left, although in general having less perimeter energy at the $\eps$-level compared to the right, \EEE has high curvature energy and therefore is not energetically convenient in our model\EEE. Note that for specific interaction ranges $V$, for \EEE the one on the  right,  the interactions in direction $\xi=(1,1,0)$ have a positive contribution to the energy, while for the one on the  left  they do not. The sidelength of the large cubes is $\eta_\varepsilon$, so that the configuration on the right is cubic inside the cube.} 
\end{figure}


\subsection{Example}\label{example:EAM}

We close this section with an example of a model with nearest-neighbors and next-nearest-neighbors  atomic interactions  to which our main  results  apply. For the elastic cell energy, given $K_1,K_2 >0$, we consider the bulk cell energy
$${W_{\rm bulk}^{\rm el}(F) := \frac{1}{8} \sum_{|z_i-z_j| = 1} \frac{K_1}{2} (|F_i-F_j|-1   )^2 + \frac{1}{4} \sum_{|z_i-z_j| = \sqrt{2}} \frac{K_2}{2} (|F_i-F_j|-\sqrt{2}   )^2  + \chi(F)}  $$
for each $F \in \R^{3\times 8}$, where $F_i := Fz_i$ for $i=1,\ldots,8$. The nonnegative term $\chi$ is supposed to be nonzero only for discrete gradients which are not locally orientation preserving. In particular, $\chi$ is zero in a neighborhood  of $\bar{SO}(3) = SO(3)Z$ and larger than some $c>0$ on $O(3)Z \setminus SO(3)Z$. The surface cell energy $W_{\rm surf}^{\rm el}$ is chosen appropriately. In \cite[Subsection 5.3]{Schmidt:2009}, it is discussed that this cell energy satisfies all assumptions given in Subsection \ref{Assumptions on the elastic energy}. 

For the discrete perimeter energy, we consider $V = \lbrace \pm e_k\colon k=1,2,3\rbrace \cup \lbrace \pm e_k \pm e_l \colon 1 \le k < l \le 3 \rbrace$ and choose any $c_\xi = c_{-\xi} 
>0 \EEE$ for $\xi \in V$. The sum of the elastic and the perimeter energy can be understood as an idealization of interactions of Lennard-Jones-type, cf.~\cite{BraLewOrt}. Finally, we show that we can choose a curvature cell energy in the framework of the \textit{Embedded Atom Model} \EEE (EAM) satisfying  the conditions (i)-(iii) below \eqref{def: curvature energy}.

Atomic positions induce electronic-cloud distributions. In 
the \EEE EAM, this is modeled by a  multi-body interaction term of the following form: for every  $E\subset \varepsilon\mathbb{Z}^d$, $d \in \N$, let  
\begin{align}\label{eq:EAM}
\Phi_\varepsilon(E):= \sum_{i \in \varepsilon\mathbb{Z}^d \setminus  E} \Psi_\varepsilon(\overline{\rho}_i)\,,\text{ where } \overline{\rho}_i\AAA:\EEE= \sum_{j \in \varepsilon\mathbb{Z}^d\setminus E }\rho\left(\frac{|j-i|}{\varepsilon}\right)\,.
\end{align}
Here,  $ \rho \colon \R_+ \to \R_+$ models the {\it electron-cloud contribution} of an atom placed at $j$ on an atom placed at
$i$.  The sum  $\overline \rho_i$ describes the  cumulative  effect on the atom placed at
$i$ of the electronic clouds  related to all other atoms.  Eventually,
the function $\Psi_\varepsilon \colon\R_+ \to\R_+$ describes the energy needed
to embed an atom at position $i$ in the host electron gas
created  by the other atoms. \EEE For the sake of simplifying the exposition, we prefer to consider an example for $d=2$ and  for $\Omega = \R^2$. 
We consider  
\begin{align*}
\rho(r) :=\begin{cases} 10 &\text{if } r=1\,,\\
1 &\text{if } r=\sqrt{2}\,,\\
0 &\text{otherwise}.
\end{cases}
\end{align*}
\EEE
With this particular choice, which corresponds to a two-dimensional EAM where only the cardinality of the nearest-neighbors and next-nearest-neighbors \EEE are taken into account, \EEE we observe that we can find $G_\varepsilon \colon\mathbb{N}\times \mathbb{N} \to \mathbb{R}$ such that
\begin{align*}
\Phi_\varepsilon(E) = \sum_{i \in \varepsilon\mathbb{Z}^2 \setminus E} G_\varepsilon(\#\mathcal{N}_\varepsilon(i),\#\mathcal{NN}_\varepsilon(i))\,,
\end{align*} 
where
\begin{align*}
\mathcal{N}_\varepsilon(i):=\{ j \in \varepsilon\mathbb{Z}^2\setminus E\colon |j-i|=\varepsilon\}\,, \quad 
\mathcal{NN}_\varepsilon(i)\EEE:=\{ j \in \varepsilon\mathbb{Z}^2\setminus E \colon |j-i|=\sqrt{2}\varepsilon\}\,.
\end{align*}
In fact, with our choice of $\rho$ we have that $\overline{\rho}_i \in \{k+ 10 j \colon j= \#\mathcal{N}_\varepsilon(i), k= \#\mathcal{NN}_\varepsilon(i) \}$ and $\#\mathcal{N}_\varepsilon(i),\#\mathcal{NN}_\varepsilon(i) \leq 4$. Hence, there is a one-to-one correspondence 
between $G_\varepsilon$ and $\Psi_\varepsilon$. 
We now choose 
\begin{align}\label{exdef:Geps}
G_\varepsilon(n_1,n_2):= \begin{cases} 0 &\text{if } (n_1,n_2) \in \{(3,2),(4,4)\}\,;\\
\gamma_{\varepsilon}\eta_\varepsilon^{1-q} &\text{otherwise}\,,
\end{cases}
\end{align}
and define 
\begin{align*}
F_\varepsilon^{\mathrm{curv}}(E):= \sum_{i \in \varepsilon\mathbb{Z}^2} \varepsilon^2 W_{\varepsilon,\mathrm{cell}}^\mathrm{curv}(i,E)\,,
\end{align*}
where 
\begin{align}\label{exdef:Wcellcurv} 
W_{\varepsilon,\mathrm{cell}}^\mathrm{curv}(i,E):= \sum_{j \in \mathbb{Z}_\varepsilon(Q_{\eta_\varepsilon}(i))\setminus E} 
\eta^{-2}_\varepsilon \, G_\varepsilon(\#\mathcal{N}_\varepsilon(j),\#\mathcal{NN}_\varepsilon(j))\,.
\end{align}
Then, \EEE recalling that $\eta_\varepsilon / \varepsilon\in \mathbb{N}$, by changing the order of summation, since for fixed $j \in \varepsilon\mathbb{Z}^2\setminus E$ there holds $\#\{i \in \varepsilon\mathbb{Z}^2 \colon j \in Q_{\eta_\eps}(i)\}=\eta_\varepsilon^2\varepsilon^{-2}$, 
we have
\begin{align*}
F_\varepsilon^{\mathrm{curv}}(E)=\sum_{i \in \varepsilon\mathbb{Z}^2} \varepsilon^2 W_{\varepsilon,\mathrm{cell}}^\mathrm{curv}(i,E) &= \sum_{i \in \varepsilon\mathbb{Z}^2} \varepsilon^2\sum_{j \in \mathbb{Z}_\varepsilon(
Q_{\eta_\varepsilon}(i))\setminus E}
\eta^{-2}_\varepsilon \, G_\varepsilon(\#\mathcal{N}_\varepsilon(j),\#\mathcal{NN}_\varepsilon(j)) \\&= \sum_{j \in \varepsilon\mathbb{Z}^2\setminus E}G_\varepsilon(\#\mathcal{N}_\varepsilon(j),\#\mathcal{NN}_\varepsilon(j)) =\Phi_\varepsilon(E)\,.
\end{align*}
This shows that 
$W_{\varepsilon,\mathrm{cell}}^\mathrm{curv}$ can be chosen such that $F_\varepsilon^{\mathrm{curv}}$ coincides with $\Phi_\eps$ given in \eqref{eq:EAM}\EEE. \EEE  

\begin{lemma} 
In this example, the analogous conditions {\rm (i)}, {\rm(iii)} for $d=2$ given below \eqref{def: curvature energy} are satisfied, and {\rm(ii)} holds for $Q_{\eta_{\varepsilon}-4\varepsilon}(i)$ in place of $Q_{\eta_{\varepsilon}}(i)$, cf.\ Remark \ref{rem: diffassu}. \EEE
\end{lemma}
\begin{proof}
Condition {\rm (i)} follows from the fact that flat configurations inside a square $Q_{\eta_\varepsilon}(i)$ are locally half-spaces intersected with $\mathbb{Z}_{\varepsilon}(Q_{\eta_\varepsilon}(i))$ (see Picture~ \ref{laminate_picture}, as well as Lemma \ref{lemma : laminate}), which corresponds to neighborhood cardinality  $(3,2)$ (at the flat interface) or $(4,4)$ (inside the set). Condition {\rm (iii)} follows since, if $E$ is cubic in $Q_{\eta_\varepsilon}(i)$, then $$\#\{j \in \mathbb{Z}_{\varepsilon}(Q_{\eta_\varepsilon}(i))\setminus E\EEE\colon (\#\mathcal{N}_\varepsilon(j),\#\mathcal{NN}_\varepsilon(j)) \notin \{((3,2),(4,4)\}\}  \leq 4\EEE\,.$$ 
Indeed, if $E$ is cubic in $Q_{\eta_\varepsilon}(i)$, then
\begin{align*}
E \cap Q_{\eta_{ \varepsilon}}(i) = \mathbb{Z}_\varepsilon\left(\bigcup_{m=1}^n Q_{\eta_\varepsilon}(k_m)\right)\cap Q_{\eta_{ \varepsilon}}(i)\,,
\end{align*}
for some $\{k_1,\ldots,k_n\} \subset i_0 + 
\eta_{ \varepsilon} \mathbb{Z}^2$ 
and $i_0 \in \mathbb{Z}_{\varepsilon}(Q_{\eta_{ \varepsilon}}(i))$. Note that $n\leq 4$ 
and the atoms $j$ such that $(\#\mathcal{N}_\varepsilon(j),\#\mathcal{NN}_\varepsilon(j)) \notin \{((3,2),(4,4)\}$ can only occur at the corner of $Q_{\eta_\varepsilon}(k_{ m})$ inside $Q_{\eta_\varepsilon}(i)$.  
By \eqref{exdef:Geps} this shows (iii). \EEE

Finally, we show (ii) for $\eta_\eps-4\eps$ in place of $\eta_\eps$. To this end, in view of \eqref{exdef:Geps} and \EEE\eqref{exdef:Wcellcurv}, it suffices to show that, if $E$ is not locally flat in $Q_{\eta_\varepsilon-4\varepsilon}(i)$, then there exists $j \in \mathbb{Z}_\varepsilon(Q_{\eta_\varepsilon}(i))\setminus E$ with $(\#\mathcal{N}_\varepsilon(j),\#\mathcal{NN}_\varepsilon(j)) \notin \{(3,2),(4,4)\}$. By contradiction we assume that 
\begin{align}\label{eq: contradici}
(\#\mathcal{N}_\varepsilon(j),\#\mathcal{NN}_\varepsilon(j)) \in \{(3,2),(4,4)\} \quad \text{for all $j \in \mathbb{Z}_\varepsilon(Q_{\eta_\varepsilon}(i))\setminus E$}\,.
\end{align}
The main step of the proof is then to prove that  (up to a single rotation of a multiple of $\pi/2$  independently of $j$) 
\begin{align}\label{eq:possibleN}
(\mathcal{N}_\varepsilon(j),\mathcal{NN}_\varepsilon(j)) =\begin{cases}
(\varepsilon\{-e_1,e_1,e_2\},\varepsilon\{-e_1+e_2,e_1+e_2\}) &\text{if } (\#\mathcal{N}_\varepsilon(j),\#\mathcal{NN}_\varepsilon(j))=(3,2)\,,\\
(\varepsilon\{\pm e_1, \pm e_2\},\varepsilon\{\pm(e_1-e_2),\pm (e_1+e_2)\})&\text{if } (\#\mathcal{N}_\varepsilon(j),\#\mathcal{NN}_\varepsilon(j))=(4,4)
\end{cases}
\end{align}
for all $j \in \mathbb{Z}_\varepsilon(Q_{\eta_\varepsilon}(i))\setminus E$, see also Figure~\ref{fig:adm}, where the two possible configurations (up to rotation) are depicted. 
Indeed,  
\eqref{eq:possibleN} 
would then contradict the assumption that $E$ is not locally flat in $Q_{\eta_\eps-4\eps}(i)$. 
	
It  now remains to prove \eqref{eq:possibleN}. The case $(\#\mathcal{N}_\varepsilon(j),\#\mathcal{NN}_\varepsilon(j))=(4,4)$ is clear, so we only need to consider the case $(\#\mathcal{N}_\varepsilon(j),\#\mathcal{NN}_\varepsilon(j))=(3,2)$. Assume by contradiction that, even up to rotation,
\begin{align*}
(\mathcal{N}_\varepsilon(j),\mathcal{NN}_\varepsilon(j))\neq (\varepsilon\{-e_1,e_1, e_2\},\varepsilon\{-e_1+e_2,e_1+e_2\})\,,
\end{align*}
cf.~Figure \ref{fig:adm}. Then, up to rotation and reflection, there are three cases to consider, depicted in Figure~\ref{fig:contr}, where the argument that follows is illustrated:
\begin{itemize}
\item[(a)] $(\mathcal{N}_\varepsilon(j),\mathcal{NN}_\varepsilon(j))= (\varepsilon\{-e_1,e_1,-e_2\},\varepsilon\{-e_1+e_2,e_1+e_2\})$;
\item[(b)] $(\mathcal{N}_\varepsilon(j),\mathcal{NN}_\varepsilon(j))= (\varepsilon\{-e_1,e_1,-e_2\},\varepsilon\{-e_1+e_2,e_1-e_2\})$;
\item[(c)] $(\mathcal{N}_\varepsilon(j),\mathcal{NN}_\varepsilon(j))= (\varepsilon\{-e_1,e_1,-e_2\},\varepsilon\{-e_1+e_2,-e_1-e_2\})$.
\end{itemize}
Case (a)\EEE: It is easy to see that $\#\mathcal{N}_\varepsilon(j-\varepsilon e_2) \leq 2$. Hence,  $(\#\mathcal{N}_\varepsilon(j-\varepsilon e_2),\#\mathcal{NN}_\varepsilon(j-\varepsilon e_2)) \notin \{(3,2),(4,4)\}$. This contradicts \eqref{eq: contradici}. 
\\ 
Case (b)\EEE: Here, we see that  either  $(\#\mathcal{N}_\varepsilon(j-\varepsilon e_2),\#\mathcal{NN}_\varepsilon(j-\varepsilon e_2)) \notin (3,2)$ or, if $(\#\mathcal{N}_\varepsilon(j-\varepsilon e_2),\#\mathcal{NN}_\varepsilon(j-\varepsilon e_2)) = (3,2)$, then  $\#\mathcal{N}_\varepsilon(j-2\varepsilon e_2) \leq 2$. In both cases, we get a contradiction to \eqref{eq: contradici}. 
\\
Case (c)\EEE: Here, we see that   $\#\mathcal{N}_\varepsilon(j+\varepsilon e_1) \leq 2$ which again contradicts \eqref{eq: contradici}. This concludes the proof of \eqref{eq:possibleN} and of the lemma. \qedhere
	
 \EEE


\end{proof}


\begin{figure}[H]
	\begin{tikzpicture}
	
	
\begin{scope}[shift={(4,0)}]
\draw[ultra thin, gray](-1,1)--++(2,0);
\draw[ultra thin, gray](-1,1)--++(0,-1)--++(2,0)--++(0,1);
\draw[ultra thin, gray](0,0)--++(0,1);
\draw[fill=black](0,0) circle(.05);
\draw[fill=black](1,0) circle(.05);
\draw[fill=black](-1,0) circle(.05);
\draw[fill=black](1,1) circle(.05);
\draw[fill=black](-1,1) circle(.05);
\draw[fill=black](0,1) circle(.05);

\draw(0,0) node[anchor=north]{$j$};

\end{scope}
	
\begin{scope}[shift={(8,0)}]
	
\draw[ultra thin, gray](-1,-1)--++(2,0);
\draw[ultra thin, gray](-1,0)--++(0,-1)--++(2,0)--++(0,1);
\draw[ultra thin, gray](0,0)--++(0,-1);
	
\draw[ultra thin, gray](-1,1)--++(2,0);
\draw[ultra thin, gray](-1,1)--++(0,-1)--++(2,0)--++(0,1);
\draw[ultra thin, gray](0,0)--++(0,1);
\draw[fill=black](0,0) circle(.05);
\draw[fill=black](1,0) circle(.05);
\draw[fill=black](-1,0) circle(.05);
\draw[fill=black](1,1) circle(.05);
\draw[fill=black](-1,1) circle(.05);
\draw[fill=black](0,1) circle(.05);
\draw[fill=black](1,-1) circle(.05);
\draw[fill=black](-1,-1) circle(.05);
\draw[fill=black](0,-1) circle(.05);

\draw(0,0) node[anchor=north]{$j$};

\end{scope}

\end{tikzpicture}
\caption{The two different neighborhoods for atoms $j \in Q_{\eta_{\varepsilon}}(i)\setminus E$ such that $W_{\varepsilon,\mathrm{cell}}^\mathrm{curv}(i,E)=0$.} 
\label{fig:adm}
\end{figure}
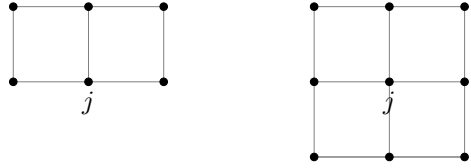

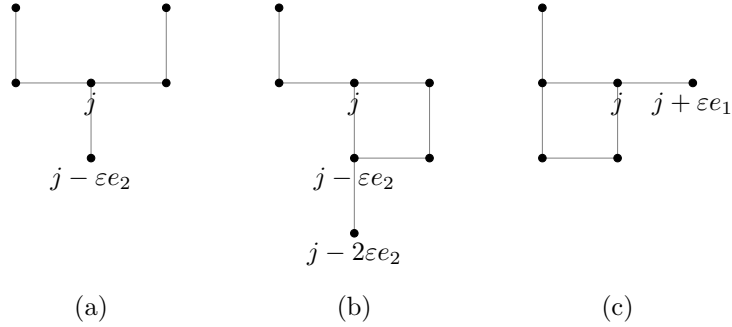
\begin{figure}[H]
\begin{tikzpicture}

\begin{scope}[shift={(2.5,0)}]
\draw[ultra thin, gray](-1,1)--++(0,-1)--++(1,0)--++(1,0)--++(0,1);
\draw[ultra thin, gray](0,0)--++(0,-1);
	
\draw[fill=black](0,0) circle(.05);
\draw[fill=black](1,0) circle(.05);
\draw[fill=black](-1,0) circle(.05);
\draw[fill=black](0,-1) circle(.05);
\draw[fill=black](1,1) circle(.05);
\draw[fill=black](-1,1) circle(.05);
\draw(0,0) node[anchor=north]{$j$};
\draw(0,-1) node[anchor=north]{$j-\varepsilon e_2$};
\draw(0,-3) node{(a)};

\end{scope}

\begin{scope}[shift={(6,0)}]
\draw[ultra thin, gray](-1,1)--++(0,-1)--++(1,0)--++(1,0)--++(0,-1);
\draw[ultra thin, gray](0,0)--++(0,-1)--++(1,0);
\draw[ultra thin, gray](0,-1)--++(0,-1);
	
\draw[fill=black](0,0) circle(.05);
\draw[fill=black](1,0) circle(.05);
\draw[fill=black](-1,0) circle(.05);
\draw[fill=black](0,-1) circle(.05);
\draw[fill=black](0,-2) circle(.05);
\draw[fill=black](1,-1) circle(.05);
\draw[fill=black](-1,1) circle(.05);
\draw(0,0) node[anchor=north]{$j$};
\draw(0,-1) node[anchor=north]{$j-\varepsilon e_2$};
\draw(0,-2) node[anchor=north]{$j-2\varepsilon e_2$};
\draw(0,-3) node{(b)};

\end{scope}
	
\begin{scope}[shift={(9.5,0)}]
\draw[ultra thin, gray](-1,1)--++(0,-1)--++(1,0)--++(1,0);
\draw[ultra thin, gray](0,0)--++(0,-1)--++(-1,0)--++(0,1);
	
\draw[fill=black](0,0) circle(.05);
\draw[fill=black](1,0) circle(.05);
\draw[fill=black](-1,0) circle(.05);
\draw[fill=black](0,-1) circle(.05);

\draw[fill=black](-1,-1) circle(.05);
\draw[fill=black](-1,1) circle(.05);
\draw(0,0) node[anchor=north]{$j$};
\draw(1,0) node[anchor=north]{$j+\varepsilon e_1$};
	
\draw(0,-3) node{(c)};

\end{scope}
	
\end{tikzpicture}
\caption{The different cases of atoms $j$ having low energy but having an atom nearby with high energy, illustrating cases (a)-(c) above, respectively. 
}
\label{fig:contr}
\end{figure}



\EEE

\section{Compactness}\label{sec: 3}

The goal of this section is to prove Theorem \ref{theorem:compactness}(i). Before we come to the actual proof, we present some auxiliary lemmata which will be essential for the following. Let $\eta_\eps>0$ be the mesoscale  introduced in Subsection \ref{subsec:surface_energy}, which  satisfies \eqref{eq: ratio} and \eqref{eq:rate gammadelta}. Let $\Omega\subset U\EEE$ be bounded open  Lipschitz  sets  in $\mathbb{R}^3$,  as introduced before Definition \ref{boundary values}, and recall also Remark \ref{U_independence}. \EEE

\begin{definition}\label{def: lam}  Let $A\in \mathfrak{M}(U\EEE)$ and let $E\subset \mathbb{Z}_{\varepsilon}(\Omega)$. We say that $E$ is a \textit{simple coordinate laminate} in $A$ if there exists $\nu \in \{e_1,e_2,e_3\}$ and a function $h \colon \varepsilon\mathbb{Z} \to \{0,1\}$ such that  for every $i\in \mathbb{Z}_{\varepsilon}(U\EEE)$, 
\begin{align*}
\chi_{\mathbb{Z}_\varepsilon(A)\cap E}(i) = h(  i \cdot \nu) \, \chi_{\mathbb{Z}_{\varepsilon}(A)\EEE}(i)\,.
\end{align*}
\end{definition}

\begin{lemma}[Laminates]\label{lemma : laminate}  Let $i\in \mathbb{Z}_{\varepsilon}(U\EEE)$ and $\eta\in [\eta_\eps,+\infty)$ be such that $Q_{\eta+ 4\varepsilon}(i)\subset U\EEE$.  
If $E \subset \mathbb{Z}_\varepsilon(\Omega)$ is locally  flat in $Q_{\eta}(i)$ in the sense of Definition~\ref{locally_flat_configurations}, then $E$ is a simple coordinate laminate in $Q_{\eta}(i)$.
\end{lemma}

\begin{proof}  We divide the proof into three steps.\\
\begin{step}{1}(Laminate structure of the ``interface cubes'')
Assume that $E\cap Q_{\eta}(i) \neq \emptyset$, since otherwise we can choose $h\equiv0$ and $\nu=e_1$. 
Let $Q := Q_{2\varepsilon}(\hat{k}) \in \hat{\mathcal{Q}}_{\eps,2}$ (cf. \eqref{def: Qeps}) be such that $Q_\eta(i)\cap Q \neq \emptyset$, $E\cap Q \neq \emptyset$, and $E\cap Q \neq \mathbb{Z}_\varepsilon(Q)$. Then, by  Definition~\ref{locally_flat_configurations}, up to rotation, we may assume that 
\begin{align}\label{eq:k}
E\cap Q = \{k, k+\varepsilon e_1, k+\varepsilon e_2,k+\varepsilon (e_1+e_2)\}\,.
\end{align}
We show that for all $q:= k  +  k_1\varepsilon e_1 + k_2\varepsilon e_2$, $k_1,k_2 \in \mathbb{Z}$, such that $Q_{\eta}(i)\cap Q_{2\varepsilon}(\hat{q})\neq\emptyset$, we have that
\begin{align}\label{eq: q}
E\cap Q_{2\varepsilon}(\hat{q})= \{q, q+\varepsilon e_1, q+\varepsilon e_2,q+\varepsilon (e_1+e_2)\}\,.
\end{align}
We prove this property inductively. First, assume that for $l:= k +\varepsilon e_1$ it holds $Q_\eta(i)\cap Q_{2\varepsilon}(\hat{l}) \neq\emptyset$.
We claim that 
\begin{align}\label{eq: l}
E\cap Q_{2\varepsilon}(\hat{l})= \{l, l+\varepsilon e_1, l+\varepsilon e_2,l+\varepsilon (e_1+e_2)\}\,.
\end{align}
Indeed, assume by contradiction that this was  not the case. Then, since $E$ is locally flat in $Q_{\eta}(i)$ and $\{l,l+\varepsilon e_2\}\subset E\cap Q_{2\varepsilon}(\hat{l})$ by \eqref{eq:k}, one of the following three possibilities holds true: 
\begin{align*}
&\mathrm{(a)}\quad \, E\cap Q_{2\varepsilon}(\hat{l}) = \mathbb{Z}_\varepsilon(Q_{2\varepsilon}(\hat{l}))\,;\  \mathrm{or}\ \\
&\mathrm{(b)}\quad \, E\cap Q_{2\varepsilon}(\hat{l}) = \{l, l+\varepsilon e_2, l+\varepsilon e_3,l+\varepsilon (e_{2}+e_{3})\}\,;\  \mathrm{or}\ \\
&\mathrm{(c)}\quad \, E\cap Q_{2\varepsilon}(\hat{l}) = \{l, l+\varepsilon e_1, l+\varepsilon e_2,l+\varepsilon (e_1+e_2)\}\,.
\end{align*}
However, due to the local flatness of $E$ in $Q_{\eta}(i)$ and \eqref{eq:k}, (a) and (b) are not possible, since in both of these cases the local flatness of $E$ with respect to $Q$ would be violated. Therefore, (c) applies, i.e., \eqref{eq: l} holds true. The same argument shows that for $m:= k+ \varepsilon e_2$, if $Q_{\eta}(i)\cap Q_{2\varepsilon}(\hat m)\neq\emptyset$, then
\begin{align*}
E\cap Q_{2\varepsilon}(\hat{m})=\{m, m+\varepsilon e_1, m+\varepsilon e_2,m+\varepsilon (e_1+e_2)\}\,.
\end{align*}
Arguing inductively in exactly the same manner, we indeed get \eqref{eq: q}.
\end{step}\\
\begin{step}{2}(One normal to the interface) Next, we show  
the following, \EEE up to a rotation of a multiple of $\pi/2$ with respect to an axis passing through a point in the shifted lattice $\varepsilon(\mathbb{Z}^3+(\tfrac{1}{2},\tfrac{1}{2},\tfrac{1}{2}))$ in the direction of one of the coordinate vectors: 
for \EEE all $n \in \varepsilon\mathbb{Z}^3$ such that for ${\hat Q\EEE} := Q_{2\varepsilon}(\hat{n})$ 
satisfying \EEE $Q_{\eta}(i)\cap \hat Q\EEE\neq \emptyset$, $E\cap {\hat Q\EEE}\neq \emptyset$, and $E\cap\hat Q\EEE \neq \mathbb{Z}_\varepsilon(\hat Q\EEE)$, we have that 
\begin{align}\label{eq:n}
E\cap\hat Q\EEE= \{n, n+\varepsilon e_1, n+\varepsilon e_2,n+\varepsilon (e_1+e_2)\}\,.
\end{align}
Without restriction we assume  by  contradiction that for $j=1,2$ there exist cubes $ \hat Q\EEE^j := Q_{2\varepsilon}(\hat{n}^j)$  with $Q_{\AAA\eta\EEE}(i)\cap \hat Q\EEE^j \neq \emptyset$, $E\cap \hat Q\EEE^j\neq \emptyset$, $E\cap\hat Q\EEE^j  \neq \mathbb{Z}_\varepsilon(\hat Q\EEE^j)$, such that 
\begin{align*}
E\cap \hat Q\EEE^1& = \{n^1, n^1+\varepsilon e_1, n^1+\varepsilon e_2,n^1+\varepsilon (e_1+e_2)\}\,, \\ E\cap \hat Q\EEE^2 & = \{n^2, n^2+\varepsilon e_1, n^2+\varepsilon e_3,n^2+\varepsilon (e_1+e_3)\}\,.
\end{align*}
 By Step 1, we get   for all $p^1:=n^1+ k_1\varepsilon e_1 + k_2\varepsilon e_2$, $k_1,k_2 \in \mathbb{Z}$, and all $p^2:=n^2+ m_1\varepsilon e_1 + m_2\varepsilon e_3$, $m_1,m_2 \in \mathbb{Z}$, such that $Q_{\eta}(i)\cap Q_{2\varepsilon}(\hat{p}^j)\neq\emptyset$ for $j=1,2$, respectively, that
\begin{align}\label{eq: p}
\begin{split}
E\cap Q_{2\varepsilon}(\hat{p}^1)&= \{p^1, p^1+\varepsilon e_1, p^1+\varepsilon e_2,p^1+\varepsilon (e_1+e_2)\}\,,  \\ E\cap Q_{2\varepsilon}(\hat{p}^2) &= \{p^2, p^2+\varepsilon e_1, p^2+\varepsilon e_3,p^2+\varepsilon (e_1+e_3)\}\,.
\end{split}
\end{align}
Now choose $q_0 \in \{ n^1+k_1\varepsilon e_1+k_2\varepsilon e_2, \ k_1, k_2 \in \mathbb{Z}\} \cap \{n^2+m_1\varepsilon e_1 + m_2\varepsilon e_3,\ m_1,m_2 \in \mathbb{Z}\}$ such that $q_0 \in \mathbb{Z}_\varepsilon(Q_{\eta}(i))$. Then \eqref{eq: p} yields  a contradiction. Therefore, \eqref{eq:n} must hold true.
\end{step}\\
\begin{step}{3}(Conclusions) We claim that (up to a possible rotation of a multiple of $\pi/2$ as before) 
\begin{align}\label{eq:conclusion} 
k \in E\cap Q_{\eta}(i)\implies \{k +\varepsilon \mathbb{Z}^2 \times\{0\}\}\cap Q_{\eta}(i) \subset E\cap Q_{\eta}(i)\,.
\end{align}
Indeed, if $k\in E\cap Q_{\eta}(i)$ is such that $E\cap  Q_{2\varepsilon}(\hat{k}) \neq \mathbb{Z}_\varepsilon(Q_{2\varepsilon}(\hat{k}))$, then  due to Step 2 and the fact that $E$ is locally flat in $Q_{\eta}(i)$, we have (up to rotation) \eqref{eq:k}. In this case, due to Step 1, \eqref{eq:conclusion} holds true. On the other hand, if $k\in  E\cap Q_{\eta}(i)$ is such that $E\cap  Q_{2\varepsilon}(\hat{k}) = \mathbb{Z}_\varepsilon(Q_{2\varepsilon}(\hat{k}))$, then again  due to Step 2 and the  fact that $E$ is locally flat in $Q_{\eta}(i)$, we have that $E\cap  Q_{2\varepsilon}(\hat{m}) = \mathbb{Z}_\varepsilon(Q_{2\varepsilon}(\hat{m}))$ for all $m \in\{k +\varepsilon \mathbb{Z}^2 \times\{0\}\} \cap Q_{\eta}(i)$. This finishes the proof of the lemma. 
\end{step}
\end{proof}


\begin{definition} Given $E \subset \mathbb{Z}_{\varepsilon}(\Omega)$, we define the \emph{$\eta_{\varepsilon}$-scale replacement} $E_{\eta_{\varepsilon}}$ of $E$  by
\begin{align}\label{def: Eeta}
E_{\eta_{\varepsilon}} := \bigcup_{i \in \eta_{ \varepsilon\EEE}\mathbb{Z}^3} \{ \mathbb{Z}_\varepsilon(Q_{\eta_{\varepsilon}}(i)) \colon E\cap Q_{\eta_{\varepsilon\EEE}}(i) \neq \emptyset\}\cap \Omega \EEE\,.
\end{align}
\end{definition}

By definition, we have $E \subset E_{\eta_\varepsilon}$.  In what follows, we will use the following fact several times: given bounded, open \EEE Lipschitz sets \EEE 
$\tilde U\EEE\subset \subset U\EEE$ in $\mathbb{R}^3$, there exists $\varepsilon_0 = \varepsilon_0(\tilde U\EEE, U\EEE)\in (0,1)$ sufficiently small so that for all $\varepsilon\in(0,\varepsilon_0]$,
\begin{equation}\label{locality of the argument}
\tilde U\EEE\subset\subset \bigcup_{Q_{\eta_{\varepsilon}}(i) \in \mathcal{Q}_{\eta_{\varepsilon},\tilde U\EEE}} Q_{4\eta_{\varepsilon}}(i)\subset\subset U\EEE\,,
\end{equation}
where $\mathcal{Q}_{\eta_{\varepsilon},\tilde U\EEE}:=\lbrace Q_{\eta_{\varepsilon}}(i) \colon i \in \eta_{\varepsilon}\mathbb{Z}^3,\ Q_{\eta_\varepsilon}(i)\cap\tilde U\EEE\neq \emptyset \rbrace$. This follows from \eqref{eq: ratio}. 
As a first step towards the proof of the compactness result, we replace a set $E\subset \mathbb{Z}_{\varepsilon}(\Omega)$ by a set that essentially can only vary its shape at the mesoscale $\eta_{\varepsilon}$. This is formulated in a local version in the next lemma.

\begin{lemma}[$\eta_{\varepsilon}$-scale replacement]\label{lemma : eta lattice}  There exists a constant  $C= C(\{c_\xi\}_{\xi \in V})>0$ such that for every $E \subset \mathbb{Z}_\varepsilon(\Omega)$, the set $E_{\eta_{\varepsilon}} \subset \mathbb{Z}_{\varepsilon}(\Omega)$ given by \eqref{def: Eeta} satisfies the following. For every open subset $\tilde U\subset \subset U$ with Lipschitz boundary, there exists $\varepsilon_0 = \varepsilon_0(\tilde U, U)\in (0,1)$, so that for every  $\varepsilon\in (0,\varepsilon_0]$ the following three properties are satisfied:\EEE
\begin{itemize}
\item[$\mathrm{(i)}$] The discrete perimeter energy of $E_{\eta_{\varepsilon}}
$, see  \eqref{def: discrete perimeter}, satisfies the estimate
\begin{align}\label{lemma ineq: eta lattice 1}
F_{\varepsilon}^\mathrm{per}(E_{\eta_{\varepsilon}},\tilde U\EEE) \leq F_\varepsilon^\mathrm{per}(E) + \omega(\varepsilon) F_\varepsilon^\mathrm{curv}(E)\,,
\end{align}
where $\omega(\varepsilon) \to 0$ as $\varepsilon\to 0$.\EEE
\item[$\mathrm{(ii)}$] The discrete curvature energy of $E_{\eta_{\varepsilon}}
$, see \eqref{def: curvature energy}, satisfies the estimate 
\begin{align}\label{lemma ineq: eta lattice 2}
F_{\varepsilon}^\mathrm{curv}(E_{\eta_{\varepsilon}},\tilde U\EEE) \leq CF_\varepsilon^\mathrm{curv}(E)\,.\EEE
\end{align}
\item[$\mathrm{(iii)}$] The cardinality of $(E_{\eta_{ \varepsilon}}
\setminus E)\cap\tilde U\EEE$ satisfies the estimate
\begin{align}\label{lemma ineq: eta lattice 3}
\#((E_{\eta_{\varepsilon}}
\setminus E)\cap \tilde U\EEE) \leq  \EEE C\eta_{\varepsilon}\EEE\varepsilon^{-3}\left( F_\varepsilon^\mathrm{per}(E) +F_\varepsilon^\mathrm{curv}(E)  \right)\,.
\end{align}
\end{itemize}
\end{lemma}
\begin{proof} 
In the proof, $C>0$ denotes a generic constant that possibly depends on $\{c_\xi\}_{\xi\in V}$ and may vary from line to line.  \EEE We will assume without loss of generality that $E_{\eta_\varepsilon}\cap\tilde U\EEE\neq \emptyset$, otherwise there is trivially nothing to prove. In the following, we suppose that $\varepsilon_0=\varepsilon_0(\tilde U, U\EEE)
\in (0,1)$ is chosen sufficiently small so that all our subsequent statements are valid for every $\varepsilon\in (0,\varepsilon_0]$.  In particular, we assume that \eqref{locality of the argument} holds true. We divide the proof into several steps. We first prove \eqref{lemma ineq: eta lattice 1} by distinguishing two different types of cubes (depending on the local geometry of $E$) and showing the estimate for each one of them separately (Step 1). Then we prove  \eqref{lemma ineq: eta lattice 2} (Step 2) and finally \eqref{lemma ineq: eta lattice 3} (Step 3). The two subfamilies of cubes in $\mathcal{Q}_{\eta_{\varepsilon},\tilde U\EEE}$ are defined as \begin{align}\label{def:good cubes}
\mathcal{Q}^{\mathrm{good}}_{\eta_{\varepsilon},\tilde U\EEE}=\left\{Q_{\eta_{\varepsilon}}(i) \in \mathcal{Q}_{\eta_{\varepsilon},\tilde U\EEE}\colon \, 
\text{$E$ simple coordinate laminate in $Q_{\frac{3}{2}\eta_\varepsilon}(i)$}\right\}\,, \  \mathcal{Q}^{\mathrm{bad}}_{\eta_{\varepsilon},\tilde U\EEE}= \mathcal{Q}_{\eta_{\varepsilon},\tilde U\EEE}\EEE \setminus\mathcal{Q}^{\mathrm{good}}_{\eta_{\varepsilon},\tilde U\EEE}\,.
\end{align}
\begin{step}{1}(Perimeter Estimate) In order to obtain estimate \eqref{lemma ineq: eta lattice 1}, we treat the good and bad cubes separately. We first start with the case of bad cubes.\\
\begin{step}{1.1}(Bad cubes) For $Q_{\eta_{\varepsilon}}(i) \in \mathcal{Q}^{\mathrm{bad}}_{\eta_{\varepsilon},\tilde U\EEE}$ we will  prove that 
\begin{align}\label{ineq:Step11}
F_\varepsilon^\mathrm{per}(E_{\eta_{\varepsilon}}
,Q_{\frac{3}{2}\eta_{\varepsilon}}(i))  \leq \omega(\varepsilon) F_\varepsilon^\mathrm{curv}(E,Q_{3\eta_{\varepsilon}}(i))\,,
\end{align}
for some $\omega(\varepsilon)\to 0$ as $\varepsilon \to 0$.

Indeed, in view of \eqref{locality of the argument}, for $Q_{\eta_{\varepsilon}}(i) \in \mathcal{Q}^{\mathrm{bad}}_{\eta_{\varepsilon},\tilde U\EEE}$  we have that $Q_{
\frac{3}{2}\EEE\eta_\varepsilon+4\varepsilon}(i)\subset U\EEE$ for all $\eps \in 
(0,\eps_0]\EEE$, given $\eps_0>0$ is small enough. (Here, recall that \EEE  $\eta_\eps/\eps \to +\infty$ as $\eps \to 0$ by \eqref{eq: ratio}.)  Then, the negation of Lemma \ref{lemma : laminate} shows that $E$ is not locally flat in  $Q_{\frac{3}{2}\eta_\varepsilon}(i)$. Thus, there exists $k\in \varepsilon\mathbb{Z}^3$ such that for $Q_{2\varepsilon}(\hat{k}) \in\hat{\mathcal{Q}}_{\eps,2}$, 
$Q_{\frac{3}{2}\eta_{\varepsilon}}(i)\cap Q_{2\varepsilon}(\hat k) \neq \emptyset$, and 
$E\cap Q_{2\varepsilon}(\hat{k})$ does not satisfy  any of the properties (i)-(iii) of Definition  \ref{locally_flat_configurations}. \EEE We actually have that $k\in \mathbb{Z}_{\varepsilon}\big(Q_{2 \eta_{\varepsilon}}(i)\big)$, as long as we take $\varepsilon_0>0$ small enough.
 
Now, for all $j\in \mathbb{Z}_{\varepsilon}\big(Q_{\frac{\eta_{\varepsilon}}{4}}(k)\big)$ we have that $Q_{2\varepsilon}(\hat{k}) \subset Q_{\frac{\eta_{\varepsilon}}{2}}(j)$ and therefore, due to assumption (ii) of the curvature energy \eqref{def: curvature energy} (see also Remark \ref{rem: diffassu}), we get 
\begin{align}\label{ineq: Wcellk}
W_{\varepsilon,\mathrm{cell}}^\mathrm{curv}(j,E) \geq c\gamma_{\varepsilon} \eta_{\varepsilon}^{-1-q}\,.
\end{align}
Notice that, since $j\in \mathbb{Z}_{\varepsilon}\big(Q_{\frac{\eta_{\varepsilon}}{4}}(k)\big)$ and $k\in \mathbb{Z}_{\varepsilon}\big(Q_{2\eta_{\varepsilon}}(i)\big)$, we get that $j\in \mathbb{Z}_{\varepsilon}\big(Q_{3\eta_{\varepsilon}}(i)\big)$. This, together with the simple fact that  
$\# \mathbb{Z}_\varepsilon(Q_{\frac{\eta_{\varepsilon}}{4}}(k)) \geq C\frac{\eta^3_{\varepsilon}}{\varepsilon^3}$ 
(for an absolute constant $C>0$)  and \eqref{ineq: Wcellk}, allows us to deduce that
\begin{align}\label{ineq:curv}
\begin{split}
F_\varepsilon^\mathrm{curv}\big(E,Q_{3 \eta_{\varepsilon}}(i)\big) &=\sum_{j\in \mathbb{Z}_{\varepsilon}(Q_{ 3\eta_{\varepsilon}}(i))}\varepsilon^3W^{\mathrm{curv}}_{\varepsilon,\mathrm{cell}}(j,E) \EEE \geq \sum_{j \in \mathbb{Z}_\varepsilon(Q_{\frac{\eta_{\varepsilon}}{4}}(k))} \varepsilon^3 W_{\varepsilon,\mathrm{cell}}^\mathrm{curv}(j,E)\\
&\geq \# \mathbb{Z}_\varepsilon(Q_{\frac{\eta_{\varepsilon}}{4}}(k))\varepsilon^3\cdot c\gamma_{\varepsilon}\eta_{\varepsilon}^{-1-q} \geq c\gamma_{\varepsilon}\eta_{\varepsilon}^{2-q}\,.
\end{split}
\end{align}
On the other hand, we claim that 
\begin{align}\label{ineq:surfacebadcubes}
\begin{split}
F_\varepsilon^\mathrm{per}(E
_{\eta_{\varepsilon}},Q_{\frac{3}{2}\eta_{\varepsilon}}(i)) \leq C\eta_{\varepsilon}^2\,.
\end{split}
\end{align}
Indeed, due to the assumptions on the set of neighbors $V$, see the beginning of Subsection \hyperref[subsec:surface_energy]{2.3}, we have $\#V \leq 27$ and therefore for $\varepsilon_0\in (0,1)$ sufficiently small, for every $\varepsilon\in (0,\varepsilon_0]$ (and hence also for $\eta_\varepsilon>0$ small accordingly) we can  estimate  
\begin{align*}
F_\varepsilon^\mathrm{per}(E_{\eta_{\varepsilon}}, Q_{\frac{3}{2}\eta_{\varepsilon}}(i)) &=\sum_{j\in \mathbb{Z}_{\varepsilon}(Q_{\frac{3}{2}\eta_{\varepsilon}}(i))\cap E_{\eta_{\varepsilon}}
}\underset{j+\varepsilon\xi \in U}{\sum_{\xi\in V}}\varepsilon^2c_{\xi}\big(1-\chi\EEE_{E_{\eta_{\varepsilon}}
}(j+\varepsilon\xi)\big)\EEE \\
& \leq 27\big(\max_{\xi \in V} c_\xi\big)\varepsilon^2\cdot \#\mathbb{Z}_\varepsilon( \BBB (\partial Q_{
\eta_{\varepsilon}}(i))_{2\eps} \EEE ) \leq C  \varepsilon^{-1} \mathcal{L}^3( \BBB (\partial Q_{
\eta_{\varepsilon}}(i))_{2\eps} \EEE) \leq C\eta^2_{ \varepsilon}\,,
\end{align*}
where we used the fact that $0<\varepsilon\leq 
\eta_\varepsilon$, and that by the construction of the set $E_{\eta_\varepsilon}$, missing neighbors can only occur at an $\varepsilon$-layer of atoms around $\partial Q_{\eta_\varepsilon}(i)$. The last constant $C>0$ depends on $\{c_\xi\}_{\xi\in V}$ here. \EEE This implies \eqref{ineq:surfacebadcubes}. Noting that, due to \eqref{eq:rate gammadelta}, we have that $\gamma_{\varepsilon} \eta_{\varepsilon}^{-q} \to +\infty$ as $\varepsilon\to 0$, \EEE and using \eqref{ineq:curv} together with \eqref{ineq:surfacebadcubes}, we obtain \eqref{ineq:Step11} for $\omega(\varepsilon):=C\gamma_{\varepsilon}^{-1}\eta_\varepsilon^q$.  This concludes Step 1.1.
\end{step} 
Next, we treat the case of good cubes.\\
\begin{step}{1.2}(Good cubes) For $Q_{\eta_{\varepsilon}}(i) \in\mathcal{Q}^{\mathrm{good}}_{\eta_{\varepsilon},\tilde U\EEE}$ we will prove that
\begin{align}\label{ineq:Step12}
F_\varepsilon^\mathrm{per}(E_{\eta_{\varepsilon}}, Q_{\eta_{\varepsilon}}(i)) \leq F_\varepsilon^\mathrm{per}(E,Q_{\eta_{\varepsilon\EEE}}(i))\,.
\end{align}
By \eqref{def: discrete perimeter} and \eqref{def: Eeta} we can assume without restriction that $E\cap Q_{\eta_{\varepsilon}}(i)
\neq \emptyset$. Additionally, due to the definitions in \eqref{def:good cubes}, $E$ is a simple coordinate laminate in $Q_{\frac{3}{2}\eta_{\varepsilon}}(i)$ with respect to a unit normal vector $\nu\in\{e_1,e_2,e_3\}$. We assume without loss of generality that $\nu=e_3$. Since $ E\cap Q_{\eta_{\varepsilon}}(i)\neq \emptyset$, we have that $E_{\eta_{\varepsilon}}
\cap Q_{\eta_{\varepsilon}}(i) = \mathbb{Z}_\varepsilon(Q_{\eta_{\varepsilon}}(i))$. As $E$ is a simple laminate in $Q_{\frac{3}{2}\eta_{\varepsilon}}(i)$ with respect to $e_{3}$, also $E_{\eta_\eps}$ is a simple laminate in $Q_{\frac{3}{2}\eta_{\varepsilon}}(i)$ with respect to $e_3$. Thus, we have that in the formula \eqref{def: discrete perimeter} for $F_{\varepsilon}^{\mathrm{per}}(E_{\eta_\varepsilon}, Q_{\eta_{\varepsilon}}(i))$, the addends $\xi\in V$ yielding a positive contribution satisfy 
$\xi_3\neq 0$\OOO. \EEE Without restriction we can  consider the case $\xi_3>0$. Let $j \in \mathbb{Z}_\varepsilon(Q_{\eta_{\varepsilon}}(i))$ and $\xi \in V$ be such that $j+\varepsilon\xi \notin E_{\eta_{\varepsilon}}
$, hence, 
$|(j+\varepsilon\xi)_3-i_3|=\frac{\eta_\varepsilon}{2}$.  Let now  $k \in  E\cap Q_{\eta_{\varepsilon}}(i)$ be the point with maximal $e_3$-coordinate among all other points in $E\cap Q_{\eta_\varepsilon}(i)$ in the same $e_3$-column as $j$, i.e., \EEE  $k_1=j_1,k_2=j_2$, and  $ m_3\leq k_3\leq j_3$ for all $m:=(j_1,j_2,m_3) \in  E\cap Q_{\eta_{\varepsilon}}(i)$. Now either $k=j$, in which case we have that  $k+\varepsilon\xi \notin E$ since $ E\subset E_{\eta_\varepsilon}\EEE$. Otherwise, if $k \neq j$, then $-\frac{\eta_{\varepsilon}}{2}\leq  (k+\varepsilon\xi)_3 -i_3 \le \EEE \frac{\eta_{\varepsilon}}{2}- \varepsilon$. Therefore, due to the choice of $k$, we again have that $k+\varepsilon\xi \notin  E $. Note that in the previous procedure, since $\|\xi\|_\infty\leq 1$, for each $\xi \in V$ we have that each $k \in E\cap Q_{\eta_{\varepsilon}}(i)$ (with $k+\varepsilon\xi\notin E$)  is chosen at most once for a point $j \in \mathbb{Z}_\varepsilon(Q_{\eta_{\varepsilon}}(i)) \cap  E_{\eta_{\varepsilon}}$ with $j+\varepsilon\xi\notin E_{\eta_\varepsilon}$. Therefore, in \eqref{def: discrete perimeter}, summing over all $j$ for $E_{\eta_\varepsilon}$ \EEE and the corresponding $k$ for $E$, we obtain \eqref{ineq:Step12}.
\end{step}\\
\begin{step}{1.3}(Conclusion) Combining now \eqref{ineq:Step11} and \eqref{ineq:Step12}, and recalling the notation in \eqref{def:good cubes}, we obtain
\begin{align*}
F_{\varepsilon}^\mathrm{per}(E_{\eta_{\varepsilon}}, \tilde U\EEE)&
\leq \sum_{Q_{\eta_{\varepsilon}}(i)\in \mathcal{Q}^{\mathrm{good}}_{\eta_{\varepsilon},\tilde U\EEE}}
F_{\varepsilon}^\mathrm{per}(E_{\eta_{\varepsilon}}, Q_{\eta_\varepsilon}(i))+\sum_{Q_{\eta_\varepsilon}(i)\in \mathcal{Q}^{\mathrm{bad}}_{\eta_{\varepsilon},\tilde U\EEE}
}F_{\varepsilon}^\mathrm{per}(E_{\eta_{ \varepsilon}},Q_{\frac{3}{2}\eta_\varepsilon}(i))\\
&\leq \sum_{Q_{\eta_{\varepsilon}}(i)\in \mathcal{Q}^{\mathrm{good}}_{\eta_{\varepsilon}, \tilde U\EEE}
}F_{\varepsilon}^\mathrm{per}(E,Q_{\eta_\varepsilon}(i))+ \omega(\eps)\EEE \sum_{Q_{\eta_\eps}(i)\EEE\in \mathcal{Q}^{\mathrm{bad}}_{\eta_{\varepsilon}, \tilde U\EEE}
}F_{\varepsilon}^\mathrm{curv}(E,Q_{3\eta_\varepsilon}(i))\\
&\leq F_{\varepsilon}^\mathrm{per}(E)+C\omega(\varepsilon)F_{\varepsilon}^\mathrm{curv}(E)\,,
\end{align*}
where we used \eqref{ineq:Qintersectbound} and \eqref{locality of the argument}. This finishes the proof of \eqref{lemma ineq: eta lattice 1}.
\end{step}
\end{step}\\
\begin{step}{2}(Curvature Estimate) Next, we prove \eqref{lemma ineq: eta lattice 2}. 
This follows from the property that for every $Q_{\eta_\eps}(i)\in \mathcal{Q}_{\eta_\eps, \tilde U\EEE}
$ \EEE
we have
\begin{align}\label{ineq:curvcube}
F_\varepsilon^\mathrm{curv}(E_{\eta_{\varepsilon}},Q_{\eta_{\varepsilon}}(i)) \leq C F_\varepsilon^\mathrm{curv}(E, Q_{3\eta_{\varepsilon}}(i))\,,
\end{align}
for an absolute constant $C>0$. 
Indeed, once the above inequality \EEE is shown, by \EEE summing over all $Q_{\eta_{\varepsilon}}(i)\in \mathcal{Q}_{\eta_\eps,\tilde U\EEE}\EEE$, \eqref{lemma ineq: eta lattice 2} follows from \eqref{ineq:Qintersectbound}, \eqref{locality of the argument}, and \eqref{ineq:curvcube}. 

We now show \eqref{ineq:curvcube}. \EEE  To this end, let $Q_{\eta_{\varepsilon}}(i)\in \mathcal{Q}_{\eta_\eps,\tilde U\EEE}$ \EEE
and assume without loss of generality that $E_{\eta_{\varepsilon}}$ is not flat in $Q_{2\eta_{\varepsilon}}(i)$, since otherwise the statement is trivial by \eqref{def: curvature energy} and assumption (i) on $W_{\varepsilon,\mathrm{cell}}^{\mathrm{curv}}$. Since $E_{\eta_\eps}$ is cubic in $Q_{2\eta_{\varepsilon}}(i)$\EEE, this also implies that $E_{\eta_\eps}$ is not locally flat in $Q_{2\eta_{\varepsilon}}(i)$\EEE, see Definition \ref{flatness-new}. We can also assume that $E$ is not locally flat in $Q_{\frac{5}{2} \eta_{\varepsilon}}(i)$, since otherwise $E_{\eta_\varepsilon}$ would be flat in $Q_{\frac{5}{2}\eta_{\varepsilon}}(i)$, and again the statement would be trivial. 
Consequently, as  $E_{\eta_{\varepsilon}}$ is not  locally flat in $Q_{\frac{5}{2}\eta_{\varepsilon}}(i)$, \EEE we have that, up to rotation, there exists $\nu \in \R^3$ with $\nu_i\in\{0,1\}$ for all $i=1,2,3$, such that $ E_{\eta_{ \varepsilon}} \cap Q_{2\varepsilon}(i+\frac{\eta_{\varepsilon}}{2}\nu)$ 
\EEE does not satisfy any of the conditions (i)-(iii) of Definition \ref{locally_flat_configurations}. Moreover, since $E$ is not locally flat in $Q_{\frac{5}{2}\eta_{\varepsilon}}(i)$, there exists $Q_{2\varepsilon}(\hat{k}) \in \hat{\mathcal{Q}}_{\eps,2}$ such that $Q_{\frac{5}{2}\eta_{\varepsilon}}(i)\cap Q_{2\varepsilon}(\hat{k})\neq \emptyset$, but in $Q_{2\varepsilon}(\hat{k})$ the set $E$  does not satisfy any of the conditions  (i)-(iii) of Definition \ref{locally_flat_configurations}. We can now proceed similarly to Step~1.1: for all $j\in \Z_\eps(Q_{\frac{\eta_{ \varepsilon}}{4}}(k))\EEE$ we have that $Q_{2\varepsilon}(\hat{k}) \subset Q_{\frac{\eta_{\varepsilon}}{2}}(j)\subset Q_{3\eta_\eps}(i)\EEE$ and therefore, due to  property (ii) below \eqref{def: curvature energy} (see also Remark \ref{rem: diffassu}), we have 
\begin{align*}
W_{\varepsilon,\mathrm{cell}}^\mathrm{curv}(j,E) \geq c\gamma_{\varepsilon} \eta^{-1-q}_{\varepsilon}\,.
\end{align*}
As before in Step 1.1, along with the fact that $\# \mathbb{Z}_\varepsilon(Q_{\frac{\eta_{\varepsilon}}{4}}(k))\geq C\frac{\eta_\varepsilon^3}{\varepsilon^3}$, 
this yields \EEE  
\begin{align}\label{ineq:curv12}
F_\varepsilon^\mathrm{curv}(E,Q_{3\eta_{\varepsilon}}(i)) \geq \sum_{j \in \mathbb{Z}_\varepsilon(Q_{\frac{\eta_{\varepsilon}}{4}}(k))} \varepsilon^3 W_{\varepsilon,\mathrm{cell}}^\mathrm{curv}(j,E) \geq 
c\gamma_{\varepsilon}\eta_{\varepsilon}^{2-q}\,.
\end{align}
On the other hand, due to \eqref{def: Eeta}, we have that $E_{\eta_\varepsilon}$ is cubic in $Q_{\eta_\varepsilon}(j)$ for all $j \in \mathbb{Z}_\varepsilon(Q_{\eta_{\varepsilon}}(i))$, \EEE 
and therefore by property (iii) below \eqref{def: curvature energy} and  the fact that $\#\mathbb{Z}_\varepsilon(Q_{\eta_{\varepsilon}}(i)) =\eta^3_{\varepsilon}/\varepsilon^{3}$, we have  
\begin{align}\label{ineq:curv3}
F_\varepsilon^\mathrm{curv}(E_{\eta_{\varepsilon}},Q_{\eta_{\varepsilon}}(i)) =\sum_{j \in \mathbb{Z}_\varepsilon(Q_{\eta_{\varepsilon}}(i))} \varepsilon^3 W_{\varepsilon,\mathrm{cell}}^\mathrm{curv}(j, E_{\eta_{\varepsilon}})
\leq C\gamma_{\varepsilon}\eta^{2-q}_{\varepsilon}\,.
\end{align}
Hence, \eqref{ineq:curv12} and \eqref{ineq:curv3} imply \eqref{ineq:curvcube}. 
This \EEE concludes Step 2.
\end{step}\\
\begin{step}{3}(Cardinality of the difference) The goal of this final step is to prove \eqref{lemma ineq: eta lattice 3}. Recall the notation introduced in \eqref{def:good cubes}. Due to \eqref{ineq:Qintersectbound}, \eqref{eq:rate gammadelta}, \eqref{locality of the argument}, and \eqref{ineq:curv}, for all $\varepsilon\in (0,\varepsilon_0]$ we have 
\begin{align}\label{ineq:curvaturemeasure}
\begin{split}
\hspace{-1em}\#\mathcal{Q}^{\mathrm{bad}}_{\eta_\varepsilon,\tilde U\EEE} \EEE &\leq C\gamma_{\varepsilon}^{-1}\eta^{-2+q}_{\varepsilon} \sum_{Q_{\eta_{\varepsilon}}(i)\in \mathcal{Q}^{\mathrm{bad}}_{\eta_\varepsilon,\tilde U\EEE} \EEE} F_\varepsilon^\mathrm{curv}(E,Q_{3 \eta_{\varepsilon}}(i))\\&\leq C\eta^{-2}_{\varepsilon} \sum_{Q_{\eta_{\varepsilon}}(i)\in\mathcal{Q}^{\mathrm{bad}}_{\eta_\varepsilon,\tilde U \EEE}} F_\varepsilon^\mathrm{curv}(E,Q_{ 3\eta_{\varepsilon}}(i))\leq C\eta_\varepsilon^{-2} F_\varepsilon^\mathrm{curv}(E)\,.
\end{split}
\end{align}\EEE
\EEE Now, let $Q_{\eta_{\varepsilon}}(i) \in \mathcal{Q}_{\eta_\varepsilon,\tilde U}^{\mathrm{good}} \EEE$ be such that $(E_{\eta_{\varepsilon}}\setminus E)\cap Q_{\eta_{\varepsilon}}(i)\EEE\neq \emptyset$, i.e., $E\cap Q_{\eta_{ \varepsilon\EEE}}(i) \neq \mathbb{Z}_\varepsilon(Q_{\eta_{\varepsilon}}(i))$ and also $ E\cap Q_{\eta_{\varepsilon}}(i) \neq \emptyset$. Then, since $E$ is a simple coordinate laminate in $Q_{\frac{3}{2}\eta_{\varepsilon}}(i)$, assuming again without loss of generality that $\nu=e_3$, \EEE we can find $k \in  E\cap\EEE Q_{\eta_{\varepsilon\EEE}}(i)$ such that $k+\varepsilon e_3 \notin E\EEE$ or $k- \varepsilon  e_3 \notin E$. By the laminate structure of $E$, the same holds true for all $m \in \{k+(\varepsilon\mathbb{Z}^2\times\{0\})\} \cap Q_{\eta_{\varepsilon}}(i)$. Thus,
\begin{align*}
F_\varepsilon^\mathrm{per}(E,Q_{\eta_{\varepsilon}}(i)) \geq c_{e_3} \eps^2   (\eta_\eps^2/\eps^2)= c_{e_3}\eta_{\varepsilon}^2\,,
\end{align*}
and therefore, for a constant $C>0$ that additionally depends on $\{c_\xi\}_{\xi\in V}$,
\begin{align}\label{ineq:surfacemeasuregood}
\begin{split}
\#\big\{Q_{\eta_{\varepsilon}}(i)\in\mathcal{Q}_{\eta_\varepsilon,\tilde U}^{\mathrm{good}} \EEE \colon (E_{\eta_{\varepsilon}}\setminus E)\cap Q_{\eta_{\varepsilon}}(i)\neq \emptyset\big\}&\leq C\eta_{\varepsilon}^{-2} \sum_{Q_{\eta_{\varepsilon}}(i)\in\mathcal{Q}_{\eta_\varepsilon,\tilde U}^{\mathrm{good}} \EEE}\EEE F_\varepsilon^\mathrm{per}(E\EEE,Q_{\eta_{\varepsilon}}(i))\\&\leq C\eta_\varepsilon^{-2}F_\varepsilon^\mathrm{per}(E)\,,
\end{split}
\end{align}
where we again used \eqref{locality of the argument}. 
Since on each cube we have that $\#((E_{\eta_{\varepsilon}}\setminus E)\cap Q_{\eta_{\varepsilon}}(i))\leq \eta^3_{\varepsilon}/\varepsilon^3$, by \eqref{ineq:curvaturemeasure} and \eqref{ineq:surfacemeasuregood},   we deduce that
\begin{align*}
\#((E_{\eta_{\varepsilon}}\setminus E)\cap \tilde U\EEE) &
\leq \frac{\eta_{\varepsilon}^3}{\varepsilon^3}\left(\#\mathcal{Q}_{\eta_\varepsilon,\tilde U}^{\mathrm{bad}}\EEE + \#\big\{Q_{\eta_{\varepsilon}}(i)\in \mathcal{Q}_{\eta_\varepsilon,\tilde U}^{\mathrm{good}} \EEE \colon (E_{\eta_{\varepsilon}}\setminus E)\cap Q_{\eta_{\varepsilon}}(i)\neq \emptyset\big\} \right) \\&\leq C\eta_{\varepsilon} \varepsilon^{-3}\left( F_\varepsilon^\mathrm{per}(E)+F_\varepsilon^\mathrm{curv}(E)\right)\,.
\end{align*}
This proves \eqref{lemma ineq: eta lattice 3} and concludes the proof of the lemma.
\end{step}
\end{proof}
 
For a set $\hat E \EEE\subset U\EEE$ of finite perimeter, we will also denote by 
\begin{equation*}
\mathrm{Per}_\varphi(\hat E \EEE, A):=\int_{\partial^* \hat E \EEE\cap A}\varphi(\nu_{\hat E \EEE})\,\mathrm{d}\mathcal{H}^2\, 
\end{equation*}
its anistropic perimeter in $A \subset U\EEE$ with respect to the density $\varphi$ defined in $\eqref{def:varphi}$. Again, if $A=U\EEE$, we omit the dependence on $A$ in the notation. Moreover, for a regular set $\hat E \EEE \in  \mathcal{A}_{\mathrm{reg}}(U\EEE)$, we denote by $\bm{A}$ the second fundamental form  of $\partial \hat E \EEE\cap U\EEE$, i.e., $|\bm{A}| = \sqrt{\kappa_1^2 + \kappa_2^2}$, where $\kappa_1$ and $\kappa_2$  are the principal curvatures of $\partial\hat E \EEE\cap U\EEE$. As a second preliminary step, we  replace the set $E_{\eta_{\varepsilon}}$ by a regular set, whose perimeter and curvature energy can be controlled by the discrete perimeter and curvature energy of $E_{\eta_{\varepsilon}}$, respectively. This is formulated, again in a local fashion, in the following lemma. 

\begin{lemma}[Smoothening of cubic sets]  \label{lemma:eta scale smoothening}  There exist constants $C= C(\{c_\xi\}_{\xi \in V})>0$,  $c_0>0$, and for every $E \subset \mathbb{Z}_\varepsilon(\Omega)$ there exists a set $\hat E\in \mathcal{A}_{\mathrm{reg}}(U\EEE)$ with $ Q_\varepsilon(E)\cap U\EEE\subset\hat E\subset U\EEE$ (see \eqref{def:Aeps}), so that the following holds true. For every open subset $\tilde U\EEE\subset \subset U\EEE$ with Lipschitz boundary, there exists  $\varepsilon_0=\varepsilon_0(\tilde U\EEE,U\EEE)\in (0,1)$, 
such that: \EEE

\begin{itemize}
\item[\rm (i)] For the anisotropic perimeter of $\hat E$ with density $\varphi$, we have 
\begin{align}\label{lemma:perimeterdis-cont}
\mathrm{Per}_{\varphi}(\hat E,\tilde U\EEE) \leq F_\varepsilon^\mathrm{per}( E) + \omega(\varepsilon) F_\varepsilon^\mathrm{curv}(E)\,,
\end{align}
 where $\omega(\varepsilon)\to 0$ as $\varepsilon\to 0$\OOO.\EEE
\item[\rm (ii)]  For the curvature energy of 
$\hat E$, \EEE we have
\begin{align}\label{lemma:curvaturedis-cont}
\gamma_{\varepsilon}\int_{\partial \hat E\cap \tilde U\EEE} |\bm{A}|^q\,\mathrm{d}\mathcal{H}^{2} \leq CF_\varepsilon^\mathrm{curv}(E)\,.\EEE
\end{align}
\item[\rm (iii)] For the difference in volumes, we have 
\begin{align}\label{lemma:measuredis-cont}
 \mathcal{L}^3((\hat E\setminus Q_\varepsilon(E))\cap\tilde U\EEE) \leq C\eta_{\varepsilon} \left(F_\varepsilon^\mathrm{per}(E) + F_\varepsilon^\mathrm{curv}(E)\right)\,.\EEE
\end{align}
\item[\rm (iv)] We have
\begin{align}\label{lemma:distance-cont}
\mathrm{dist}(\tilde U\EEE\setminus\hat E, Q_\eps(E)\cap U\EEE) \geq c_0\eta_\varepsilon\,.
\end{align}
\end{itemize}
\end{lemma}
\begin{proof} We divide the proof into several steps. In the first step, we explain how to reduce the problem to the case that $E$ is an $\eta_\varepsilon$-scale set in the sense of \eqref{def: Eeta}. In Step 2 we show how to construct the set $\hat{E}$ for $\eta_\varepsilon=1$. In Step 3 we discuss the shape of the interface of the smooth replacement set and derive the estimates in the case that $\eta_\varepsilon=1$. In Step 4 we perform a scaling argument and derive the corresponding estimates for general $\eta_\varepsilon>0$, which allows us to conclude.\\
\begin{step}{1}(Reduction to $\eta_\varepsilon$-scale sets) 
First of all, without loss of generality we can assume that $E$ is an $\eta_\varepsilon$-scale set in the sense of \eqref{def: Eeta}. To this end, let $U'\EEE \subset\subset U\EEE$ be an open Lipschitz set \EEE such that $\tilde{U}\EEE\subset\subset U'\EEE$. Let $\varepsilon_0>0$ be small enough such that Lemma \ref{lemma : eta lattice} is applicable for $U'\EEE$ in place of $U\EEE$. 
We  \EEE apply Lemma~\ref{lemma : eta lattice} for $ U'\EEE$ to obtain $E_{\eta_\eps}$  such that $E \subset E_{\eta_\eps}$ and \eqref{lemma ineq: eta lattice 1}--\eqref{lemma ineq: eta lattice 3} hold true for $E_{\eta_\eps}$. We perform the construction presented in Steps \EEE 2--4 below for $E_{\eta_\eps}$ in place of $E$ and for $U'\EEE$ in place of $U\EEE$ to get \eqref{lemma:perimeterdis-cont}--\eqref{lemma:distance-cont} for $E_{\eta_\eps}$ 
in place of $E$. Then, we see \EEE that \eqref{lemma:perimeterdis-cont}--\eqref{lemma:distance-cont} also hold for $E$ due to \eqref{lemma ineq: eta lattice 1}--\eqref{lemma ineq: eta lattice 3}, $E \subset E_{\eta_\eps}$, and $\mathcal{L}^3((Q_\eps(E_{\eta_\eps}) \setminus Q_\eps(E)) \cap\tilde U\EEE) \le \eps^3 \#( (E_{\eta_\eps} \setminus E) \cap \tilde U\EEE)$.  
\EEE Therefore, from now on we can assume that $E$ is an $\eta_\varepsilon$-scale set.
\end{step}\\
\begin{step}{2}(Construction of the set at scale $\eta_\varepsilon=1$) Let $E\subset \mathbb{Z}^3$ be a finite set and (according to the notation we introduced in \eqref{def:Aeps} and the comments 
thereafter)  \EEE let $Q(E) := \bigcup_{i \in E}Q(i)\,.$
\\ Let $(\zeta_\sigma)_{\sigma>0}$ be a family of smooth symmetric mollifiers in $\mathbb{R}^3$, i.e., there exists a radially symmetric function $\zeta \in C_c^\infty(\mathbb{R}^3; [0,+\infty))$ such that
\begin{align}\label{eq:propmolifier}
\int_{\mathbb{R}^3} \zeta(x)\,\mathrm{d}x=1\,, \quad \mathrm{supp}\,\zeta \subset\subset B_1\,, \quad \zeta_\sigma(x) = \sigma^{-3} \zeta(\sigma^{-1} x)\,. \end{align}
For every $\sigma>0$ consider the function $f_\sigma(x):=\left(\chi_{Q(E)}\ast \zeta_\sigma\right)(x)$ and for $t \in (0,1) $ consider the super-level set $E_{\sigma,t}$ defined as $E_{\sigma,t}:=\{x \in \mathbb{R}^3 \colon f_\sigma(x) >t\}$. (Here, for notational convenience, we have suppressed the dependence of $f_\sigma$ on $E$.) Our goal is to choose $c_0>0$, $\sigma_0>0$, and $t_0 \in (0,1)$ independently of $E$ such that
\begin{itemize}
\item[(i)] $\partial E_{\sigma_0,t_0}$ is a regular $2$-dimensional manifold;
\item[(ii)] $\mathrm{dist}_{\mathcal{H}}(E_{\sigma_0,t_0},Q(E)) \leq \sigma_0$\,;
\item[(iii)] $\displaystyle\inf_{x\in \mathbb{R}^3 \setminus E_{\sigma_0,t_0}} \mathrm{dist}(x,Q(E)) \geq c_0$\,.
\end{itemize} 
To see these properties\EEE, we define the $N_0 := 2^{27}$ different sets  \EEE $(E^k)_{k=1,\ldots,N_0}$ by $E^k := \bigcup_{j \in I_k} Q(j)$, $I_k \subset \{-1,0,1\}^3$. We note that  for each   $i \in \mathbb{Z}^3$ we find $k$ such that
\begin{align}\label{eq:localLambda}
Q(E) \cap Q_3(i) = E^k +i\,.
\end{align}
In other words, there are only $N_0$ different possibilities for the shape of $Q(E) \cap Q_3(i)$, $i \in \mathbb{Z}^3$. For $\sigma_0>0$ small enough, by \eqref{eq:localLambda}, we have that $f_{\sigma_0}^k(x):=(\chi_{E^k} \ast \zeta_{\sigma_0})(x) =(\chi_{Q(E)} \ast \zeta_{\sigma_0})(x)$ for all $x\in Q_{\frac{3}{2}}(i)$. Now let $\mathcal{N}_k$ be the set of critical values of $f_{\sigma_0}^k$, which by the Morse-Sard Lemma (cf. Lemma 13.15 in \cite{maggi2012sets}) is a set of $\mathcal{L}^1$-measure zero. In particular, for $\mathcal{N}:=\bigcup_{k=1}^{N_0}\mathcal{N}_k$ we have that $\mathcal{L}^1(\mathcal{N})=0$. Thus, we can choose $t_0 \in (0,1)\EEE\setminus \mathcal{N}$ such that (i) is satisfied. In order to obtain (iii), due to \eqref{eq:localLambda}, it suffices to choose $t_0$ and $c_0>0$ small enough depending on $\zeta_{\sigma_0}$ such that (iii) is satisfied for all $E^k$, $k \in \{1,\ldots,N_0\}$, in place of $E$. Then, as $E_{\sigma_0,t_0} \supset Q(E)$, (ii) follows from the fact that $\mathrm{supp}\,\zeta_{\sigma_0} \subset\subset B_{\sigma_0}$.
\end{step}\\
\begin{step}{3}(Surface energy estimate in the case  $\eta_\varepsilon=1$)
Similarly to the proof of Lemma \ref{lemma : eta lattice}, we will consider the two families of good and bad cubes with respect to $E$, namely 
\begin{align}\label{def:good cubes_eta_scale}
\mathcal{Q}^{\mathrm{good}}_{E} =\left\{Q(i) \colon i \in E,\ E \ \text{simple coordinate laminate in } Q_3(i) \right\}\,, \  \mathcal{Q}^{\mathrm{bad}}_{E}=\{Q(i)\colon i \in E \}\setminus\mathcal{Q}^{\mathrm{good}}_{E}\, .
\end{align}
Let us denote by
\begin{align} \label{def:bdrycubes}
\partial\mathcal{Q}:=\big\{Q(i)\colon i\in E, \BBB \mathcal{H}^2(\partial Q(i)\cap \partial{Q(E)})>0 \EEE \big\}
\end{align}
the collection of boundary cubes. For a cube $Q:=Q(i)\in \partial\mathcal{Q}$, let us denote by $\mathcal{T}_Q$ the (nonempty) collection of those faces of $Q$ that also belong to $\partial Q(E)$. For each $T\in \mathcal{T}_Q$, let $\nu_T\in \{\pm e_i\}_{i=1,2,3}$ be the constant outward pointing unit normal to $\partial Q(E)$ at the face $T$. For the set $E_{\sigma_0,t_0}$ we denote its outward unit normal at a point $x\in \partial E_{\sigma_0,t_0}$  by $\nu_{\sigma_0,t_0}(x)$.\\
\begin{step}{3.1}(Good cubes) Let $Q=Q(i) \in \mathcal{Q}^{\mathrm{good}}_{E}\cap \partial\mathcal{Q}$. Fix a boundary face  $T\in \mathcal{T}_Q$. Provided that $\sigma_0>0$ is small enough, we can check that $E_{\sigma_0,t_0}$ is also a coordinate laminate in $(T)_{2\sigma_0}$, (cf.\ \eqref{eq: thick-def}), in the sense that $\nu_{\sigma_0,t_0}(x)=\nu_T$ for every $x\in \partial E_{\sigma_0,t_0}\cap (T)_{2\sigma_0}$. Indeed, suppose without restriction (after possibly a rotation by a multiple of $\tfrac{\pi}{2}$ with respect to a coordinate axis and a reflection) that $\nu_T=e_3$. Then, since $\partial E_{\sigma_0,t_0}=\{f_{\sigma_0}=t_0\}$ and $\nabla f_{\sigma_0}$ is nowhere vanishing on $\partial E_{\sigma_0,t_0}$ (this is again a consequence of the choice of $\sigma_0$ and the Morse-Sard Lemma), for every $x\in \partial E_{\sigma_0,t_0}\cap(T)_{2\sigma_0}$ and for $l=1,2,3$, we have
\begin{align}\label{long}
\begin{split}
\nu_{\sigma_0,t_0}(x)\cdot e_l&=-\frac{\nabla f_{\sigma_0}}{|\nabla f_{\sigma_0}|}(x)\cdot e_l=-\frac{1}{|\nabla f_{\sigma_0}(x)|}\int_{\mathbb{R}^3}\chi_{Q(E)}(y)\nabla\zeta_{\sigma_0}(x-y)\cdot e_l\, \mathrm{d}y \\
&=-\frac{1}{|\nabla f_{\sigma_0}(x)|}\int_{{Q(E)}\cap B_{\sigma_0}(x)}\partial_l\zeta_{\sigma_0}(x-y)\, \mathrm{d}y 
\\&=\frac{1\EEE}{|\nabla f_{\sigma_0}(x)|}\int_{\partial(Q(E)\cap B_{\sigma_0}(x))}\zeta_{\sigma_0}(x-y)\nu_l(y)\, \mathrm{d}\mathcal{H}^2(y)\,,
\end{split}
\end{align}
where $\nu_l$ denotes the $l$-th component of the outer unit normal to $\partial(Q(E)\cap B_{\sigma_0}(x))$, which satisfies $\partial(Q(E)\cap B_{\sigma_0}(x)) =(\partial B_{\sigma_0}(x)\cap \overline{Q(E)} ) \cup (\partial(Q(E))\cap \overline{B_{\sigma_0}(x)})$. Due to \eqref{eq:propmolifier}, we have $\zeta_{\sigma_0}(x-y)=0$ on $\partial B_{\sigma_0}(x)$. On the other hand, as a consequence of the laminate structure of $E$ in  $Q_3(i)$, for $\sigma_0>0$ small enough, we get that $\partial(Q(E))\cap \overline{B_{\sigma_0}(x)}$ is either empty or is contained in the union of $T$ and the other 8 faces with outer normal $e_3$ of neighboring cubes to $Q(i)$ that share either an edge or a vertex with $T$. (Note that these neighboring cubes also belong to $Q(E)$.) Therefore, we have that $\nu_l=\delta_{3l}$ on $\partial(Q(E))\cap \overline{B_{\sigma_0}(x)}$, where $\delta_{3l}$ denotes the Kronecker delta. In view of \eqref{long},  this implies
 \begin{align}\label{eq:nue3}
\nu_{\sigma_0,t_0}(x) =\nu_T = e_3\,.
\end{align}
For good boundary cubes $Q  \in  \mathcal{Q}_E^{\mathrm{good}}\cap \partial\mathcal{Q}\EEE$, the laminate structure implies that  \EEE  $\# \mathcal{T}_Q\in \{1,2\}$ and if $\#\mathcal{T}_Q =2$, then two opposite faces are in $\mathcal{T}_Q$. Therefore, the outer unit normal vector to $\partial Q(E)\EEE$ of boundary faces is well-defined up to sign, and denoted  by $\nu_Q$. We define $$(Q)_{2\sigma_0}^{\mathrm{vert}}:=\{x+s\nu_Q\colon x\in Q, -2\sigma_0\leq s\leq 2\sigma_0\}\,.$$ Due to \eqref{eq:nue3}, the fact that $\nu_T = \pm \nu_Q$, and property (ii) established in Step 2\EEE, we have 
\begin{align*}
Q \in  \mathcal{Q}_E^{\mathrm{good}}\cap \partial\mathcal{Q}\EEE \quad  \implies \quad \partial (E_{\sigma_0,t_0})\cap(Q)_{2\sigma_0}^{\mathrm{vert}} =\bigcup_{T \in \mathcal{T}_Q} (T+s_T\nu_Q) \text{ for some } -\sigma_0<s_T< \sigma_0\,.
\end{align*}
Hence, we obtain 
\begin{align}\label{ineq:perimeterflatcubes}
\mathrm{Per}_{\varphi}(E_{\sigma_0,t_0},(Q)_{2\sigma_0}^{\mathrm{vert}}) = \mathrm{Per}_{\varphi}(Q(E),(Q)_{2\sigma_0}^{\mathrm{vert}}) \,,
\end{align}
while, by the flat laminate structure of $E_{\sigma_0,t_0}$ in $(Q)_{2\sigma_0}^{\mathrm{vert}}$, it is clear that
 \begin{align}\label{ineq:curvatureflatcubes}
\int_{\partial E_{\sigma_0,t_0}\cap(Q)_{2\sigma_0}^{\mathrm{vert}}}|\bm{A}|^q\, \mathrm{d}\mathcal{H}^2=0\,.
\end{align}	
Finally, by the choice of $\sigma_0>0$ and the definition of $\mathcal{Q}^\mathrm{good}_{E}$, see \eqref{def:good cubes_eta_scale}, we have 
\begin{align}\label{eq:emptyintersectionflat}
(Q)_{2\sigma_0}^{\mathrm{vert}} \cap (Q')_{2\sigma_0}^{\mathrm{vert}}=\emptyset\, \text{ for all } Q, Q' \in \mathcal{Q}^\mathrm{good}_E
 \,.
\end{align}
\end{step}\\
\begin{step}{3.2}(Bad cubes) Let $Q=Q(i) \in \mathcal{Q}^{\mathrm{bad}}_{E}$. As noted in \eqref{eq:localLambda}, there are at most $N_0$ different possibilities for the local shape of $E$. In particular, there exists a universal constant $C>0$ (only depending on $\sigma_0,t_0$, and these $N_0$ different configurations $(E^k)_{k=1,\ldots,N_0}$) so that for every $Q \in \mathcal{Q}^{\mathrm{bad}}_{E}$,
\begin{equation}\label{ineq:uniform_area_bound}
\mathcal{H}^{2}(\partial E_{\sigma_0,t_0}\cap Q_{\frac{3}{2}}(i))\leq C  \quad \quad  \text{ and }  \quad \quad  \int_{\partial E_{\sigma_0,t_0}\cap Q_{\frac{3}{2}}(i)}|\bm{A}|^q\, \mathrm{d}\mathcal{H}^2\leq C\,.
\end{equation}
We also note that $\partial E_{\sigma_0,t_0}$ is covered by the union of $\partial E_{\sigma_0,t_0} \cap \EEE (Q)_{2\sigma_0}^{\mathrm{vert}}$, $Q\in \mathcal{Q}^{\mathrm{good}}_{E}\cap \partial\mathcal{Q}\EEE$, together with $\partial E_{\sigma_0,t_0} \cap  \EEE Q_{\frac{3}{2}}(i)$, $Q(i)\in \mathcal{Q}^{\mathrm{bad}}_{E}$.
\end{step}
\end{step}\\
\begin{step}{4}($\eta_\varepsilon$-scale construction) We now use the previous steps in order to construct a regular set $\hat{E}$ from a set $E \subset \mathbb{Z}_\varepsilon(\Omega)$. By Step 1, we can assume that $E$ is an $\eta_\varepsilon$-scale set in the sense of \eqref{def: Eeta}. 
Now, let $E_1 := \eta_\varepsilon^{-1} 
E \EEE $, which is a finite subset of $\mathbb{Z}^3$. We then choose $(E_1)_{\sigma_0,t_0}$ as constructed in Step 2 and define $\hat{E}$ by $\hat{E}:= \eta_\varepsilon (E_1)_{\sigma_0,t_0} \cap U$. By standard properties of scaling of sets, 
we have $Q_\eps(E) = \eta_\eps Q(E_1)$ (recall \eqref{def:Aeps}), as well as \EEE  
\begin{itemize}
\item[(i)] $\partial \hat{E}\cap U\EEE$ is a regular $2$-dimensional manifold;
\item[(ii)] $\mathrm{dist}_{\mathcal{H}}(\hat{E},Q_{\eps}(E)) \leq \sigma_0\eta_\varepsilon$\,;
\item[(iii)] $\displaystyle\inf_{x\in\R^3\setminus\hat{E}} \mathrm{dist}(x,Q_\eps(E)) \geq c_0\eta_\varepsilon$\,.
\end{itemize} 
The last property directly implies \eqref{lemma:distance-cont}, and we also get $ Q_\varepsilon(E)\cap U\EEE\subset\hat E\subset U\EEE$. Now, denote by  
\begin{align}\label{eq: much not}
\begin{split}
\partial\mathcal{Q}_{\eta_\varepsilon} &:= \{ Q \colon Q=Q_{\eta_\varepsilon}(i), i \in \eta_\varepsilon\mathbb{Z}^3, \eta_\varepsilon^{-1}Q \in \partial\mathcal{Q}\}\,;\\
\mathcal{Q}^{\mathrm{good}}_{\eta_{\varepsilon}, \tilde U\EEE}
&:= \{Q \colon Q=Q_{\eta_\varepsilon}(i), i \in \eta_\varepsilon\mathbb{Z}^3, \eta_\varepsilon^{-1}Q \in \mathcal{Q}^\mathrm{good}_{E_1}, Q \cap \tilde U\EEE \neq \emptyset\}\,;\\
\mathcal{Q}^{\mathrm{bad}}_{\eta_{\varepsilon},\tilde U\EEE}
&:= \{Q \colon Q=Q_{\eta_\varepsilon}(i), i \in \eta_\varepsilon\mathbb{Z}^3, \eta_\varepsilon^{-1}Q \in \mathcal{Q}^\mathrm{bad}_{E_1}, Q \cap \tilde U\EEE \neq \emptyset\}\,,\\
\end{split}
\end{align}
where we have used the same notation as in \eqref{def:good cubes_eta_scale}, \EEE\eqref{def:bdrycubes} for the families $\mathcal{Q}^\mathrm{good}_{E_1}, \mathcal{Q}^\mathrm{bad}_{E_1}$, $\partial \mathcal{Q}$ (with respect to $E_1$), and for notational convenience we have suppressed the dependence on the original set $E$ in the notation for the families $\mathcal{Q}^{\mathrm{good}}_{\eta_{\varepsilon},\tilde U\EEE}
$ and $\mathcal{Q}^{\mathrm{bad}}_{\eta_{\varepsilon},\tilde U\EEE}
$. 
By \eqref{ineq:perimeterflatcubes}, \eqref{ineq:curvatureflatcubes}, and \eqref{eq:emptyintersectionflat}, we have
\begin{align}\label{eq:perimeterepsilon}
\mathrm{Per}_{\varphi}(\hat{E},(Q)_{2\eta_\varepsilon\sigma_0}^{\mathrm{vert}}) = \mathrm{Per}_{\varphi}(Q_{\varepsilon}(E),(Q)_{2\eta_\varepsilon\sigma_0}^{\mathrm{vert}}), \ \ \   \int_{\partial \hat{E}\cap(Q)_{2\eta_\varepsilon\sigma_0}^{\mathrm{vert}}}|\bm{A}|^q\, \mathrm{d}\mathcal{H}^2=0
\end{align}
for all \EEE $Q \in \mathcal{Q}^{\mathrm{good}}_{\eta_{\varepsilon},\tilde U\EEE}$, 
where \EEE we set $(Q)_{2\eta_\varepsilon\sigma_0}^{\mathrm{vert}} := \eta_\varepsilon (\eta_\varepsilon^{-1}Q)_{2\sigma_0}^{\mathrm{vert}}$,  and 
\begin{align}\label{eq:intersectionflatepsilon}
(Q)_{2\eta_\varepsilon\sigma_0}^{\mathrm{vert}} \cap (Q')_{2\eta_\varepsilon\sigma_0}^{\mathrm{vert}}=\emptyset\, \text{ for all } Q, Q' \in \mathcal{Q}^{\mathrm{good}}_{\eta_{\varepsilon},\tilde U\EEE}
\,.
\end{align}
On the other hand, due to \eqref{ineq:uniform_area_bound} and elementary scaling properties of $\mathcal{H}^2$ and the second fundamental form, for $Q:=Q_{\eta_{\varepsilon}}(i)\in  \mathcal{Q}^{\mathrm{bad}}_{\eta_{\varepsilon},\tilde U\EEE}
$, we have 
\begin{equation}\label{ineq:uniform_area_bound_eta}
\mathcal{H}^{2}(\partial \hat{E}\cap Q_{
\frac{3}{2}\eta_\varepsilon}\EEE)\leq C\eta_\varepsilon^2 \quad \quad  \text{ and } \quad \quad   \int_{\partial \hat{E}\cap Q_{
\frac{3}{2}\eta_\varepsilon}\EEE}|\bm{A}|^q\, \mathrm{d}\mathcal{H}^2\leq C\eta_\varepsilon^{2-q}\,.
\end{equation}
Finally, by the simple observation we made at the end of Step 3 and by scaling, we get 
\begin{align}\label{incl:boundarylambdahat}
\partial\hat E \cap \tilde U\EEE \subset \bigcup_{Q \in \mathcal{Q}^{\mathrm{good}}_{\eta_{\varepsilon},\tilde U\EEE}
\cap \partial\Q_{\eta_\eps}} (\partial\hat E \cap(Q)_{2\eta_\varepsilon\sigma_0}^{\mathrm{vert}}) \cup \bigcup_{Q \in \mathcal{Q}^{\mathrm{bad}}_{\eta_{\varepsilon},\tilde U\EEE}
}  (\partial\hat E \cap Q_{\frac{3}{2}\eta_\varepsilon})\,.
\end{align}
\\
\begin{step}{4.1}(Proof of \eqref{lemma:perimeterdis-cont}) By \eqref{incl:boundarylambdahat} we can estimate
\begin{equation}\label{ineq:perimeter_estimate_smoothened_set}
\mathrm{Per}_{\varphi}(\hat{E},\tilde U\EEE)\leq \sum_{Q\in \mathcal{Q}^{\mathrm{good}}_{\eta_{\varepsilon},\tilde U\EEE}
\cap \partial\Q_{\eta_\eps}}\int_{\partial \hat{ E}\cap(Q)_{2\eta_\varepsilon\sigma_0}^{\mathrm{vert}}}\varphi(\nu)\, \mathrm{d}\mathcal{H}^2+ \sum_{Q\in \mathcal{Q}^{\mathrm{bad}}_{\eta_{\varepsilon},\tilde U\EEE}
}\int_{\partial \hat{E}\cap Q_{\frac{3}{2}\eta_\varepsilon}}\varphi(\nu)\, \mathrm{d}\mathcal{H}^2\,.
\end{equation}
We start with the good cubes. We use the same notation $\mathcal{T}_Q$ and $\nu_T$ as before in Step 3.1. \EEE Using \eqref{def: discrete perimeter}, \eqref{def:varphi}, \eqref{eq:perimeterepsilon}, and \eqref{eq:intersectionflatepsilon}, we have
\begin{align}\label{surface_equality_flat_part}
\begin{split}
\sum_{Q\in \mathcal{Q}^{\mathrm{good}}_{\eta_{\varepsilon},\tilde U\EEE}
\cap \partial\Q_{\eta_\eps}}\int_{\partial \hat{ E}\cap(Q)_{2\eta_\varepsilon\sigma_0}^{\mathrm{vert}}}\varphi(\nu)\, \mathrm{d}\mathcal{H}^2&= \sum_{Q\in \mathcal{Q}^{\mathrm{good}}_{\eta_{\varepsilon},\tilde U\EEE}
\cap \partial\Q_{\eta_\eps}}\int_{\partial Q_\eps(E)\cap(Q)_{2\eta_\varepsilon\sigma_0}^{\mathrm{vert}}}\varphi(\nu)\, \mathrm{d}\mathcal{H}^2\\&= \sum_{Q\in \mathcal{Q}^{\mathrm{good}}_{\eta_{\varepsilon},\tilde U\EEE}
\cap \partial\Q_{\eta_\eps}}\sum_{T\in\mathcal{T}_Q}\frac{\eta_{\varepsilon}^2}{\varepsilon^2}\underset{\xi\cdot \nu_T >0}{\sum_{\xi \in V}} \varepsilon^2c_\xi  \leq F_\varepsilon^{\mathrm{per}}(E)\,,
\end{split}
\end{align}
where we also used \eqref{locality of the argument} and the fact that  $\eta_{\varepsilon}/\varepsilon\in \mathbb{N}$, see \eqref{eq: ratio}. Now, let $Q \in \mathcal{Q}^{\mathrm{bad}}_{\eta_\varepsilon,\tilde U}$. By the negation of Lemma \ref{lemma : laminate} there exists $k\in \varepsilon\mathbb{Z}^3$ such that for $Q_{2\varepsilon}(\hat{k}) \in \hat{\mathcal{Q}}_{\eps,2}$,  
$Q_{3\eta_{\varepsilon}}(i)\cap Q_{2\varepsilon}(\hat k)\neq \emptyset$ 
and 
$E\cap Q_{2\varepsilon}(\hat{k})$ does not satisfy  any of the properties (i)-(iii) of Definition  \ref{locally_flat_configurations}. By \eqref{eq: ratio}, we actually have that $k\in \mathbb{Z}_{\varepsilon}\big(Q_{\frac{7}{2}\eta_{\varepsilon}}(i)\big)$, provided that we take $\varepsilon_0>0$ small enough.

Arguing as in the proof of Lemma \ref{lemma : eta lattice}, for all $j\in \mathbb{Z}_{\varepsilon}\big(Q_{\frac{\eta_{\varepsilon}}{4}}(k)\big)$ we have that $Q_{2\varepsilon}(\hat{k}) \subset Q_{\frac{\eta_{\varepsilon}}{2}}(j)$ and therefore, due to property (ii) below \eqref{def: curvature energy}, we get
\begin{align*}
W_{\varepsilon,\mathrm{cell}}^\mathrm{curv}(j,E) \geq c\gamma_{\varepsilon} \eta_{\varepsilon}^{-1-q}\,.
\end{align*}
Notice that, since $j\in \mathbb{Z}_{\varepsilon}\big(Q_{\frac{\eta_{\varepsilon}}{4}}(k)\big)$ and $k\in \mathbb{Z}_{\varepsilon}\big(Q_{\frac{7}{2}\eta_{\varepsilon}}(i)\big)$, we get that $j\in \mathbb{Z}_{\varepsilon}\big(Q_{4\eta_{\varepsilon}}(i)\big)$. This, together with the fact that \EEE $\# \mathbb{Z}_\varepsilon(Q_{\frac{\eta_{\varepsilon}}{4}}(k)) \geq C\frac{\eta^3_{\varepsilon}}{\varepsilon^3}$ for an absolute constant $C>0$, leads to \begin{align}\label{ineq:curv2}
F_\varepsilon^\mathrm{curv}\big(E,Q_{4 \eta_{\varepsilon}}(i)\big) & \geq \sum_{j \in \mathbb{Z}_\varepsilon(Q_{\frac{\eta_{\varepsilon}}{4}}(k))} \varepsilon^3 W_{\varepsilon,\mathrm{cell}}^\mathrm{curv}(j,E) \geq \# \mathbb{Z}_\varepsilon(Q_{\frac{\eta_{\varepsilon}}{4}}(k))\,\varepsilon^3\cdot c\gamma_{\varepsilon}\eta_{\varepsilon}^{-1-q} \geq c\gamma_{\varepsilon}\eta_{\varepsilon}^{2-q}\,.
\end{align}
In view of \eqref{ineq:uniform_area_bound_eta} and \eqref{ineq:curv2}, we can simply estimate 
\begin{align}\label{surface_inequality_non_flat_part}
\begin{split}
\sum_{Q\in \mathcal{Q}^{\rm{bad}}_{\eta_\eps,\tilde U\EEE}}
\int_{\partial \hat{E}\cap Q_{\frac{3}{2}\eta_\varepsilon}}\varphi(\nu)\, \mathrm{d}\mathcal{H}^2&\leq \varphi_{\mathrm{max}} \sum_{Q\in \mathcal{Q}^{\rm{bad}}_{\eta_\eps,\tilde U\EEE}}
\mathcal{H}^{2}(\partial \hat{E}\cap Q_{\frac{3}{2}\eta_\varepsilon})\\&
\leq C\eta_{\varepsilon}^2\,\# \mathcal{Q}^{\rm{bad}}_{\eta_\eps,\tilde U\EEE}
\leq C\eta_{\varepsilon}^2\cdot\gamma_{\varepsilon}^{-1} \eta_{\varepsilon}^{-2+q}  \sum_{Q \in \mathcal{Q}^{\rm{bad}}_{\eta_\eps,\tilde U\EEE}}
 F_\varepsilon^\mathrm{curv}(E,Q_{4\eta_{\varepsilon}}(i)) \\
&\leq C\gamma_{\varepsilon}^{-1}\eta_{\varepsilon}^qF_\varepsilon^\mathrm{curv}(E)\,,
\end{split}
\end{align}
where we again used \eqref{ineq:Qintersectbound} and \eqref{locality of the argument}, and for brevity we set $\varphi_{\mathrm{max}} :=\max_{\mathbb{S}^{2}}\varphi$. 
Here,  $C:=C(\{c_\xi\}_{\xi \in V})>0$, \EEE since $\varphi$ depends on $\{c_\xi\}_{\xi\in V}$.  Now, \eqref{ineq:perimeter_estimate_smoothened_set}, \eqref{surface_equality_flat_part}, and \eqref{surface_inequality_non_flat_part} imply \eqref{lemma:perimeterdis-cont} for $\omega(\varepsilon):=C\gamma_{\varepsilon}^{-1}\eta_{\varepsilon}^q\to 0$ as $\varepsilon\to 0$ (see \eqref{eq:rate gammadelta}).
\end{step}\\
\begin{step}{4.2}(Proof of \eqref{lemma:curvaturedis-cont}) By using \eqref{eq:perimeterepsilon}, \eqref{ineq:uniform_area_bound_eta}, \eqref{incl:boundarylambdahat}, and \eqref{ineq:curv2} we can estimate 
\begin{align*}
\begin{split}
\gamma_{\varepsilon}\int_{\partial \hat {E}\cap\tilde U\EEE} |\bm{A}|^q\,\mathrm{d}\mathcal{H}^{2}& \leq  \gamma_{\varepsilon}  \sum_{Q\in \mathcal{Q}^{\rm{bad}}_{\eta_\eps,\tilde U\EEE}}
\int_{\partial \hat {E}\cap Q_{\frac{3}{2}\eta_\varepsilon}}|\bm{A}|^q\, \mathrm{d}\mathcal{H}^2\leq C\gamma_{\varepsilon}\eta_{\varepsilon}^{2-q}\# \mathcal{Q}^{\rm{bad}}_{\eta_\eps,\tilde U\EEE}
\\
&\leq C\gamma_{\varepsilon} \eta_\eps^{2-q}\gamma_{\varepsilon}^{-1}\eta_{\varepsilon}^{-2+q} \sum_{Q\in\mathcal{Q}^{\rm{bad}}_{\eta_\eps,\tilde U\EEE}}
F_\varepsilon^\mathrm{curv}(E,Q_{4\eta_{\varepsilon}}(i))\leq CF_\varepsilon^{\mathrm{curv}}(E)\,,
\end{split}
\end{align*}
which proves \eqref{lemma:curvaturedis-cont}, where the last step follows as in \eqref{surface_inequality_non_flat_part}.
\end{step}\\
\begin{step}{4.3}(Proof of \eqref{lemma:measuredis-cont}) By the construction of the set $\hat {E}$ we have that 
\begin{align*}
\left(\hat{E} \setminus Q_\varepsilon(E)\right) \cap \tilde U\EEE \subset \bigcup_{Q \in \partial\mathcal{Q}_{\eta_\varepsilon, \tilde U\EEE}}Q_{3\eta_\varepsilon}\,,
\end{align*}
where we recall the notation in \eqref{eq: much not}, and set $\partial\mathcal{Q}_{\eta_\varepsilon,\tilde U\EEE}:= \lbrace Q \in \partial\mathcal{Q}_{\eta_\varepsilon}\colon Q\cap \tilde U\EEE \neq \emptyset\rbrace$. 
By the fact that we have reduced to the case that $E$ is an $\eta_\varepsilon$-scale set in the sense of \eqref{def: Eeta}, we have $F_\varepsilon^\mathrm{per}(E, Q_{3\eta_\varepsilon})\geq c\eta_\varepsilon^2$  for all $Q \in \partial\mathcal{Q}_{\eta_\varepsilon,\tilde U\EEE}$ 
for some \EEE constant $c>0$ depending on $\{c_\xi\}_{\xi \in V}$. Therefore, by \eqref{ineq:Qintersectbound} and \eqref{locality of the argument} we obtain 
\begin{align*}
\mathcal{L}^3\big((\hat {E}\setminus Q_\varepsilon(E))\cap\tilde U\EEE\big) \leq C\eta^3_\varepsilon \, \#\partial\mathcal{Q}_{\eta_\varepsilon,\tilde U\EEE} &\leq C c^{-1} \eta_\varepsilon \sum_{Q\in\partial\mathcal{Q}_{\eta_\varepsilon,\tilde U\EEE}}F_\varepsilon^\mathrm{per}(E,Q_{3\eta_\varepsilon})\leq  C\eta_\varepsilon F_\varepsilon^\mathrm{per}(E)\,,
\end{align*}
where $C>0$ only depends on $\{c_\xi\}_{\xi \in V}$. This finishes the proof of \eqref{lemma:measuredis-cont}. 
\end{step}
\end{step}
\end{proof}
We now have all the necessary ingredients to prove the first part of  our compactness theorem.  The second part on the validity of the \textit{Cauchy-Born rule for discrete symmetric gradients}, i.e., \eqref{symmetric_Cauchy-Born}, will be postponed to the next section.

\begin{proof}[Proof of Theorem \ref{theorem:compactness}\,{\rm (i)}] 
Having fixed the boundary condition $u_0\in W^{1,\infty}(\R^3;\R^3)$, let us consider  $(u_\varepsilon,E_\varepsilon)_{\varepsilon>0}$, $E_\varepsilon\subset\mathbb{Z}_{\varepsilon}(\Omega)$ and $u_\varepsilon \in \mathcal{U}_\varepsilon(u_0,\partial_D\Omega,E_\eps)$, such that
\begin{align}\label{ineq:equiboundedness}
\sup_{\varepsilon>0} \mathcal{F}_\varepsilon(u_\varepsilon,E_\varepsilon) <+\infty\,.
\end{align}
In particular, $u_\varepsilon=u_0$ in $\Z_\eps(U\setminus\Omega)$ (cf.\ Definition \ref{boundary values}). 
Consider also a sequence of smooth open \EEE sets $(U_m)_{m\in \mathbb{N}}$, with $U_m\EEE\subset\subset U$ and \EEE
\begin{equation}\label{compact_exhaustion}
\lim_{m\to\infty}\mathcal{L}^3(U\setminus
U_m)=0, \quad \sup_{m\in \mathbb{N}}\mathcal{H}^2(\partial
U_m)<+\infty\,.
\end{equation}
Let $\hat{E}_\varepsilon\in \mathcal{A}_{\mathrm{reg}}(U)$ be the smooth set obtained from $E_\varepsilon$ through Lemma \ref{lemma:eta scale smoothening}, applied for $(U,U_m)$ in the place of $(U,\tilde U\EEE)$ here. This lemma guarantees that $Q_\varepsilon(E_{\varepsilon})\cap U\subset\hat {E}_\varepsilon\subset U$, and that we can find a decreasing sequence $(\varepsilon_m)_{m\in \mathbb{N}}$, $\varepsilon_m:=\varepsilon (U_m,U)\in (0,1)$ such that for every $m\in \mathbb{N}$ and every $\varepsilon\in(0,\varepsilon_m]$, 
\begin{align}\label{disc_cont_est_combined}
\begin{split}
\rm{(i)} &\quad \mathrm{Per}_{\varphi}(\hat {E}_\varepsilon,U_m) 
\leq F_\varepsilon^\mathrm{per}(E_\varepsilon) + \omega(\varepsilon) F_\varepsilon^\mathrm{curv}(E_\varepsilon)\EEE\,,\\
\rm{(ii)} &\quad \hat{\gamma}_{\varepsilon} \int_{\partial\hat {E}_\varepsilon\cap U_m} |\bm{A}|^q\,\mathrm{d}\mathcal{H}^{2} 
\leq F_\varepsilon^\mathrm{curv}(E_{\varepsilon})\,\EEE,\\
\rm{(iii)} &\quad \mathcal{L}^3\big((\hat {E}_\varepsilon\setminus Q_\varepsilon(E_{\varepsilon}))\cap U_m\big)\leq C\eta_{\varepsilon} \left(F_\varepsilon^\mathrm{per}(E_{\varepsilon}) + F_\varepsilon^\mathrm{curv}(E_{\varepsilon})\EEE\right)\EEE\,, \\
\rm{(iv)} &\quad \dist(U_m \setminus \hat{E}_\eps, Q_\eps(E_\eps) \cap U) \ge c_0\eta_\eps\,,
\end{split}
\end{align}
where $\omega(\varepsilon)\to 0$ as $\varepsilon\to 0$ and $\hat{\gamma}_{\varepsilon} := \gamma_\eps/C$ for a sufficiently large constant $C>0$, depending only on the coefficients $\{c_\xi\}_{\xi\in V}$. We also \EEE recall that $F_\eps^{\rm{per}}(E)\,, F_\eps^{\rm{curv}}(E)$ are given by \eqref{def: discrete perimeter} and \eqref{def: curvature energy} respectively.  

\BBB By $\tilde {u}_\eps$ we denote the piecewise affine interpolation of $u_\eps$ with respect to a standard triangulation subordinate to the lattice $\varepsilon\mathbb{Z}^3$, see also Subsection 3.2 in \cite{Schmidt:2009} for more details on the definition of the interpolation. By construction, it is clear that $\tilde {u}_\eps|_{U\setminus \hat{E}_\eps}\in H^1(U\setminus \hat{E}_\eps; \R^3)$. By \eqref{cubic_set_of_A}, \eqref{eq: ratio}, and \eqref{disc_cont_est_combined}(iv) we find that \EEE $Q \subset U \setminus Q_\eps(E_\eps)$ for every $Q \in \hat\Q_{\eps,U_m\setminus \hat{E}_\eps}$. Moreover, for each such \EEE $Q$, we get 
$$\fint_{Q} {\rm dist}^2({\rm Id}+ \delta_\eps \nabla\tilde u_\eps, SO(3))\, \mathrm{d}x \le C {\rm dist}^2(Z + \delta_\eps \overline{\nabla}u_\eps|_{Q}, \bar{SO}(3))$$
due to \cite[Lemma \EEE 3.6(i)]{Schmidt:2009}. Thus, \EEE  by our structural assumptions on the discrete elastic energy (cf.\ Subsection \hyperref[Assumptions on the elastic energy]{2.2}), which are the same ones as in \cite{Schmidt:2009}, by \cite[Lemma 3.7]{Schmidt:2009}, as well as \eqref{ineq:equiboundedness} and \eqref{disc_cont_est_combined}, for every $m\in \mathbb{N}$ we deduce that 
\begin{equation}\label{continuum_energy_bounded}
\sup_{\varepsilon\in (0,\varepsilon_m]}\left(\delta_\varepsilon^{-2} \int_{
U_m\setminus \hat{E}_{\varepsilon}}\mathrm{dist}^2(\mathrm{Id}+\delta_{\varepsilon}\nabla\tilde{u}_{\varepsilon}, SO(3))
+ \mathrm{Per}_\varphi(\hat{E}_{\varepsilon},
U_m) + \hat{\gamma}_{\varepsilon} \int_{\partial\hat{E}_{\varepsilon}\cap
U_m} |\bm{A}|^q\right) \le {C}\,, \EEE
\end{equation}
where the constant $C>0$ is independent of $m$. 

Regarding the boundary conditions, we observe the following: let $\tilde u_{0,\eps}$ be the piecewise affine interpolation of the boundary datum $u_0\in W^{1,\infty}(\R^3;\R^3)$ with respect to the same triangulation of the lattice $\eps\Z^3$. Then \EEE $\tilde u_{0,\eps}\to u_0$ locally uniformly on $\R^3$ as $\eps\to 0$. The fact that $\tilde u_\eps\equiv \tilde u_{0,\eps} \text{ on } 
\{x\in U\setminus \Omega\colon \mathrm{dist}(x,\partial(U\setminus \Omega))\geq \sqrt{3}\eps\}\,,$ along with \eqref{eq:rate gammadelta}, that $\eta_\varepsilon\to 0$ as $\varepsilon\to0$, \cite[Remark 3.10]{KFZ:2021}, \eqref{continuum_energy_bounded}, and \eqref{compact_exhaustion}, \EEE shows that we can apply the compactness result \cite[Proposition 3.1]{KFZ:2021} (with $\hat{\gamma}_\eps$ in place of $\gamma_\delta$, and ${U}_m$ in place of $\Omega$) and a diagonal argument in order to extract a (non-relabeled) subsequence $(u_\varepsilon,E_\varepsilon)_{\varepsilon>0}$ for which the following holds:

There exists $u \in GSBD^2(U)$,  sets of finite perimeter $E \in \M(U)$,  $(E_\varepsilon^\ast)_{\varepsilon>0} \subset \M(\R^3)$ with  $\hat{E}_\varepsilon\subset E_\varepsilon^\ast$, as well as sets (of finite perimeter) $\hat{\omega}_u , (\hat{\omega}_u^\varepsilon)_{\varepsilon>0} \subset \M(U)$ with $\H^2(\partial^* \hat\omega_u) + \sup_{\eps>0}\H^2(\partial^* \hat\omega_u^\eps) \le C$ \EEE such that $u\equiv 0$ on $E \cup \hat{\omega}_u$, and as $\varepsilon\to 0$ we have  
\begin{align}\label{eq:lsc0}
\begin{split}
{\rm (i)} & \ \ \tilde{u}_\varepsilon \to u \text{ in measure on $U\setminus \hat{\omega}_u$}\,;\\
{\rm (ii)} &  \ \ \chi_{U \setminus (E^\ast_\varepsilon \cup \hat{\omega}^\varepsilon_u)}e(\tilde{u}_\varepsilon) \rightharpoonup \chi_{U \setminus (E \cup \hat{\omega}_u)} e(u)  \ \ \text{weakly in $L^2_{\rm loc}(U;\R^{3\times 3}_{\rm sym})$}\,;\\
{\rm (iii)} & \ \ \mathcal{L}^3(\lbrace |\nabla\tilde{u}_\varepsilon|>\kappa_{\varepsilon} \rbrace \setminus \hat{\omega}_u)  \to 0\,;\\
{\rm (iv)} & \ \   \liminf_{\varepsilon \to 0}   \int_{\partial{E_{\varepsilon}^{\ast}}\cap U} \varphi(\nu_{E^\ast_{\varepsilon}}) \, {\rm d}\mathcal{H}^{2} \le  \liminf_{\varepsilon \to 0} \mathcal{F}_{\mathrm{surf}}^{\varphi,\hat{\gamma}_{\varepsilon},q}(\hat{E}_{\varepsilon\EEE},U)\,;\\
{\rm (v)} & \ \ \lim_{\varepsilon \to 0}\mathcal{L}^3(\hat{E}_\eps \triangle E) = \lim_{\varepsilon \to 0 }\mathcal{L}^3(\hat{\omega}_u^\varepsilon \triangle\hat{\omega}_u) = \lim_{\varepsilon \to 0} \mathcal{L}^3(E_\varepsilon^\ast \setminus \hat{E}_\varepsilon) = 0 \,,  
\end{split}
\end{align}
where, as in \cite[Equation (3.5)]{KFZ:2021}, $(\kappa_{\varepsilon})_{\varepsilon>0}$ is such that 
\begin{equation}\label{rate_kappa_epsilon}
\delta_\varepsilon \kappa_{\varepsilon}^3 \to 0, \ \gamma_{\varepsilon}^{3/q}\kappa_{\varepsilon} \to \infty\, \ \text{as}\ \  \varepsilon\to 0\,.
\end{equation}
For \eqref{eq:lsc0}(iv), we have also  used the notation 
$$ \mathcal{F}_{\mathrm{surf}}^{\varphi,\hat{\gamma}_{\varepsilon},q}(\hat{E}_{\varepsilon},U)= \mathrm{Per}_\varphi(\hat{E}_{\varepsilon},U) + \hat{\gamma}_{\varepsilon}  \int_{\partial\hat{E}_{\varepsilon}\cap U} |\bm{A}|^q\,,$$
see  \cite[Equation  (2.1)]{KFZ:2021}. In particular, 
the first two \EEE estimates in \eqref{disc_cont_est_combined} imply that 
\begin{align}\label{eq: fflat}
\liminf_{\varepsilon \to 0}  \mathcal{F}_{\mathrm{surf}}^{\varphi,\hat{\gamma}_{\varepsilon},q}(\hat{E}_{\varepsilon},U) 
& \le \liminf_{\varepsilon \to 0}\big( F_\varepsilon^{\mathrm{per}}(E_\varepsilon)+F_\varepsilon^{\mathrm{curv}}(E_\varepsilon) \big)\,.
\end{align}
Let $\bar{u}_\varepsilon$ be the interpolation of $u_\eps$ being constant on cubes of sidelength $\eps$ in the shifted lattice $\eps(\Z^3+(\tfrac{1}{2},\tfrac{1}{2},\tfrac{1}{2}))$ \EEE (see the comments before Definition~\ref{def:convergence}). One can directly check that the measure convergence of $\tilde{u}_\varepsilon$ and $\bar{u}_\varepsilon$ are equivalent, as commented below \cite[Definition 2.5]{Schmidt:2009}. 
For instance, \EEE this can be done using \cite[Proposition A.1 and Remark A.2]{AliFocGel}, together with the fact that for every $i\in\eps\Z^3$ the piecewise affine interpolation $\tilde u_\eps$ is constructed in such a way that
$$\tilde u_\eps(\hat i)=\bar u_\eps(\hat i)=\fint_{Q_\eps(\hat i)}\tilde u_\eps\, \mathrm{d}x\,,$$
cf. \cite[Equation (6)]{Schmidt:2009}.  \EEE  Thus, \eqref{eq:lsc0}(i) implies that $\bar{u}_\varepsilon \to u$ in measure on $\Omega \setminus \omega_u$, where $\omega_u:=\hat{\omega}_u\cap\Omega$.  
Moreover, by the fact $E_{\varepsilon}\subset\mathbb{Z}_{\varepsilon}(\Omega)$, 
\eqref{disc_cont_est_combined}(iii), \EEE and the $L^1(U)$-convergence of $\hat{E}_{\varepsilon}$ to $E$, we indeed obtain  $E \subset \Omega$ and $\chi_{Q_\varepsilon(E_\varepsilon)\cap\Omega}\to\chi_E$ in $L^1(\Omega)$. This finishes the proof of the first part of the theorem.  
\end{proof} 

\section{The validity of the Cauchy-Born rule for discrete symmetric gradients}\label{sec: 4}
In view of our compactness result \cite[Proposition 3.1]{KFZ:2021}, in particular \eqref{eq:lsc0} above, the aim of this section is to show that, loosely speaking, weak $L^2$-convergence of the symmetric gradients of the piecewise affine interpolations of $(u_\eps)_{\eps>0}$ (outside the void set) implies the weak $L^2$-convergence of the discrete symmetric gradients $(\bar e(u_\eps))_{\eps>0}$ (outside a slightly modified void set) towards the same continuum \EEE limit. As already mentioned, since we believe that this result can be useful also in other discrete-to-continuum problems for elastic materials with surface discontinuities, we phrase the statement in any dimension. 

For its formulation, we adopt all the notations introduced in the last subsection of the Introduction, as well as Subsection~\ref{subsection 2.1}. Moreover, as before, for a discrete map $u\colon\eps\Z^d\to \R^d$ we denote by $\tilde u$ its piecewise affine interpolation according to a triangulation subordinate to the lattice $\eps\Z^d$, constructed as in \cite[Section~3.2]{Schmidt:2009}. The collection of $d$-dimensional simplices used in this construction will be denoted by $\mathcal{S}_\varepsilon$, and we also define
\begin{align}\label{def:simplices}
\mathcal{S}_{\varepsilon,Q}:=\{S \in \mathcal{S}_\eps \colon \, S \subset Q \rbrace \ \ \ \text{for all }  Q\in \hat\Q_\eps\,,
\end{align}
where we recall the notation in and below \eqref{cubic_set_of_A}.  Moreover, note  that by  construction all $d$-simplices in $\mathcal{S}_\eps$  have the same $\L^d$-measure. 

\begin{theorem}[Cauchy-Born rule\EEE]\label{discrete_Cauchy_Born_any_d}
Let $d\in \N$, $d\geq 2$, and $\Omega\subset \R^d$ be a bounded open \EEE Lipschitz set\EEE. Consider a sequence of discrete displacements $(
u_\eps)_{\eps>0}$, $u_\eps\colon \eps\Z^d
\mapsto \R^d$\,,
for which the following holds true. 
There exists $u\in GSBD^2(\Omega)$, a set of finite perimeter $W
\in \mathfrak{M}(\Omega)$ with $u\equiv 0$ on $W,$
as well as sets of finite perimeter $(W_\eps)_{\eps>0}\subset \mathfrak{M}(\R^d)$ such that 
\begin{enumerate}[label=\textrm{{\rm (H.\arabic*)}}, leftmargin=3.2cm]
\item\label{H1}
$\lim_{\varepsilon \to 0 }\mathcal{L}^d(W_\eps\triangle W
)=0\,, \quad \sup_{\eps>0}\H^{d-1}(\partial^* W_\eps
)  <+\infty\,;$\\[-5pt]
\item
\label{H2}
$\tilde{u}_\varepsilon \to u \text{ in measure on } \Omega\setminus W
\,;$\\[-5pt]
\item
\label{H3}
$\chi_{\Omega\setminus W_\eps
}e(\tilde{u}_\varepsilon) \rightharpoonup \chi_{\Omega\setminus W
} e(u)  \ \ \text{weakly in }L^2_{\rm loc}(\Omega;\R^{d\times d}_{\rm sym})\,.$
\end{enumerate}
Then, there exist sets of finite perimeter $(V_\eps)_{\eps>0}\subset \mathfrak{M}(\R^d)$, with $W_\eps\subset V_\eps$, such that  
\begin{align}\label{conclusions_of_symmetric_Cauchy-Born}
\begin{split} 
{\rm (i)}&\quad \mathcal{L}^d(Q \setminus  V_\eps) = 0 \text{ or }   \mathcal{L}^d(S \setminus  V_\eps) \ge  \frac{1}{2}  \mathcal{L}^d(S) \ \forall \, S \in \mathcal{S}_{\varepsilon,Q}\,, \forall \, Q\in  \hat\Q_\eps\,;\\ 
{\rm (ii)}&\quad \lim_{\varepsilon \to 0 }\mathcal{L}^d(V_\eps\setminus W_\eps \EEE)=0\,, \  \sup\nolimits_{\eps>0}\H^{d-1}(\partial^* V_\eps
)<+\infty\,;\\
{\rm (iii)}&\quad \chi_{\Omega\setminus V_\eps}\bar{e}(u_\varepsilon) \rightharpoonup \chi_{\Omega \setminus W
} e(u)Z  \ \ \text{weakly in }L^2_{\rm loc}(\Omega;\R^{d\times 2^d})\,,
\end{split}
\end{align}
where $Z$ is defined in \eqref{Z_points}.  
\end{theorem}




The second part of the compactness result is a direct consequence of Theorem \ref{discrete_Cauchy_Born_any_d}. 
\begin{proof}[Proof of Theorem \ref{theorem:compactness}\,{\rm (ii)}] Having in mind \eqref{disc_cont_est_combined} and \eqref{eq:lsc0}, and the definition $\omega_u := \hat{\omega}_u \cap \Omega$, we consider the sets 
$(W_\eps)_{\eps>0}$ and $W$ given by 
\begin{align}\label{eq: tildeVdef}
W_\eps:=(E_\eps^*\cup \hat\omega_u^\eps 
\cup \hat{\omega}_u\EEE)\cap \Omega\,,\quad W:=E\cup \omega_u\,.
\end{align} 
The first condition in \textit{\ref{H1}} \EEE is clearly satisfied. The second one 
follows by the construction of the sets. Indeed,  by $\min_{\S^2}\varphi>0$, \eqref{eq:lsc0}(iv), \eqref{eq: fflat}, and the fact that $\sup_{\eps>0}\mathcal{F}_\eps(u_\eps,E_\eps)<+\infty$, we get $\sup_{\eps>0}\H^2(\partial E_\eps^\ast\cap \Omega)<+\infty$. Moreover, we have $\sup_{\eps>0}\H^2(\partial^* \hat\omega_u^\eps)<+\infty$, see before \eqref{eq:lsc0}.  
Conditions \textit{\ref{H2}} and \textit{\ref{H3}} are also satisfied because of \eqref{eq:lsc0}(i) and (ii), respectively. Then, we can  apply Theorem \ref{discrete_Cauchy_Born_any_d} (for $d=3$) to obtain sets of finite perimeter $(V_\eps)_{\eps>0}\subset \mathfrak{M}(\R^3)$ such that the desired properties in \eqref{ineq:small difference in measures} and \eqref{symmetric_Cauchy-Born} are satisfied, where we also use \eqref{disc_cont_est_combined}(iv), \eqref{eq:lsc0}(v), and \eqref{conclusions_of_symmetric_Cauchy-Born}(ii),(iii). 
\end{proof}

The main point is that \textit{\ref{H3}} 
implies \eqref{conclusions_of_symmetric_Cauchy-Born}{(iii)}, up to a modification of the void set which is negligible in the limit (see \eqref{conclusions_of_symmetric_Cauchy-Born}{(ii)}). A property of this form has been established by {\sc Schmidt} in \cite[Theorem~2.6]{Schmidt:2009}: given discrete displacements $u_\eps \colon \eps \Z^d \EEE \to \R^d$ and corresponding piecewise affine interpolations $\tilde{u}_\eps$ such that $\tilde{u}_\eps \rightharpoonup u$ weakly in $H^1(\Omega;\R^d)$ for some $u \in H^1(\Omega;\R^d)$, then it holds that
\begin{align}\label{eq: bernd}
\overline{\nabla} u_\eps \rightharpoonup \nabla u Z \text{ weakly in $L^2_{\rm loc}(\Omega;\R^{d \times 2^d} \EEE )$}\,.
\end{align}
The proof particularly uses integration by parts for Sobolev functions and cannot be reproduced in our setting with voids. Therefore, a delicate blow-up argument and an approximation of $GSBD^2$ functions by Sobolev functions needs to be employed to reduce our problem to the setting in \eqref{eq: bernd}  for the corresponding symmetric gradients. The remainder of this section is devoted to the proof of Theorem \ref{discrete_Cauchy_Born_any_d}. The reader is invited to skip the following subsections on first reading and to proceed directly with the lower bound in Section~\ref{sec: 5}.

\subsection{Proof of Theorem \ref{discrete_Cauchy_Born_any_d}}

We give an outline of the  main ideas of the proof, stating the key intermediate results as preliminary lemmata. The proofs of these lemmata are postponed to Subsection \ref{sec: int lemmata}.  We start by showing     the existence of a weak $L^2$-limit for the discrete symmetric gradients.


\begin{lemma}[Compactness for discrete symmetric gradients]\label{weak_limit_for_discrete_multies}
Let $d\in \N$, $d\geq 2$, let $\Omega\subset \R^d$ be a bounded open \EEE Lipschitz set\EEE, and let $(u_\eps)
_{\eps>0}$ and \EEE $(W_\eps)_{\eps>0}$ satisfy all hypotheses of Theorem~\ref{discrete_Cauchy_Born_any_d}. There exist sets of finite perimeter $(V_\eps)_{\eps>0}\subset \mathfrak{M}(\R^d)$ with $
W_\eps \subset V_\eps$ satisfying \eqref{conclusions_of_symmetric_Cauchy-Born}{\rm (i),(ii)}, and for each bounded open \EEE Lipschitz subset \EEE $\tilde\Omega\subset\subset\Omega$ there exists $\eps_0:=\eps_0(\tilde\Omega,\Omega\EEE)\in (0,1)$ such that
\begin{equation}\label{uniform_L^2_bound_symmetric gradients}
\sup_{\eps\in(0,\eps_0)}\big\|\chi_{\Omega\setminus V_\eps}\bar{e}(u_\eps)\big\|_{L^2(\tilde \Omega)}<+\infty\,.
\end{equation}
In particular, there exists $\xi\in L^2_{\rm loc}(\Omega;\R^{d\times 2^d})$ such that, up to subsequence (not relabeled), 
\begin{equation}\label{abstract_weak_convergence}
\chi_{\Omega\setminus V_\eps}\bar{e}(u_\varepsilon)\rightharpoonup \xi\quad \text{weakly in } L^2_{\rm loc}( \Omega;\R^{d\times 2^d})\,. 
\end{equation}
\end{lemma}

The core idea for the proof of the lemma is that discrete gradients can be related to the gradients of the corresponding piecewise affine \EEE interpolations, in the sense that there exists an absolute constant $C_*>0$ such that for all $Q \in \hat\Q_{\eps}$ it holds that 
\begin{equation}\label{eq: equivalence gradients}
\big|\overline{\nabla}u_\eps|_{Q}\big|^2\leq C_*\fint_{Q}|\nabla\tilde u_\eps|^2\, \mathrm{d}x\,,
\end{equation}
see \cite[Lemma 3.5(i)\EEE]{Schmidt:2009}.  (Note that at this point we are dealing with discrete displacements defined on the whole lattice $\eps\Z^d$, and hence \cite[Lemma 3.5
(ii),(iii)\EEE
]{Schmidt:2009} do not appear in our setting.) As this result needs to be applied on cubes of sidelength $\eps$, we need to regularize the void sets $W_\eps$ by introducing $V_\eps$, which are, loosely speaking, \textit{$\eps$-cubic approximations} of $W_\eps$. Then, \eqref{uniform_L^2_bound_symmetric gradients} follows from 
\textit{\ref{H3}}, 
\eqref{eq: equivalence gradients}, and Korn's inequality on each cube.

As Lemma \ref{weak_limit_for_discrete_multies} already shows  \eqref{conclusions_of_symmetric_Cauchy-Born}(i),(ii), it will remain to prove \eqref{conclusions_of_symmetric_Cauchy-Born}(iii). To this end, in view of \eqref{abstract_weak_convergence}, we need to show that 
$$\xi(x_0)\equiv [\chi_{\Omega\setminus W 
}e(u)](x_0)Z \quad \text{for }\L^d\text{-}\text{a.e. } x_0\in \Omega\,.$$ 
Since by \textit{\ref{H1}} and \eqref{conclusions_of_symmetric_Cauchy-Born}{(ii)} \EEE we have that $\chi_{\Omega\setminus V_\eps}\to\chi_{\Omega\setminus W
}$ strongly in $L^2(\Omega)$, it suffices to identify the weak limit $\xi$ for $\mathcal{L}^d$-a.e.\ Lebesgue point $x_0\in{\Omega}\setminus W
$ of $u$. It is not restrictive to further  suppose that the \emph{approximate gradient} of $u$ at $x_0$ exists. In fact, recall \EEE that for any $u\in GSBD^2(\Omega)$,  for $\mathcal{L}^d$-a.e.\ $x_0\in \Omega$ there exists $\nabla u(x_0) \in \R^{d \times d}$ such that
\begin{align}\label{lemma: approx-grad}
\lim_{\rho\to 0} \  \rho^{-d} \mathcal{L}^d\Big(\Big\{x \in B_\rho(x_0) \colon \,  \frac{|u(x) - u(x_0) - \nabla u(x_0)(x-x_0)|}{|x - x_0|}  > \lambda \Big\} \Big)  = 0 \text{ for all $\lambda >0$},
\end{align} 
see \cite[Corollary 5.2]{ChaCagSca} or  \cite[Theorem 2.9]{Friedrich:15-4}.

We can now perform a \textit{standard blow-up procedure} around such a point $x_0$. To this end, we introduce some further notation. Given a sequence $(\rho_\eps)_{\eps>0}\subset (0,+\infty)$, and setting $\hat\eps:=\eps/\rho_\eps>0$, \EEE consider 
the rescaled discrete maps $\zeta_{\eps}\colon\hat\eps \Z^d \EEE \to \R^d$, defined via
\begin{equation}\label{discrete_blow_ups}
\zeta_{\eps}(i):=\frac{u_\varepsilon(\rho_\varepsilon i+x_{0,\eps})-u(x_0)}{\rho_\eps}  \quad \forall i\in\hat\eps\Z^d\,, \ 
\end{equation}
where $x_{0,\varepsilon}\in \eps \Z^d \EEE $ are chosen such that, for an absolute constant $C>0$, there holds
\begin{equation}\label{approximation_by_lattice_points}
|x_{0,\eps}-x_0|\leq C\eps \quad \forall \eps>0\,. 
\end{equation}
The piecewise affine interpolations of $\zeta_{\eps}$ (again according to \cite[Section~3.2]{Schmidt:2009}) will be denoted by $\tilde{\zeta}_{\eps}$. \AAA We mention again that we identify $\bar{e}(u_\eps)$ and $\bar{e}(\zeta_{\eps})$ with their piecewise constant interpolations (being constant on each cell 
in the shifted lattice), i.e., for every $i\in \hat{\eps}\Z^d$ (using again the notation $\hat{i}:=i+\hat\eps(\tfrac{1}{2},\dots,\tfrac{1}{2})$), we have $\bar{e}(\zeta_\eps)|_{Q_{\hat\eps}(\hat i)}:=\bar{e}(\zeta_\eps)(i)$, where the fact that \EEE $\eps = \hat{\eps}\rho_\eps$ and \eqref{discrete_gradient_def} imply that \EEE
\begin{align}\label{eq: new equation}
\bar{e}(\zeta_\eps)(i)=\bar{e}(u_\eps)(\rho_\eps i+x_{0,\eps})\,.
\end{align}
As in the notations in the Introduction, we  set 
$V_\eps^{\rm c} := \Omega \setminus V_\eps$, and we define \EEE 
\begin{equation}\label{final_blow_ups}
\tilde z_{\eps}:=\chi_{V_{\eps}^{\rm c, bl}}\tilde\zeta_\eps|_{B_1}\in GSBD^2(B_1),\quad \text{where \ } V_{\eps}^{\rm c, bl} := \frac{1}{\rho_\eps}(V_\eps^{\rm c}-x_{0,\eps}).
\end{equation}
(Here and in the following, the superscript bl indicates rescaled sets and quantities in the blow-up procedure.) 
With this at hand, we have the following. 


\begin{lemma}[Blow-up]\label{blow-up}
For $\mathcal{L}^d$-a.e.\ point $x_0\in \Omega \setminus W$ the following holds true. There exists a sequence $(\rho_\eps)_{\eps>0}\subset (0,+\infty)$, with $\rho_\eps\to 0$ and $\eps/\rho_\eps\to 0$ as $\eps\to 0$, so that (up to a non-relabeled subsequence),
\begin{align}\label{blow_up_properties}
\begin{split}
\text{\rm{(i)}}&\quad \lim_{\eps\to 0}\fint_{B_{\rho_\eps}(x_0)}\chi_{V_\eps^{\rm c}}\bar{e}(u_\eps)\, \mathrm{d}x=\xi(x_0)\,, \quad \lim_{\eps\to 0}\fint_{B_{\rho_\eps}(x_0)}\chi_{V_\eps^{\rm c}}e(\tilde u_\eps)\, \mathrm{d}x=e(u)(x_0)\,;
\\
\text{\rm{(ii)}}&
\ \  \limsup_{\eps\to 0} \Big(\fint_{B_{2\rho_\eps(x_0)}}\chi_{V_\eps^{\rm c}}|\bar{e}(u_\eps)|^2\, \mathrm{d}x+\fint_{B_{2\rho_\eps(x_0)}}\chi_{V_\eps^{\rm c}}|e(\tilde u_\eps)|^2\, \mathrm{d}x\Big)<+\infty\,;\\
\text{\rm{(iii)}}&\quad \tilde z_{\eps}(y)\to \nabla u(x_0)y \text{\ \ in measure on } B_1 \text{ as }\eps\to 0;\\
\text{\rm{(iv)}}&\quad \lim_{\eps\to 0}\mathcal{H}^{d-1}(J_{\tilde{z}_{\eps}}\cap B_1)=0\,;\\
\text{\rm{(v)}}&\quad e(\tilde z_{\eps}) \rightharpoonup e(u)(x_0) \text{ weakly in } L^2(B_1;\R^{d \times d}_{\rm{sym}})\,;\\
\text{\rm{(vi)}}&\quad \chi_{V_\eps^{\rm c, bl}}\bar{e}(\zeta_{\eps})\rightharpoonup F \text{ weakly in } L^2(B_1;\R^{d\times 2^d}) \text{ for some }F  \in L^2(B_1;\R^{d\times 2^d})\,,
\end{split}
\end{align}
where $F$ satisfies $\fint_{B_1}F(y)\, \mathrm{d}y=\xi(x_0)$, and $\zeta_{\eps}, \tilde z_{\eps}$ are defined in \eqref{discrete_blow_ups}, \eqref{final_blow_ups}.  \EEE
\end{lemma}

Properties (i)--(v) essentially follow by standard blow-up arguments for $GSBD^2$-functions around points where the approximate gradient exists. For convenience of the reader, we will include the main arguments in the proof below. Eventually, property (vi) essentially follows from (i) and \eqref{eq: new equation} \EEE by a compactness and rescaling argument.

The remaining step is the following lemma. 

\begin{lemma}[Identification of the weak limit $F$]\label{lemma: weak limit F}
	In the setting of Lemma \ref{blow-up}, the function $F$ as introduced  in \eqref{blow_up_properties}\rm{(vi)} satisfies
	\begin{equation}\label{identification_of_blow_up_limit}
	F(y)=e(u)(x_0)Z \ \ \text{for } \L^d\text{-a.e.\ } y\in B_1\,.
	\end{equation}
\end{lemma}

The key idea in the proof of the lemma is to approximate the functions  $(\tilde{z}_\eps)_{\eps >0}$ \EEE by a sequence of Sobolev functions $(v_\eps)_{\eps >0}$ also converging to $\nabla u(x_0)y$. This approximation is achieved by a deep result in the theory of $GSBD^2$-functions, namely a \emph{Korn inequality for functions with small jump set}. From this sequence, we define discrete maps $w_\eps$ (later denoted by $v^\eta_\eps$, but simplified here for convenience) by a \emph{sampling argument}. We also introduce  the corresponding piecewise affine interpolations $\tilde{w}_\eps$ and we prove that $\tilde{w}_\eps$ are close to $\tilde{z}_\eps$, $\overline{e}(w_\eps)$ are close to $\bar{e}(\zeta_{\eps})$, and $\tilde{w}_\eps$ are bounded in $H^1(B_1;\R^d)$  independently of $\eps>0$\EEE. In particular, $\nabla \tilde{w}_\eps \rightharpoonup \nabla u(x_0)$ weakly in $L^2$.  We then conclude by applying \eqref{eq: bernd} on the sequence of Sobolev functions $(\tilde{w}_\eps)_{\eps >0}$ which allows us to deduce $\overline{\nabla} w_\eps \to \nabla u(x_0)Z$.

%
%
%
%
%
%

With these preliminary technical lemmata at hand, Theorem \ref{discrete_Cauchy_Born_any_d} follows.

\begin{proof}[Proof of Theorem \ref{discrete_Cauchy_Born_any_d}] Properties \eqref{conclusions_of_symmetric_Cauchy-Born}{(i),(ii)} are established in Lemma \ref{weak_limit_for_discrete_multies}. In view of Lemma \ref{blow-up} \EEE and Lemma \ref{lemma: weak limit F}, we get that $\xi(x_0)=\fint_{B_1} \EEE
F(y)\,\mathrm{d}y=e(u)(x_0)Z$ for a.e.\ $x_0\in  \Omega\setminus W
$. Therefore, \eqref{conclusions_of_symmetric_Cauchy-Born}{(iii)} follows from \eqref{abstract_weak_convergence}.
\end{proof}

\EEE

\subsection{Proofs of technical lemmata}\label{sec: int lemmata}


In this subsection we collect the proofs of the lemmata.
\EEE
\begin{proof}[Proof of Lemma \ref{weak_limit_for_discrete_multies}]
We first define the sets $(V_\eps)_{\eps>0}$ such that \eqref{conclusions_of_symmetric_Cauchy-Born}{(i),(ii)} are satisfied. Loosely speaking, the sets are constructed as  \textit{$\eps$-cubic approximations of $(W_\eps)_{\eps>0}$}. Then we prove \eqref{uniform_L^2_bound_symmetric gradients} and \EEE \eqref{abstract_weak_convergence}.\\
\begin{step}{1}(Construction)
We start by introducing some extra notation. Recall the definition of the simplices in \eqref{def:simplices} and the comment below it. 
In particular, $n_d:=\# \mathcal{S}_{\varepsilon,Q}$ is a purely dimensional constant, $Q=\bigcup_{S\in \mathcal{S}_{\varepsilon,Q}}S$, and $\nabla\tilde {u}_\eps|_S$ is constant for all $S\in \mathcal{S}_{\varepsilon,Q}$ and all $Q\in \hat\Q_\eps\EEE$.  Recalling also the notation introduced in \eqref{cubic_set_of_A}, let us set
\begin{align}\label{theta_good_cubes}
\begin{split}
\hat\Q^{\rm{good}}&:=\big\{Q\in \hat\Q_{\eps,W_\eps}\colon \H^{d-1}\big(\partial^\ast W_\eps\cap Q\big)\geq \theta_d\eps^{d-1}\big\}\,,\quad \hat\Q^{\rm{bad}}:=\hat\Q_{\eps,W_\eps}\setminus \hat\Q^{\rm{good}}\,;\\
\hat\Q^{\rm{bad},1}&:=\big\{Q\in \hat\Q^{\rm{bad}}\colon \L^{d}\big(Q\setminus W_\eps \big)\geq (1-\beta_d)\eps^{d}\big\} \,,\quad \hat\Q^{\rm{bad},2}:=\hat\Q^{\rm{bad}}\setminus\hat\Q^{\rm{bad},1}\,, 
\end{split}
\end{align} 
where $\beta_d:=\tfrac{1}{2n_d}>0$,  $\theta_d:=(\tfrac{\beta_d}{2C_d})^{(d-1)/d}$, and $C_d>0$ denotes the relative isoperimetric constant of the $d$-dimensional cube (see \cite[(3.43)]{Ambrosio-Fusco-Pallara:2000} for a version stated on balls instead of cubes). 
We now define
\begin{equation}\label{modified_set}
V_\eps:=W_\eps \cup\underset{Q\in \hat\Q^{\rm{good}}\cup\hat\Q^{\rm{bad},2}}{\bigcup} Q\,.
\end{equation}
Clearly, by construction and the definition of $\beta_d>0$ we get \eqref{conclusions_of_symmetric_Cauchy-Born}(i) since for each $Q \in\hat\Q
^{\rm{bad},1}$ and for each $S\in \mathcal{S}_{\varepsilon,Q}$ it holds that  \EEE
$$\L^d(S\setminus V_\eps)\ge \L^{d}(S) - \L^d(Q\cap W_\eps)\geq \frac{\eps^d}{n_d} -  \beta_d\eps^d =\frac{\eps^d}{2n_d} = \frac{1}{2}\L^{d}(S)\,.$$
The key point of the decomposition \eqref{theta_good_cubes} and the modification in \eqref{modified_set} lies in the fact that 
\begin{equation}\label{energy_estimate_on_mildly_bad_cubes}
\int_{Q}|e(\tilde u_\eps)|^2\, \mathrm{d}x\leq 2\int_{Q\setminus W_\eps}|e(\tilde u_\eps)|^2\, \mathrm{d}x \ \text{ for every $Q \subset \subset \Omega$ with } Q\in \hat\Q_{\eps, V^{\rm{c}}_\eps}\,,
\end{equation}
and that \eqref{conclusions_of_symmetric_Cauchy-Born}{(ii)} holds. \EEE  
\end{step}\\
\begin{step}{2}(Proof of \eqref{energy_estimate_on_mildly_bad_cubes} and \eqref{conclusions_of_symmetric_Cauchy-Born}{(ii)})  Let us  show \eqref{energy_estimate_on_mildly_bad_cubes}. First,  \EEE $Q\in\hat\Q_{\eps, V^{\rm{c}}_\eps}$ implies $Q\in \hat\Q
^{\rm{bad},1} \cup (\hat\Q_{\eps,\Omega} \setminus\hat\Q_{\eps,W_\eps})$. For cubes in $\hat\Q_{\eps,\Omega} \setminus \hat\Q_{\eps,W_\eps}$ \eqref{energy_estimate_on_mildly_bad_cubes} is clear, and for $Q \in \hat\Q
^{\rm{bad},1}$ we use \eqref{conclusions_of_symmetric_Cauchy-Born}{(i)} and the fact that $e(\tilde u_\eps)$ is constant on each \EEE $S\in \mathcal{S}_{\varepsilon,Q}$ to obtain
\begin{align*}
\int_{Q}|e(\tilde u_\eps)|^2\, \mathrm{d}x=\sum_{S\in \mathcal{S}_{\varepsilon,Q}}\L^d(S)|e(\tilde u_\eps)|_S|^2\leq 2\sum_{S\in \mathcal{S}_{\varepsilon,Q}}\L^d(S\setminus W_\eps)|e(\tilde u_\eps)|_S|^2=2\int_{Q\setminus W_\eps}|e(\tilde u_\eps)|^2\, \mathrm{d}x\,.
\end{align*} 
We now show \eqref{conclusions_of_symmetric_Cauchy-Born}{(ii)}. To this end, we first observe that by \textit{\ref{H1}} 
we have 
\begin{equation}\label{cardinality_of_good_cubes}
\#\hat\Q
^{\rm{good}}\leq \frac{\eps^{1-d}}{\theta_d}\sum_{Q\in \hat\Q
^{\rm{good}}}\H^{d-1}(\partial^\ast W_\eps \cap Q)\leq C\mathcal{H}^{d-1}(\partial^\ast W_\eps) \eps^{1-d} \leq C \eps^{1-d}
\end{equation}
 for some constant  $C>0$ depending only on {\rm (H.1)} and $d$. Regarding $\hat\Q
^{\rm{bad},2}$, we claim for all $Q\in \hat\Q
^{\rm{bad},2}$ that
\begin{equation}\label{last_family_small_volume}
\L^{d}(Q\setminus  W_\eps)\leq \L^{d}(Q\cap  W_\eps)  \ \text{ and } \  \L^d(Q\setminus  W_\eps )\leq C_d \big(\H^{d-1}(\partial^*W_\eps  \cap Q)\big)^{\frac{d}{d-1}} \le \tfrac{\beta_d}{2}\EEE \eps^d\,. 
\end{equation}
Indeed, if for some $Q\in \hat\Q
^{\rm{bad},2}$ we had $\L^{d}(Q\cap W_\eps)<\L^{d}(Q\setminus W_\eps)$, we would get that 
$$\beta_d\eps^d\leq \L^{d}(Q\cap  W_\eps )\leq\frac{1}{2}\eps^d\,,$$ and by applying the relative isoperimetric inequality (cf.\ \cite[(3.43)]{Ambrosio-Fusco-Pallara:2000}) on the pair $(W_\eps,Q)$, we would obtain
\begin{equation}\label{eq: repeat argument}
\beta_d\eps^d\leq \L^d(Q\cap  W_\eps )\leq C_d\big(\H^{d-1}(\partial^* W_\eps \cap Q)\big)^{\tfrac{d}{d-1}}\leq C_d(\theta_d\eps^{d-1})^{\tfrac{d}{d-1}}=\tfrac{\beta_d}{2}\eps^{d}\,, 
\end{equation}
which would yield a contradiction. Hence, $\L^{d}(Q\setminus W_\eps)\leq \L^{d}(Q\cap W_\eps)$, and by the relative isoperimetric inequality,  we also deduce the second inequality in \eqref{last_family_small_volume}. 
		
Now, \EEE for the first assertion in \eqref{conclusions_of_symmetric_Cauchy-Born}{(ii)}, by  \textit{\ref{H1}}, 
\eqref{cardinality_of_good_cubes} and \eqref{last_family_small_volume}, we can estimate
\begin{align*}
\L^d(V_\eps\setminus W_\eps)&\leq \sum_{Q\in \hat\Q
^{\rm{good}}\cup\Q
^{\rm{bad},2}}\L^d(Q\setminus W_\eps)
\leq \eps^d\#\hat\Q^{\rm{good}}+C_d\sum_{Q\in \hat\Q
^{\rm{bad},2}}\big(\H^{d-1}(\partial^*W_\eps\cap Q)\big)^{\frac{d}{d-1}}\\
&\leq C\eps+ C_d (\theta_d
\eps^{d-1})^{\frac{1}{d-1}} \sum_{Q\in \hat\Q
^{\rm{bad},2}}\H^{d-1}(\partial^*W_\eps \cap Q)\leq C\eps\,,
\end{align*}
and the claim follows. Regarding the second assertion in \eqref{conclusions_of_symmetric_Cauchy-Born}{(ii)}, by construction   we have 
\begin{align*}
\partial^\ast V_\eps \subset \partial^\ast W_\eps \cup \bigcup_{Q \in \hat\Q^{\rm{good}}} \partial^\ast Q \cup \bigcup_{Q\in \hat\Q
^{\rm{bad},2}} (\partial^\ast Q \setminus W_\eps) \cup \mathcal{N}\,,
\end{align*}
where $\mathcal{H}^{d-1}(\mathcal{N})=0$, see \cite[Theorem 16.3]{maggi2012sets}. Therefore, noting that for $Q \in \hat\Q
^{\rm{bad},2}$ , due to  \cite[Theorem 16.3]{maggi2012sets} (again up to a set of zero $\mathcal{H}^{d-1}$ measure), there holds  $\partial^\ast Q\setminus W_\eps =\partial^\ast (Q\setminus W_\eps)\cap \partial Q $, we have \EEE
\begin{align*}
\H^{d-1}(\partial^\ast V_\eps)&\leq \H^{d-1}(\partial^\ast W_\eps)+ 2d\eps^{d-1}\#\hat\Q^{\rm{good}}+\sum_{Q\in \hat\Q
^{\rm{bad},2}}\H^{d-1}\big(\partial^\ast (Q\setminus W_\eps)\cap \partial Q\big)\,.
\end{align*}
In view of \eqref{last_family_small_volume}, we can apply the the geometric Lemma \ref{lemma: boundary-fp} stated at the end of the section for $s=\frac{1}{2}$ on $P = Q \setminus W_\eps$ for $Q \in \hat\Q^{\rm{bad},2}$. By this, and by {\ref{H1}} and   \eqref{cardinality_of_good_cubes} we conclude
\begin{align*}
\H^{d-1}(\partial^\ast V_\eps)&\leq \H^{d-1}(\partial^\ast W_\eps) +  C+ C\sum_{Q\in \hat\Q
^{\rm{bad},2}} \H^{d-1}\big(\partial^\ast (Q\setminus W_\eps) \cap {\rm int}(Q)\big) \\& \le C(\H^{d-1}(\partial^\ast W_\eps) + 1)\leq C\,,
\end{align*}
where we used that $\partial^\ast (Q\setminus W_\eps) \cap  {\rm int}(Q) = \partial^\ast W_\eps \cap  {\rm int}(Q)$ by construction.
Thus, the second assertion in \eqref{conclusions_of_symmetric_Cauchy-Born}{(ii)} follows.  
\end{step}\\
\begin{step}{3}(Proof of \eqref{uniform_L^2_bound_symmetric gradients}--\eqref{abstract_weak_convergence}) \EEE 
It remains to prove \eqref{uniform_L^2_bound_symmetric gradients}--\eqref{abstract_weak_convergence}. \EEE For this purpose, let $\tilde\Omega\subset\subset\Omega\EEE$ be an open Lipschitz subset, and with a slight abuse of notation in this step, let us denote now $V_\eps^{\rm c}\EEE:=\tilde\Omega\setminus V_\eps$. 
Let $\eps_0:=\eps_0(\tilde\Omega,\Omega)\in (0,1)$ be small enough so that 
\begin{align}\label{eq: uuu}
\bigcup_{Q\in \hat\Q_{\eps,\tilde \Omega}}Q\subset\subset \Omega
\end{align}
for all $\eps\in (0,\eps_0]$. For every $Q\in \hat\Q_{\eps,V_\eps^{\rm c}}$, by the classical Korn's inequality, there exists an affine mapping $r_Q(x)\EEE:=A_Qx$, where $A_Q\in \R^{d\times d}_{\mathrm{skew}}$, \EEE so that for a purely dimensional constant $C>0$ it holds that 
$$\int_Q |\nabla(\tilde u_\eps-r_Q)|^2\, \mathrm{d}x\leq C \int_Q |e(\tilde u_\eps)|^2\, \mathrm{d}x\,.$$ 
This inequality, together with the fact that $\bar{e}(r_Q)=0$,  \eqref{eq: equivalence gradients} and \eqref{energy_estimate_on_mildly_bad_cubes}, allows us to estimate 
\begin{align*}
\int_{\tilde\Omega}|\chi_{\Omega\setminus V_\eps}\bar{e}(u_\eps)|^2\, \mathrm{d}x&\leq \sum_{Q\in\hat\Q_{\eps,V_\eps^{\rm c}}}\int_Q |\bar{e}(u_\eps)|^2\, \mathrm{d}x=\sum_{Q\in\hat\Q_{\eps,V_\eps^{\rm c}}}\int_Q |\bar{e}(u_\eps-r_Q)|^2\, \mathrm{d}x \\
&\leq \sum_{Q\in\hat\Q_{\eps,V_\eps^{\rm c}}\EEE}\int_Q |\overline{\nabla}(u_\eps-r_Q)|^2\, \mathrm{d}x\leq C\sum_{Q\in\hat\Q_{\eps,V_\eps^{\rm c}}\EEE}\int_Q |\nabla(\tilde u_\eps-r_Q)|^2\, \mathrm{d}x\\
&\leq C\sum_{Q\in\hat\Q_{\eps,V_\eps^{\rm c}}\EEE}\int_Q |e(\tilde u_\eps)|^2\, \mathrm{d}x \leq C \sum_{Q\in\hat\Q_{\eps,\tilde \Omega}} \int_{Q} \EEE |\chi_{\Omega\setminus W_\eps}e(\tilde u_\eps)|^2\, \mathrm{d}x\,. 
\end{align*}
Here, we used that the affine interpolation $\tilde{r}_Q\EEE$ coincides with $r_Q$. \EEE This along with \textit{\ref{H3}} \EEE
and \eqref{eq: uuu}  implies \EEE \eqref{uniform_L^2_bound_symmetric gradients}. Eventually,   \eqref{abstract_weak_convergence} follows from weak compactness  and \eqref{uniform_L^2_bound_symmetric gradients}, along with a suitable diagonal argument. \EEE 
\end{step}
\end{proof}




\begin{proof}[Proof of Lemma \ref{blow-up}]
\begin{step}{1} We again let $V_\eps^{\rm c} := \Omega \setminus V_\eps$. Choose a subsequence (not relabeled) such that \eqref{abstract_weak_convergence} holds. Without loss of generality, we can assume that $x_0\in \Omega\cap W^0$ (cf.\ \cite[Definition 3.60]{Ambrosio-Fusco-Pallara:2000} for the definition of density points), 
that $x_0$
is a Lebesgue point of both $e(u)$ and $\xi$, and that the approximate gradient $\nabla u(x_0)$ exists. We claim that 
\begin{align}\label{non_blow_up_properties_1}
\begin{split}
\text{(i)}& \lim_{\rho\to 0^+}\lim_{\eps\to 0}\fint_{B_{\rho}(x_0)}\chi_{V_\eps^{\rm c}}\bar{e}(u_\eps)\, \mathrm{d}x=\xi(x_0)\,, \quad  \lim_{\rho\to 0^+} \EEE\lim_{\eps\to 0}\fint_{B_{\rho}(x_0)}\chi_{V_\eps^{\rm c}}e(\tilde u_\eps)\, \mathrm{d}x=e(u)(x_0)\,;\\ 
\text{(ii)}&\ \underset{\rho\to 0^+}{\liminf
}\ \underset{\varepsilon\to 0}{\lim
}\Big(\fint_{B_{2\rho}(x_0)}\chi_{V_\eps^{\rm c}}|\bar{e}(u_\varepsilon)|^2\, \mathrm{d}x+\fint_{B_{2\rho}(x_0)}\chi_{V_\eps^{\rm c}}|{e}(\tilde u_\varepsilon)|^2\,\mathrm{d}x\Big)<+\infty\,; \\
\text{(iii)}&\ u_{x_0,\rho}(y):=\frac{u(x_0+\rho y)-u(x_0)}{\rho}\to \nabla u(x_0)y \text{\ \ in measure on } B_1 \text{ as }\rho\to 0^+\,;\\
\text{(iv)}& \  \underset{\rho\to 0^+}{\liminf
}\ \underset{\varepsilon\to 0}{\lim
}\  \rho^{-(d-1)}\mathcal{H}^{d-1}(\partial^*{V_\eps^{\rm c}}\cap B_{2\rho}(x_0))=0 \,.
\end{split}
\end{align} 
In fact, choose any open Lipschitz subset $\tilde \Omega \subset \subset \Omega$ such that $x_0 \in \tilde \Omega$. Regarding \eqref{non_blow_up_properties_1}(i), \eqref{abstract_weak_convergence} implies that for  $\rho>0$ sufficiently small, we have
$$ \lim_{\eps\to 0}\fint_{B_\rho(x_0)}\chi_{V_\eps^{\rm c}}\bar{e}(u_\eps)\, \mathrm{d}x=\fint_{B_{\rho}(x_0)}\xi\,\mathrm{d}x\,.$$ 
Since $x_0$ is a Lebesgue point of $\xi$, we get the first property in \eqref{non_blow_up_properties_1}(i). The second one follows analogously, by using \textit{\ref{H3}}  and \eqref{conclusions_of_symmetric_Cauchy-Born}{(ii)} instead of \eqref{abstract_weak_convergence}, by $x_0\in \tilde\Omega\cap W^0$, as well as the fact that $x_0$ is a Lebesgue point of $e(u)$. Moreover, \eqref{non_blow_up_properties_1}(iii) follows from the definition of the \EEE approximate gradient $\nabla u(x_0)$, see \eqref{lemma: approx-grad}, which we assumed to exist at $x_0$. \EEE
		
For \eqref{non_blow_up_properties_1}(ii), we notice that \eqref{uniform_L^2_bound_symmetric gradients} implies that for the family of Radon measures defined by $\mu^{\rm{bulk}}_{\eps}:=\chi_{V_\eps^{\rm c}}|\bar{e}(u_\eps)|^2\,\mathrm{d}\L^d$ we have that $\sup_{\eps\in(0,\eps_0)}\mu_{\eps}^{\rm{bulk}}(\tilde \Omega)<+\infty$. \EEE Hence, there exists a finite positive Radon measure $\mu^{\rm{bulk}}\EEE$ on $\tilde \Omega$  such that (up to a  subsequence, not relabeled) $\mu_{\eps}^{\rm{bulk}}\overset{\ast}{\rightharpoonup}\mu^{\rm{bulk}}$ in the sense of measures, as $\eps\to 0$. In particular, by the Radon-Nikodym theorem we can choose $x_0$ with the additional property that
\begin{equation*}
\lim_{\rho\to 0^+}\frac{\mu^{\rm{bulk}}(B_{2\rho}(x_0))}{\rho^d}<+\infty\,. 
\end{equation*} 
Moreover, since $\mu^{\rm{bulk}}$ is finite, except for a countable set of $\rho \in (0,1)$, we  have $\mu^{\rm{bulk}}(\partial B_{2\rho}(x_0))=0$. This along with weak convergence of measures shows the first part of \eqref{non_blow_up_properties_1}(ii). The proof of the second part is the same, by using $V_\eps \supset W_\eps$ and  \textit{\ref{H3}} \EEE in place of \eqref{uniform_L^2_bound_symmetric gradients}. \EEE

The proof of \eqref{non_blow_up_properties_1}(iv) is similar:  by  \eqref{conclusions_of_symmetric_Cauchy-Born}{(ii)}, up to a subsequence (not relabeled) there exists a finite positive Radon measure $\mu^{\rm{surf}}$ on $\Omega$ such that
the measures $\H^{d-1}|_{\partial^* V_\eps} \overset{\ast}{\rightharpoonup} \mu^{\rm{surf}}$ on $\Omega$ in the sense of measures, as $\eps\to 0$. Notice now that for $\L^d$-a.e.\ $x_0$ we have that
\begin{equation}\label{d-1_density_of_jump_goes_to_0}
\lim
_{\rho\to 0^+} \rho^{1-d}\mu^{\rm{surf}}(B_{2\rho}(x_0))=0\,,
\end{equation}
and we can thus assume without restriction that $x_0$ satisfies also \EEE \eqref{d-1_density_of_jump_goes_to_0}. \EEE
Indeed, suppose by contradiction that there exists a Borel set $C\subset \Omega$ with $\L^d(C)>0$ and $t>0$ such that
$$\limsup_{\rho\to 0^+} \rho^{1-d}\mu^{\rm{surf}}(B_{2\rho}(x))>t \quad \forall x\in C\,.$$
Then, as a consequence of \cite[Theorem 2.56]{Ambrosio-Fusco-Pallara:2000}, we would get $\mu^{\rm{surf}}(C)\geq t\H^{d-1}(C)$,
implying that $\mu^{\rm{surf}}(C)=+\infty$. But this is a contradiction to $\mu^{\rm{surf}}$ being a finite measure. We now derive \eqref{non_blow_up_properties_1}(\AAA iv\EEE) from \eqref{d-1_density_of_jump_goes_to_0} by using that $\mu^{\rm{surf}}(\partial B_{2\rho}(x_0)) = 0$ up to a countable number of values for $\rho$. \EEE  
%
\end{step}\\
\begin{step}{2}
With the properties of \eqref{non_blow_up_properties_1} at hand, we can now choose a suitable diagonal sequence $(\rho_\varepsilon)_{\varepsilon>0}$, with $\rho_\varepsilon\to 0$ and $\eps/\rho_{\eps}\to 0$ as $\eps\to 0$, for which 
\begin{align}\label{non_blow_up_properties_2}
\begin{split}
\text{(i) }& \underset{\varepsilon\to 0}{\lim}\ \fint _{B_{\rho_\eps}(x_0)}\chi_{V_\eps^{\rm c}}\bar{e}(u_\varepsilon)\, \mathrm{d}x=\xi(x_0);\ \ \
\underset{\varepsilon\to 0}{\lim}\ \fint _{B_{\rho_\eps}(x_0)}\chi_{V_\eps^{\rm c}}e(\tilde u_\varepsilon)\, \mathrm{d}x=e(u)(x_0);\\
\text{(ii) }& \underset{\varepsilon\to 0}{\limsup}\Big(\fint_{B_{2\rho_\eps}(x_0)}\chi_{V_\eps^{\rm c}}|\bar{e}(u_\varepsilon)|^2\, \mathrm{d}x+\fint_{B_{2\rho_\eps}(x_0)}\chi_{V_\eps^{\rm c}}|e(\tilde u_\varepsilon)|^2\, \mathrm{d}x\Big)<+\infty\,; \\
\text{(iii) }& \tilde {z}_{\eps}
(y)=\chi_{V_{\eps}^{\rm c, bl}}(y)\frac{\tilde u_\eps(x_{0,\eps}+\rho_\eps y)-u(x_0)}{\rho_\eps}\to \nabla u(x_0)y \text{\ \ in measure on } B_1 \text{ as }\eps\to 0;\\
\text{(iv) }& \underset{\varepsilon\to 0}{\lim}\  \rho_\eps^{-(d-1)}\mathcal{H}^{d-1}(\partial^* V_\eps^{\rm c}\cap B_{2\rho_\eps}(x_0))=0\,,
\end{split}
\end{align}
where we recall that $V_{\eps}^{\rm c, bl} = \frac{1}{\rho_\eps}(V_\eps^{\rm c}-x_{0,\eps})$. 
Indeed, recall that $\tilde{u}_\eps \to u$ in measure on $\Omega \setminus W
$ (see \ref{H2}) \EEE  
and that $\mathcal{L}^d(V_\eps
\triangle W
)\to 0$ by \textit{\ref{H1}} 
and \eqref{conclusions_of_symmetric_Cauchy-Born}(ii). Thus, denoting also $V_{\eps,\rho}^{\rm c, bl} = \frac{1}{\rho}(V_\eps^{\rm c}-x_{0,\eps})$, \eqref{non_blow_up_properties_1}(iii) and the fact that $x_{0,\eps}\to x_0$ with $x_0\in W^0$, imply that for $\L^d$-a.e. $y\in B_1$ it holds that
\begin{align*}
\lim_{\rho\to 0^+}\lim_{\eps\to 0} \left(\chi_{V_{\eps, \rho}^{\rm{c},bl}}(y)\frac{\tilde u_\eps(x_{0,\eps}+\rho y)-u(x_0)}{\rho}\right)&=\lim_{\rho\to 0^+}\chi_{\Omega \setminus W}(x_0+\rho y)u_{x_0,\rho}(y)=\nabla u(x_0)y\,,
\end{align*}
up to taking  a subsequence in $\eps>0$ (not relabeled).  
In combination with \eqref{non_blow_up_properties_1}(i), (ii), (iv), we can indeed extract a diagonal sequence \EEE $(\rho_\eps)_{\eps>0}$ with the properties claimed in \eqref{non_blow_up_properties_2}, where we use that measure convergence is metrizable.\\
We now verify \eqref{blow_up_properties}. Properties \eqref{blow_up_properties}(i)-(iii) are precisely \eqref{non_blow_up_properties_2}(i)-(iii), respectively. 
Regarding \eqref{blow_up_properties}(iv), recall 
the definition of $(\tilde{z}_{\eps})_{\eps>0}$ in \eqref{final_blow_ups}. We have that $$J_{\tilde z_{\eps}}
=\frac{1}{\rho_\eps}((J_{\chi_{V_\eps^{\rm c}}\tilde u_\eps}\cap B_{\rho_\eps}(x_{0,\eps}))-x_{0,\eps})\,.$$
Therefore, the fact that 
$\eps/\rho_\eps\to 0$, we get that $|x_{0,\eps}-x_0| \le C\eps \leq \rho_\eps$ for  $\eps>0$ small enough (see \eqref{approximation_by_lattice_points}), \EEE  and we can estimate
\begin{align*}
\H^{d-1}(J_{\tilde z_{\eps}} \cap B_1
)= \rho_\eps^{-(d-1)}\H^{d-1}(J_{\chi_{V_\eps^{\rm c}}\tilde u_\eps}\cap B_{\rho_\eps}(x_{0,\eps}))\leq \rho_\eps^{-(d-1)}\H^{d-1}\big(\partial^* V_\eps^{\rm c}\cap B_{2\rho_\eps}(x_0)\big)\,.
\end{align*}
Then, \eqref{blow_up_properties}(iv) follows from \eqref{non_blow_up_properties_2}(iv). 
		
As a preparation for  \eqref{blow_up_properties}(v),(vi), we observe that $e(\tilde{z}_{\eps})(y)=\chi_{V_{\eps}^{\rm c,bl}
}(y)e(\tilde{u}_\eps)(x_{0,\eps}+\rho_\eps y)$ for $\mathcal{L}^d$-a.e.\ $y \in B_1$ (see \eqref{final_blow_ups} and argue as in \eqref{eq: new equation}), \EEE and therefore, for all $\eps>0$ sufficiently  small so that $|x_{0,\eps}-x_0|\leq\rho_\eps$, we obtain 
\begin{equation*}
\int_{B_1}\hspace{-0.1cm}|e(\tilde z_{\eps})|^2=\int_{B_1}\hspace{-0.1cm}\chi_{V_\eps^{\rm c, bl}}(y)|e(\tilde u_\eps)|^2(x_{0,\eps}+\rho_\eps y)\, \mathrm{d}y=\frac{1}{\rho_\eps^d}\int_{B_{\rho_\eps}(x_{0,\eps})}\hspace{-0.2cm}\chi_{V_\eps^{\rm c}}|e(\tilde u_\eps)|^2\leq\frac{1}{\rho_\eps^d}\int_{B_{2\rho_\eps}(x_0)}\hspace{-0.2cm}\chi_{V_\eps^{\rm c}}|e(\tilde u_\eps)|^2\,.
\end{equation*}
In a similar fashion, again using \eqref{eq: new equation} we have \EEE
\begin{equation*}
\int_{B_1}\chi_{V_\eps^{\rm c, bl}}|\bar{e}(\zeta_{\eps})|^2
\leq\frac{1}{\rho_\eps^d} \int_{B_{2\rho_\eps}(x_0)}\chi_{V_\eps^{\rm c}}|\bar{e}(u_\eps)|^2\,.
\end{equation*}
In combination with \eqref{non_blow_up_properties_2}(ii), this yields
\begin{align*}
\limsup_{\varepsilon \to 0} \left(\|e(\tilde{z}_{\eps})\|_{L^2(B_1)} + \|\chi_{V_\eps^{\rm c, bl}} \overline{e}(\zeta_{\eps})\|_{L^2(B_1)} \right) <+\infty\,.
\end{align*}
In particular, in view of \eqref{blow_up_properties}(iv), we have that 
$$\limsup_{\eps \to 0} \Big(\int_{B_1}|e(\tilde z_\eps)|^2\,\mathrm{d}x+\H^{d-1}(J_{\tilde z_{\eps}}\cap B_1)\Big)<+\infty\,.$$
Since we already know that $\tilde z_{\eps}\EEE\to \nabla u(x_0)y$ in measure on $B_1$,  we can infer from   \cite[Theorem~1.1, (1.5b)]{Crismale-Cham} that for a suitable subsequence   we indeed have 
$$e(\tilde z_{\eps}) \rightharpoonup e(u)(x_0) \text{ weakly in } L^2(B_1;\R^{d\times d}_{\rm sym})\,,$$
i.e., \eqref{blow_up_properties}(v) holds. In a similar fashion, we get $\chi_{V_\eps^{\rm c, bl}}\bar{e}(\zeta_{\eps})\rightharpoonup F \text{ weakly in } L^2(B_1; \R^{d\times 2^d})$ for some $F \in L^2(B_1; \R^{d\times 2^d})$. Therefore, to see  \eqref{blow_up_properties}(vi), it suffices to \EEE check that $\fint_{B_1} F(y)\, \mathrm{d}y=\xi(x_0)$. We can argue as follows. By \eqref{approximation_by_lattice_points} we get \EEE 
$$\mathcal{L}^d\big(B_{\rho_\eps}(x_{0,\eps})\triangle B_{\rho_\eps}(x_0)\big)\leq C|x_{0,\eps}-x_0|\rho_\eps^{d-1}\leq C\eps\rho_\eps^{d-1}\,.$$ 
By this, along with the fact that $\eps/\rho_\eps\to 0$ as $\eps \to 0$, and \eqref{blow_up_properties}(ii), for $\eps>0$ small we can estimate
\begin{align}\label{limits_are_the_same}
\begin{split}
\left|\frac{1}{\rho_\eps^d}\int_{B_{\rho_{\eps}}(x_{0,\eps})\triangle B_{\rho_\eps}(x_0)}\chi_{V_\eps^{\rm c}}\bar{e}(u_\eps)
\right|&\leq \frac{1}{\rho_\eps^d}\int_{B_{\rho_\eps}(x_{0,\eps})\triangle B_{\rho_\eps}(x_0)}\chi_{V_\eps^{\rm c}}|\bar{e}(u_\eps)|\, \mathrm{d}x
\\
&\leq C\Big(\frac{\mathcal{L}^d(B_{\rho_\eps}(x_{0,\eps})\triangle
B_{\rho_\eps}(x_0))}{\rho_\eps^d}\Big)^{\frac{1}{2}}\Big(\fint_{B_{2\rho_\eps}(x_0)}\chi_{V_\eps^{\rm c}}|\bar{e}(u_\eps)|^2\Big)^{\frac{1}{2}}
\\
&\leq C\sqrt{\eps /\rho_\eps} \to 0\,. \EEE
\end{split}
\end{align}
Therefore, by the fact that weak $L^2$-convergence implies the convergence of averages, the change of variables $x:=x_{0,\eps}+\rho_\eps y$, \eqref{eq: new equation}, \EEE \eqref{blow_up_properties}(i),  and  \eqref{limits_are_the_same}  we obtain
\begin{align*}
\fint_{B_{1}}F(y)\, \mathrm{d}y&= \lim_{\eps\to0}\fint_{B_1}\chi_{V_\eps^{\rm c, bl}}(y)\bar{e}(\zeta_{\eps})(y)\, \mathrm{d}y=\lim_{\eps\to 0}\fint_{B_{\rho_{\eps}}(x_{0,\eps})}\chi_{V_\eps^{\rm c}}(x)\bar{e}(u_\eps)(x)\, \mathrm{d}x \\
&=\lim_{\eps\to 0}\fint_{B_{\rho_{\eps}}(x_0)}\chi_{V_\eps^{\rm c}}(x)\bar{e}(u_\eps)(x)\, \mathrm{d}x=\xi(x_0)\,.
\end{align*}
This concludes the proof.
\end{step}
\end{proof}



\begin{proof}[Proof of Lemma \ref{lemma: weak limit F}]
Let $x_0$ and $(\rho_\eps)_{\eps>0}\EEE$ be as in  Lemma \ref{blow-up}, and recall that $\hat\eps:=\eps/\rho_\eps\to 0$ as $\eps\to 0$. In the following, we choose $\eps_0 >0$ such that $B_{2\rho_{\eps_0}}(x_0) \subset \subset \Omega$, and always suppose that $\eps\in (0,\eps_0)\EEE$ without further notice. \EEE  In view of the blow-up properties established in Lemma \ref{blow-up}, we introduce the \textit{blow-up} versions of our lattice, the cubic decomposition, and the  $d$-dimensional simplices (cf.\ \eqref{def:simplices}) as
\begin{align*}
\begin{split}
\Z_\eps^{\rm{bl}}:=\rho_{\eps}^{-1}(\eps\Z^d-x_{0,\eps})\,,&\ \ \  \hat \Q_\eps^{\rm{bl}}:=\Big\{Q: Q=\rho_\eps^{-1}(Q'-x_{0,\eps})\,;\  Q'\in \hat\Q_{\eps}\Big\}\,;\\
\mathcal S_\eps^{\rm{bl}}:=\Big\{S: S=\rho_\eps^{-1}(S'-x_{0,\eps})\,;\   S'\in \mathcal{S}_\eps\Big\}\,,&\  \ \ 
\mathcal S_{\eps,Q}^{\rm{bl}}:=\Big\{S\in\mathcal S_\eps^{\rm{bl}}: S\subset Q\Big\} \  \  \forall Q\in \hat\Q_\eps^{\rm{bl}}\,.
\end{split}
\end{align*}
Since \eqref{conclusions_of_symmetric_Cauchy-Born}{(i)} is translation and scaling invariant, we also have that 
\begin{equation}\label{good_volume_bound_blow}
\mathcal{L}^d(Q\cap V_\eps^{\rm{c}, bl}) = 0 \text{ or }   \mathcal{L}^d(S\cap V_\eps^{\rm{c}, bl}) \ge  \frac{1}{2}  \mathcal{L}^d(S) \quad \forall \, S \in \mathcal{S}^{\rm{bl}}_{\varepsilon,Q}\,, \forall \, Q\in  \hat\Q_\eps^{\rm{bl}}\,.
\end{equation}
In order to localize our estimates, we fix from now on a small constant $\delta\in(0,1)$, with $\delta>\tilde C\hat\eps$, which will eventually be sent to $0$. (Here $\tilde C >0$ is a sufficiently large dimensional constant, so that all our subsequent estimates and set inclusions are true.) 
\\
\begin{step}{1\EEE}(\EEE Application of the Korn-Poincar\'e inequality  in $GSBD^2$\EEE) 
We claim that we can find an absolute \EEE
constant $C>0$, sets of finite perimeter $(\omega_\varepsilon)_{\eps>0
} \subset \EEE \mathfrak{M}(B_1)$, and Sobolev functions $(v_\varepsilon)_{\eps>0}
\subset H^1(B_1;\R^d)$ such that $v_\varepsilon = \tilde{z}_\varepsilon$ on $B_1\setminus \omega_\varepsilon$, and
\begin{align}\label{ineq:H1bound}
\begin{split}
\mathcal{H}^{ d-1 \EEE}(\partial^\ast\omega_\varepsilon)\leq C\mathcal{H}^{ d-1 \EEE}(J_{{ \tilde z\EEE}_\varepsilon}\cap B_1)\,, \quad \mathcal{L}^{ d \EEE}(\omega_\varepsilon)\leq C\mathcal{H}^{ d-1 \EEE}(J_{{ \tilde z\EEE}_\varepsilon}\cap B_1)^{{ d/(d-1) \EEE}}\,;\\
\|v_\varepsilon\|_{L^2(B_1)}+\|\nabla v_\eps\|_{L^2(B_1)} \leq C\left(\|e({ \tilde z\EEE}_\varepsilon)\|_{L^2(B_1)}+1\right)\,.
\end{split}
\end{align}
Moreover, we can find sets of finite perimeter $({T}_\varepsilon)_{\eps>0
}\subset\mathfrak{M}(B_1)$, with ${T}_\varepsilon =\bigcup_{S \in \mathcal{S}_\varepsilon^\mathrm{bad}}S$, where $\mathcal{S}^\mathrm{bad}_\varepsilon \subset \mathcal{S}_\varepsilon^{\rm{bl}}$ is to be specified in what follows, 
such that 
\begin{align}\label{ineq:gradientbound}
\begin{split}
\mathcal{L}^{d}({T}_\varepsilon)\leq C\mathcal{H}^{d-1}(J_{{\tilde z}_\varepsilon}\cap B_1)^{{d/(d-1)}}\,, \quad  
\|\nabla{\tilde z}_\varepsilon\|_{L^2(B_{1-\delta}\setminus {T}_\varepsilon)}
\leq C\left(\|e({\tilde z}_\varepsilon)\|_{L^2(B_1)}+1\right)\,.
\end{split}
\end{align} 
To see this, we first apply the Korn-Poincar\'e inequality for functions in $GSBD^2(B_1)$ with small jump set in the version of \cite[Theorem 1.2]{ChaCagSca}: \EEE there exists a  dimensional constant $C>0$ such that for every $\eps>0
$ there exists a set of finite perimeter $\omega_\varepsilon \subset B_1$ with 
\begin{align}\label{ineq:measureomegaeps}
\mathcal{H}^{d-1}(\partial^{\ast} \omega_\varepsilon) \leq C\mathcal{H}^{ d-1}(J_{\tilde{z}_\varepsilon}\cap B_1)\,, \quad \mathcal{L}^{d}(\omega_\varepsilon)\leq C\mathcal{H}^{d-1}(J_{\tilde{z}_\varepsilon}\cap B_1)^{{d/(d-1)}}\,,
\end{align}
and $v_\varepsilon\in H^1(B_1;\R^d)$ satisfying $v_\varepsilon= {\tilde z}_\varepsilon$ on $B_1\setminus \omega_\varepsilon$ and
$$ \| e(v_\varepsilon)\|_{L^2(B_1)} \leq C\|e(\tilde{z}_\varepsilon)\|_{L^2(B_1)}\,. $$
Now, \EEE 
by the classical Korn-Poincar\'e inequality, there exists an infinitesimal rigid motion $r_{\varepsilon} (y)=A_{\varepsilon}y+b_{\varepsilon}$, $A_{\varepsilon}\in \mathbb{R}^{{d}\times {d}}_{\mathrm{skew}}$, $b_{\varepsilon}\in \mathbb{R}^{d}$, \EEE 
such that \EEE
\begin{align}\label{ineq:symmetricveps}
\|v_\varepsilon -r_{\varepsilon}\|_{H^1(B_1)}\leq C \| e(v_\varepsilon)\|_{L^2(B_1)} \leq C\|e(\tilde{z}_\varepsilon)\|_{L^2(B_1)}\,.
\end{align}
By \eqref{ineq:measureomegaeps} the first part of \eqref{ineq:H1bound} follows. Now we claim that
\begin{align}\label{ineq:boundAeps}
|A_{\varepsilon}|+|b_{\varepsilon}|\leq C\left(\|e({ \tilde z\EEE}_\varepsilon)\|_{L^2(B_1)}+1\right)\,.
\end{align}
Assuming for the moment that  \eqref{ineq:boundAeps} holds, this implies 	$\|r_{\varepsilon}\|_{H^1(B_1)}\le C(\|e({ \tilde z\EEE}_\varepsilon)\|_{L^2(B_1)}+1)$ \EEE  and thus, by the triangle inequality and \eqref{ineq:symmetricveps}, also the last inequality in \eqref{ineq:H1bound} follows. 
Let us now show \eqref{ineq:boundAeps}. We define
\begin{equation}\label{ineq:Aetaeps}
K_\varepsilon:= D_\varepsilon \EEE \cup (-D_\varepsilon)\cup\omega_\varepsilon\cup(-\omega_\varepsilon)\,, \quad \text{where} \quad D_\varepsilon:= \{y\in B_1 \colon |{\tilde z}_\varepsilon(y) -\nabla u(x_0)y| >1\}  \,,
\end{equation}
and we note by \eqref{blow_up_properties}(iii),(iv) \EEE and  \eqref{ineq:measureomegaeps} that \EEE 
\begin{align}\label{eq: vanivoli}
\mathcal{L}^{d}(K_\varepsilon)\leq 2(\mathcal{L}^{ d }(D_\varepsilon \EEE )+\mathcal{L}^{ d }(\omega_\varepsilon))\leq C\big(\mathcal{L}^{ d }(D_\eps)+\mathcal{H}^{ d-1}(J_{\tilde z_\eps}\cap B_1)^{{ d/(d-1)}}\big) \to 0 \text{ \  as }\varepsilon\to 0\,.
\end{align}
By \eqref{ineq:symmetricveps}, the fact that \EEE $v_\varepsilon= {\tilde z}_\varepsilon$ on $B_1\setminus \omega_\varepsilon$, \EEE and \eqref{ineq:Aetaeps} we can estimate
\begin{align*}
\|r_{\varepsilon}\|_{L^2(B_1\setminus K_\varepsilon)}&\leq \|r_{\varepsilon}-{\tilde z}_\eps\|_{L^2(B_1\setminus K_\varepsilon)}+\|{ \tilde z}_\eps-\nabla u(x_0)\mathrm{id}\|_{L^2(B_1\setminus K_\varepsilon)}+\|\nabla u(x_0)\mathrm{id}\|_{L^2(B_1\setminus K_\varepsilon)} \\
&\leq C(\|e({\tilde z}_\varepsilon)\|_{L^2(B_1)}+1+|\nabla u(x_0)|)\,,
\end{align*}
i.e., $\| r_{\varepsilon}\|_{L^2(B_1\setminus K_\varepsilon)} \leq C(\|e({ \tilde z\EEE}_\varepsilon)\|_{L^2(B_1)}+1)$, where the last constant $C>0$ depends additionally on 
$\nabla u(x_0)\EEE$. Next, note that if $y \in B_1 \setminus K_\varepsilon$, then also $-y \in B_1\setminus K_\varepsilon$, and therefore 
\begin{align*}
\int_{B_1 \setminus K_\varepsilon} y\,\mathrm{d}y = 0\,.
\end{align*}
Combining the last  two  observations, we get 
\begin{align*}
\int_{B_1\setminus K_\varepsilon\EEE}\big(|A_{\varepsilon}y|^2 + |b_{\varepsilon}|^2\big)\,\mathrm{d}y=
\| r_{\varepsilon}\|^2_{L^2(B_1\setminus K_\varepsilon)} \leq C\left(\int_{B_1}|e({ \tilde z\EEE}_\varepsilon)|^2\, \mathrm{d}x+1\right)
\,.
\end{align*}
In view of \eqref{eq: vanivoli}, \eqref{ineq:boundAeps} follows. \\
Next, let us construct the sets $(T_\varepsilon\big)_{\eps>0
}$. We define  
\begin{align}\label{def:tilde_omega}
\mathcal{S}_{\omega_\varepsilon}^{\rm{bl}} : = \left\{ S \in \mathcal{S}_\varepsilon^{\rm{bl}}\colon\mathcal{L}^{d}( S \cap B_1\cap \omega_\varepsilon)\geq \frac{1}{4}\mathcal{L}^{d}(S\cap B_1)\right\}\,, \ {T}_\varepsilon:= \bigcup_{S \in\mathcal{S}_{\omega_\varepsilon}^{\rm{bl}}}S\cap
B_1\,. 
\end{align}
Let us first prove that
\begin{align}\label{int:gradientubound}
\|\nabla{ \tilde z}_\varepsilon\|_{L^2(B_{1-\delta}\setminus T_\eps)}\leq C(\|e({ \tilde z}_\varepsilon)\big\|_{L^2(B_1)}+1)\,.
\end{align}
In that respect, note that by construction, for every $S \in \mathcal{S}_\varepsilon^{\rm{bl}}$ there exists $ M_S^\eps\in \mathbb{R}^{{d}\times { d}}$ such that 
$\nabla \tilde{z}_\varepsilon= M_S^\eps$ on $ S\cap V_\eps^{\rm{c, bl}}%
$ and $\nabla \tilde{z}_\varepsilon=0$ on $ S\setminus V_{\eps}^{\rm{c, bl}}
$. Therefore, for every $S\in \mathcal{S}_\varepsilon^{\rm {bl}} \setminus \mathcal{S}_{\omega_\varepsilon}^{\rm{bl}}$ such that $S\cap B_{1-\delta}\neq\emptyset$ and $\mathcal{L}^d(S\cap V_\eps^{\rm{c}, \rm bl \EEE })>0$, \EEE we use the fact that $\delta>0$ is fixed large enough with respect to $\hat\eps$ (so that $S\subset B_1$) and \eqref{good_volume_bound_blow} to find
\begin{align*}
\begin{split}
\int_{S} |\nabla \tilde{z}_\varepsilon|^2\,\mathrm{d}y&\leq | M_S^\eps|^2\mathcal{L}^{d}(S)\leq 4| M_S^\eps|^2\mathcal{L}^{ d}((S \cap V_\eps^{\rm{c}, \rm bl \EEE }) 
\setminus \omega_\eps)=\EEE 4\int_{S\setminus \omega_\varepsilon}|\nabla \tilde{z}_\varepsilon|^2\,\mathrm{d} y\,.
\end{split}
\end{align*}
\EEE Summing over all $S \in \mathcal{S}_\varepsilon^{\rm{bl}} \setminus \mathcal{S}_{\omega_\varepsilon}^{\rm{bl}}$  such that $S\cap B_{1-\delta} \cap V_\eps^{\rm{c}, \rm bl \EEE } \neq \emptyset$, 
we obtain
\begin{align*}
\int_{B_{1-\delta} \setminus {T}_\varepsilon}|\nabla { \tilde z}_\varepsilon|^2\,\mathrm{d}y \leq\sum_{\underset{S\cap B_{1-\delta} \cap V_\eps^{\rm{c}, \rm bl \EEE } \neq \emptyset}{S \in \mathcal{S}_\varepsilon^{\rm{bl}} \setminus \mathcal{S}_{\omega_\varepsilon}^{\rm{bl}}}} \int_{S} |\nabla { \tilde z}_\varepsilon|^2\,\mathrm{d}y&\leq 4 \sum_{\underset{S\cap B_{1-\delta}\neq \emptyset}{S \in \mathcal{S}_\varepsilon^{\rm{bl}} \setminus \mathcal{S}_{\omega_\varepsilon}^{\rm {bl}}}}
\int_{S\setminus \omega_\varepsilon}|\nabla {\tilde z}_\varepsilon|^2\,\mathrm{d}y \leq4\int_{B_1\setminus \omega_\varepsilon}|\nabla {\tilde z}_\varepsilon|^2\,\mathrm{d} y 
\,.
\end{align*}
The last inequality, together with \eqref{ineq:H1bound} and the fact that $v_\varepsilon= {\tilde z}_\varepsilon$ on $B_1\setminus \omega_\varepsilon$, \EEE yields \eqref{int:gradientubound}. Moreover, by the definition of the sets $(T_\eps)_{\eps>0
}$, it is clear that $$\mathcal{L}^{ d \EEE}({T}_\varepsilon)\leq 4\mathcal{L}^{ d \EEE}(\omega_\varepsilon)\leq C\mathcal{H}^{d}(J_{\tilde{ z}_\varepsilon}\cap B_1)^{d/(d-1)}\,,$$
which concludes the proof of \eqref{ineq:gradientbound} and of this step. 
\end{step}\\
\noindent\begin{step}{2}(Construction of a good sample  for the discrete maps\EEE) In this step we define discrete mappings associated to the Sobolev functions $v_\eps$. To this end, apart from the parameter $\delta\in(0,1)$ that we fixed in the beginning of the previous step (recall that $\delta\gg \hat\eps$), we also fix $ \eta\in (0,\frac{1}{2})$. Eventually, in Step 5 of the proof, we will pass to the limits $\delta,\eta \to 0$. From now on, $v_\eps$ shall denote a fixed representative of the corresponding $L^1$-equivalence class. In this step, we \EEE show that 
there exists $\tau_\varepsilon^\eta \in Q_{\eta\hat\varepsilon}(0)$ such that 
\begin{align}\label{ineq:goodtauchoice}
\#\big\{ i\in\mathbb{Z}_\eps^{\rm{bl}}(B_{\AAA 1-\delta \EEE
})\colon \tilde{z}_\varepsilon(i+\tau_\varepsilon^\eta)\neq v_\varepsilon( i+\tau_\varepsilon^\eta)\big\} \leq \frac{2}{(\eta\hat\eps)^{ d}}\mathcal{L}^{d}(\omega_\varepsilon)\,,
\end{align}
and that for the maps $v_\varepsilon^\eta \colon\Z_{\eps}^{\rm{bl}}
\to \mathbb{R}^{d}$, defined by $v_\varepsilon^\eta(i) :=v_\varepsilon(i+\tau_\varepsilon^\eta)$ for $i \in \Z^{\rm bl}_\eps$, \EEE 
we have
\begin{align}\label{ineq:affineboundvepstilde}
\int_{B_{1-\delta}} |\nabla \tilde{v}_\varepsilon^\eta|^2\,\mathrm{d}y\leq \frac{C}{\eta^{ d \EEE}}\int_{B_1}|\nabla v_\varepsilon|^2\,\mathrm{d}y\,,
\end{align}
where $\tilde{v}_\varepsilon^\eta$ denotes the piecewise affine interpolation of $v_\varepsilon^\eta$, subordinate to the lattice $\Z_\eps^{\rm{bl}}$. 
To this end, let us consider the set $G_\varepsilon^\eta \subset Q_{\eta\hat\varepsilon}(0)$ defined as 
\begin{align*}
G_\eps^\eta:=\Big\{\tau\in Q_{\eta\hat\eps}(0)\colon \#\big\{ i\in \mathbb{Z}_\eps^{\rm{bl}}(B_{1-\delta
})\colon \tilde{z}_\varepsilon(i+\tau)\neq v_\varepsilon(i+\tau)\big\} \leq \frac{2}{(\eta\hat\varepsilon)^{d}}\mathcal{L}^{ d}(\omega_\varepsilon)\Big\}\,,
\end{align*}
for which we claim that 
\begin{align}\label{volume_of_good_tauset}
\L^d({G_{\eps}^\eta})\geq\frac{1}{2}(\eta\hat\eps)^d\,.
\end{align}
Indeed, consider the function $g_\eps\colon Q_{\eta\hat\eps}(0)\to\R$, defined by
$$g_\varepsilon(\tau) := \sum_{i\in \Z_\varepsilon^{\rm{bl}}(B_{1-\delta
})} \chi_{\omega_\varepsilon}(i+\tau)\,.$$
On the one hand, since $\tilde z_\eps=v_\eps$ on $B_1\setminus \omega_\eps$, i.e., $\{\tilde z_\eps\neq v_\eps\}\subset \omega_\eps$,  
we have
\begin{align*}
\int_{Q_{\eta\hat\varepsilon}(0)}g_\varepsilon\geq \int_{Q_{\eta\hat\eps}(0)\setminus G_\eps^\eta} \#\{i\in \mathbb{Z}_\eps^{\rm{bl}}(B_{1-\delta
})\colon \tilde{z}_\varepsilon(i+\tau)\neq v_\varepsilon(i+\tau)\}\, \mathrm{d}\tau>\frac{2}{(\eta\hat\varepsilon)^{d}} \mathcal{L}^{d}(\omega_\varepsilon)\mathcal{L}^{d}(Q_{\eta\hat\eps}(0)\setminus G_\eps^\eta)\,.
\end{align*}
On the other hand, for $\eta\in (0,\frac{1}{2})$, by exchanging summation and integration (note that 
the sum in the definition of the function $g_{\eps}$ is finite)\EEE, we have
\begin{align*}
\int_{Q_{\eta\hat\varepsilon}(0)}g_\varepsilon = \sum_{i\in \Z_\varepsilon^{\rm{bl}}(B_{1-\delta
})} \int_{Q_{\eta\hat\varepsilon}(0)}\chi_{\omega_\varepsilon}(i+\tau) \,\mathrm{d} \tau= \sum_{i\in \Z_\varepsilon^{\rm{bl}}(B_{1-\delta
})} \mathcal{L}^{d}(\omega_\varepsilon \cap Q_{\eta\hat\varepsilon}(i))\leq \mathcal{L}^{d}(\omega_\varepsilon)\,.
\end{align*}
Combining the last two inequalities, we arrive at \eqref{volume_of_good_tauset}.
\EEE\\ 
Now, let $\tau \in G_\varepsilon^\eta$ and $S\in \mathcal{S}_\eps^{\rm{bl}}$, say $S=\mathrm{conv}\{y_1, \dots ,y_{ d+1}\}$, where $ (y_l)_{l=1,\dots,{d+1}}\subset \EEE \Z_\eps^{\rm{bl}}
$.  Let $l\in \{1,\dots,d\}$
. By the theory of slicing of Sobolev functions (cf.\ \cite[Proposition~3.105]{Ambrosio-Fusco-Pallara:2000}), for $\mathcal{L}^{d}$-a.e.\ $\tau \in G_\varepsilon^\eta$ 
we have that $v_\varepsilon|_{[y_{l}+\tau,y_{ l+1}+\tau]} \in H^1([y_{l}+\tau, y_{ l+1}+\tau];\R^d)$, and 
\begin{align}\label{eq: slici}
{\frac{\d}{\d t} v_\varepsilon \big(y_{l} +\tau+ t(y_{l+1}-y_{l})\big) = \nabla v_\varepsilon\big(y_{l}+\tau + t(y_{l+1}-y_{l})\big)  (y_{l+1}-y_{l}) \ \  \text{ for a.e. } t\in(0,1)\,.}
\end{align}
Note that $|y_{l+1}-y_{l}|\leq C\hat\varepsilon$, that the unit vectors $\big(\frac{y_{l+1}-y_{l}}{|y_{l+1}-y_{l}|}\big)_{l=1,\dots,d}$ form a basis of $\R^d$, and that $|A|\leq C$ for a purely dimensional constant $C>0$, where $A\in \R^{d\times d}$ denotes here the corresponding transition matrix from the canonical basis $\{e_1,\dots,e_d\}$ of $\R^d$ to this basis. Let us also denote by $\tilde{v}_\varepsilon( \ZZZ \cdot \EEE +\tau)$ the piecewise affine function with $\tilde{v}_\varepsilon(y_l+\tau):= v_\varepsilon(y_l+\tau)$ for $l=1,\ldots, d+1$, so that by construction, $\nabla \tilde v_\eps|_S(\cdot+\tau)$ is constant on $S\in\mathcal{S}_\eps^{\rm{bl}}$. Then, using Jensen's inequality and Fubini's Theorem, \eqref{eq: slici}, and recalling the definition in \eqref{eq: thick-def}, we can estimate
\begin{align*}
\int_{G_\eps^\eta}\int_S |\nabla \tilde v_\eps (y+\tau)|^2\, \mathrm{d}y\, \mathrm{d}\tau&\leq C\hat\eps^d\sum_{l=1}^d\int_{G_\eps^\eta} \left|\nabla \tilde v_\eps|_S(\cdot+\tau)\Big(\frac{y_{l+1}-y_l}{|y_{l+1}-y_l|}\Big)\right|^2\, \mathrm{d}\tau 
\\
&= C\hat\eps^d\sum_{l=1}^d\frac{1}{|y_{l+1}-y_l|^2}\int_{G_{\eps}^\eta}|v_{\eps}(y_{l+1}+\tau)-v_\eps(y_l+\tau)|^2\, \mathrm{d}\tau\\
&\le C\hat\eps^d\sum_{l=1}^d\int_{G_\varepsilon^\eta}\int_0^1\left|\nabla v_\varepsilon\big(y_l+\tau +t(y_{l+1}-y_l)\big)\right|^2\mathrm{d}t\,\mathrm{d}\tau\leq C\hat\eps^d \int_{(S)_{\hat\eps}}|\nabla v_\eps|^2\,.
\end{align*}
Summing over all simplices $S\in \mathcal{S}_\eps^{\rm{bl}}$ intersecting $B_{1-\delta}$ (and using that $\delta\gg\hat \eps$), we infer  
\begin{align*}
\int_{G_\eps^\eta}\int_{B_{1-\delta}} |\nabla \tilde v_\eps(y+\tau)|^2\, \mathrm{d}y\, \mathrm{d}\tau & \le \int_{G_\eps^\eta}  \underset{S\cap B_{1-\delta}\neq\emptyset}{\sum_{S \in \mathcal{S}_\varepsilon^{\rm{bl}}}} \int_{S} |\nabla \tilde v_\eps(y+\tau)|^2\, \mathrm{d}y\, \mathrm{d}\tau  \le \underset{S\cap B_{1-\delta}\neq\emptyset}{\sum_{S \in \mathcal{S}_\varepsilon^{\rm{bl}}}}  C\hat\eps^d \int_{(S)_{\hat\eps}}|\nabla v_\eps|^2   \\
& \le C\hat\eps^d \int_{B_1} |\nabla v_\eps|^2. 
\end{align*}
Here, we used that for every  $S \in \mathcal{S}_\varepsilon^{\rm{bl}}$ there is only a bounded number of $S'\in \mathcal{S}_\varepsilon^{\rm{bl}}$, $S'\neq S$,  such that $(S)_{\hat\varepsilon} \cap (S')_{\hat\varepsilon}\neq \emptyset$. Thus, the constant $C>0$ in the last inequality is  purely dimensional. Using the fact that $\L^d(G_{\eps}^\eta)\geq\frac{1}{2}(\eta\hat\eps)^d$, we can choose $\tau_\varepsilon^\eta \in G_\varepsilon^\eta$ such that  
$${\int_{B_{1-\delta}} |\nabla \tilde v_\eps (y+ \tau^\eta_\eps \EEE )|^2\, \mathrm{d}y \le \frac{1}{\L^d(G_{\eps}^{\eta})}\int_{G_\eps^\eta}\int_{B_{1-\delta}} |\nabla \tilde v_\eps(y+\tau)|^2\, \mathrm{d}y\, \mathrm{d}\tau \leq \frac{C}{\eta^{ d}}\int_{B_1}|\nabla v_\varepsilon|^2\,\mathrm{d}y\,.}$$ 
This finishes the proof of \eqref{ineq:affineboundvepstilde}. As by definition of $G_{\eps}^\eta$ also \eqref{ineq:goodtauchoice} holds, the step is concluded. \EEE 
\end{step}\\
\begin{step}{3}(Estimate of the bad lattice points)
Let $\mathcal{L}_\eps := \lbrace i\in \Z_\eps^{\rm{bl}}\colon Q_{2\hat\varepsilon}( \hat{i}) \cap B_{1-2\delta} \neq \emptyset \rbrace $, and \EEE define the set of ``good'' lattice points $\mathcal{L}_\varepsilon^\mathrm{good} \subset \mathcal{L}_\eps \EEE $  by
\begin{align}\label{def:Lepsgood}
\hspace{-0.5em}\begin{split}
\mathcal{L}_\varepsilon^\mathrm{good}  := \big\{& i\in \mathcal{L}_\eps \EEE\colon T_\varepsilon \cap Q_{2\hat\varepsilon}( \hat{i})\EEE=\emptyset,\ { \tilde z}_\varepsilon( i+\hat\varepsilon\xi+\tau_\varepsilon^\eta)=v_\varepsilon( i+\hat\varepsilon\xi+\tau_\varepsilon^\eta)\ \ \forall\xi \in \mathbb{Z}^{ d}\cap[-1,1]^d\big\}\,.
\end{split}
\end{align}	
The other lattice points are considered to be ``bad''. We claim that
\begin{align}\label{ineq:badsetestimate}
\#\big(\mathcal{L}_\eps \EEE \setminus  \mathcal{L}_\varepsilon^\mathrm{good}\big) \leq \frac{C}{\hat\varepsilon^{d}}\left(\eta^{ -d}\mathcal{L}^{d}(\omega_\varepsilon) +\mathcal{L}^{d}(T_\varepsilon)\right)\,.
\end{align}
In order to prove this estimate, we note that
\begin{align}\label{incl:badsetlattice}
\mathcal{L}_\eps \EEE \setminus  \mathcal{L}_\varepsilon^\mathrm{good} \subset  \L_\varepsilon^1 \cup  \L_\varepsilon^2\,,
\end{align}
where 
\begin{align*}
& \L_\varepsilon^1:=\{i\in \mathcal{L}_\eps \EEE \colon T_\varepsilon\cap Q_{2\hat\varepsilon}(\hat i)\neq \emptyset\} \,;\\
&  \L_\varepsilon^2:= \{i\in \mathcal{L}_\eps \EEE \colon  \tilde{z}_\varepsilon(i+\hat\varepsilon\xi+\tau_\varepsilon^\eta)\neq v_\varepsilon(i+\hat\varepsilon\xi+\tau_\varepsilon^\eta)\text{ for some } \xi \in \mathbb{Z}^{d}\cap[-1,1]^d\}\,.
\end{align*}
Recall that,  by its definition in \eqref{def:tilde_omega}, $T_\varepsilon$ is a union of simplices (relative to $B_1$). Therefore, if $T_\varepsilon \cap Q_{2\hat\varepsilon} (\hat i)\neq \emptyset$, then $\mathcal{L}^{d}(T_\varepsilon \cap Q_{2\hat\varepsilon}(\hat i))\geq c\hat\varepsilon^{d}$ for some  purely dimensional constant $c\in (0,1)$. 
Therefore,  by \eqref{ineq:Qintersectbound} we obtain 
\begin{align}\label{donotremove1}
\#\L_\varepsilon^1 \leq  \frac{C}{\hat\varepsilon^{d}}  \sum_{i\in \mathcal{L}_\eps \EEE } \mathcal{L}^{d}\big(T_\varepsilon \cap Q_{2\hat\varepsilon}(\hat i)\big)  \le   \frac{C}{\hat\varepsilon^{d}}\mathcal{L}^{d}(T_\varepsilon)\,.
\end{align}
\EEE 
On the other hand, if $ i\in \L_\varepsilon^2$, there exists $\xi \in \mathbb{Z}^{ d} \cap [-1,1]^d$ such that $\tilde z_\varepsilon( i+\hat\varepsilon\xi+\tau_\varepsilon^\eta)\neq v_\varepsilon( i+\hat\varepsilon\xi+\tau_\varepsilon^\eta)$. Thus, as $\delta \gg \hat{\eps}$, \EEE  by \eqref{ineq:goodtauchoice} we obtain 
\begin{align}\label{ineq:B2bound}
\#\L_\varepsilon^2\leq C\#\big\{i\in \Z_\eps^{\rm{bl}}(B_{ 1-\delta\EEE
}) \colon { \tilde z}_\varepsilon(i+\tau_\varepsilon^\eta)\neq v_\varepsilon( i+\tau_\varepsilon^\eta)\big\} \leq \frac{C}{(\eta\hat\varepsilon)^{d}}\mathcal{L}^{ d}(\omega_\varepsilon)\,.
\end{align}
Combining \eqref{incl:badsetlattice}--\eqref{ineq:B2bound}, we obtain \eqref{ineq:badsetestimate}.
\end{step}\\
\begin{step}{4\EEE}(Weak convergence of $\tilde{v}_\varepsilon^\eta$) 
The goal of this step is to show that,  as $\eps\to 0$,
\begin{align}\label{eq:weakconvergencevtilda}
\hspace{-0.2cm}\tilde{v}_\varepsilon^\eta\EEE \rightharpoonup \nabla u(x_0)y \text{ weakly in } H^1(B_{1-2\delta}\EEE;\R^d)\,, \ \overline{\nabla}v_\varepsilon^\eta \rightharpoonup \nabla u(x_0)Z \text{ weakly in } L^2_{\rm loc}(\AAA B_{1-2\delta}\EEE;\R^{d\times 2^d})\,.
\end{align}
These convergence properties will be achieved by showing that
\begin{align}\label{ineq:L2differencegood}
\hspace{-1.5em} \sum_{i \in \mathcal{L}_\varepsilon^\mathrm{good}} \int_{Q_{\hat\varepsilon}(\hat{i})} \big(\hat\varepsilon^{-2}|\tilde{v}_\varepsilon^\eta-{\tilde z\EEE}_\varepsilon|^2
+|\nabla\tilde{v}_\varepsilon^\eta-\nabla{ \tilde z\EEE}_\varepsilon|^2\big)\, \mathrm{d}y 
+ \EEE \sum_{ i \EEE \in \mathcal{L}_\varepsilon^\mathrm{good}} \int_{Q_{\hat{\eps}}(\hat{i})}  \chi_{V_\eps^{\rm c,  bl}}   \big|\overline{\nabla} v_\varepsilon^\eta -  \overline{\nabla} \zeta_\varepsilon  \EEE \big|^2 \, {\rm d}y \EEE \leq C\eta^2\,,
\end{align}
%
as well as that
\begin{align}\label{ineq:H1boundvepstilda}
\sup_{\varepsilon>0\EEE
} \|\tilde{v}_\varepsilon^\eta\|_{H^1(\AAA B_{1-\AAA\delta\EEE}\EEE)} \leq \frac{C}{\eta^{{ d\EEE}/2}}  \ \ \text{ and } \ \  \sup_{\eps>0\EEE}\EEE\|\overline{\nabla}v_\varepsilon^\eta\|_{L^2(\AAA B_{1-2\delta}\EEE)} \leq \frac{C}{\eta^{{ d\EEE}/2}}\,.
\end{align}
To this  end, we first prove that for each $ i \in \mathcal{L}_\varepsilon^\mathrm{good}$ it holds that \EEE
\begin{equation}\label{ineq:closenessgood}
\hat\eps^{-1}\|\tilde{v}_\varepsilon^\eta-\tilde{z}_\varepsilon\|_{L^\infty({Q_{\hat\eps\EEE}(\hat i)\EEE})}+\|\nabla\tilde{v}_\varepsilon^\eta-\nabla\tilde{z}_\varepsilon\|_{L^\infty({Q_{\hat\eps\EEE}(\hat i)})}\leq C\eta\| \nabla \tilde{z}_\varepsilon\|_{L^\infty({Q_{2\hat\eps\EEE}(\hat i)})}\,.
\end{equation}
\EEE Indeed, for $ i \EEE \in \mathcal{L}_\varepsilon^\mathrm{good}$,  by \eqref{def:Lepsgood} we have $v_\varepsilon^\eta( j\EEE):= v_\varepsilon( j\EEE+\tau_\varepsilon^\eta) = { \tilde z}_\varepsilon( j\EEE+\tau_\varepsilon^\eta)$  for every $j\in \Z_\eps^{\rm{bl}}(\overline{ Q_{\hat\eps\EEE}(\hat{i})})$. \EEE Therefore, for all such $j$, since ${ \tilde z}_\varepsilon$ is piecewise affine in $Q_{\hat\varepsilon}( j) \subset Q_{2\hat\eps\EEE}(\hat i)$, 
\begin{align*}
|\tilde v_\varepsilon^\eta(j)-{ \tilde z}_\varepsilon(j)| = |{ \tilde z\EEE}_\varepsilon( j+\tau_\varepsilon^\eta)-{ \tilde z\EEE}_\varepsilon(j)| \leq |\tau_\varepsilon^\eta| \|\nabla \tilde{z}_\varepsilon\|_{L^\infty({Q_{\hat\eps\EEE}(j)})}\EEE\leq C\eta\hat\varepsilon\EEE\|\nabla \tilde{z}_\varepsilon\|_{L^\infty({Q_{2\hat\eps\EEE}(\hat i)})}\,.
\end{align*}
Thus, \EEE by the definition of the piecewise affine interpolation (cf.\ \cite[Section~3.2, Equation (6)]{Schmidt:2009}) \EEE and \eqref{def:Lepsgood}, we have that for $ i\EEE \in \mathcal{L}_\varepsilon^\mathrm{good}$
$${\hat\eps\EEE^{-1}\|\tilde{v}_\varepsilon^\eta-\tilde{z}_\varepsilon\|_{L^\infty( Q_{\hat\eps\EEE}(\hat i)\EEE)}
\leq C\eta\| \nabla \tilde{z}_\varepsilon\|_{L^\infty(Q_{2\hat\eps\EEE}(\hat i)\EEE)}}\,.$$ 
\EEE In order to finish the proof of \eqref{ineq:closenessgood}, we consider a point $i\in \L_\eps^{\mathrm{good}}$ and a $d$-simplex $S\in \mathcal{S}_{\eps}^{\rm{bl}\EEE}$ with $S \subset Q_{\hat\varepsilon\EEE}(\hat{ i\EEE})$, say $S=\mathrm{conv}\{ x_1, \dots , x_{ d+1 \EEE}\}$, where $ (x_l)_{l=1,\dots,{ d+1 \EEE}}\subset\EEE\Z_\eps^{\rm{bl}}\EEE$. 
Arguing as in the proof of \eqref{ineq:affineboundvepstilde} in Step 2, using that 
$c\hat\varepsilon\EEE  \leq |x_{l+1}-x_{l}|\leq C\hat\eps\EEE \ \text{for } l=1,\dots,d\EEE$, 
we \EEE deduce that
\begin{align*}
\big|(\nabla \tilde v_\eps^\eta-\nabla \tilde z_\eps)|_S\big|^2
\leq C \sum_{l=1}^d\frac{| \tilde v_\varepsilon^\eta(x_{ l+1\EEE})-{\tilde z\EEE}_\varepsilon(x_{ l+1})|^2 +| \tilde v\EEE_\varepsilon^\eta(x_{l})-{ \tilde z\EEE}_\varepsilon(x_{ l\EEE})|^2}{\hat\eps^2\EEE} \leq C\eta^{2} \| \nabla \tilde{z}_\varepsilon\|^2_{L^\infty({Q_{2\hat\eps\EEE}(\hat i)})}\,.
\end{align*}
\EEE This yields the desired estimate on each $S \subset Q_{\hat\eps\EEE}(\hat i)$, and \eqref{ineq:closenessgood} follows. 
		
\EEE Let us now justify \eqref{ineq:L2differencegood}. For this, we use \EEE the facts that  ${ \tilde z\EEE}_\varepsilon$ is piecewise affine (implying the equivalence between the $L^2$ and $L^\infty$ norms), \EEE  \eqref{ineq:closenessgood},  the definition of $\mathcal{L}_\varepsilon^\mathrm{good}$ in \eqref{def:Lepsgood},  \eqref{ineq:gradientbound}, and \eqref{blow_up_properties}(v), \EEE 
\EEE to estimate 
\begin{align*}
\sum_{ i\EEE \in \mathcal{L}_\varepsilon^\mathrm{good}}\int_{Q_{\hat\eps\EEE}
(\hat{ i\EEE})}\big({\hat\eps\EEE}
^{-2}\EEE|\tilde{v}_\varepsilon^\eta-{ \tilde z\EEE}_\varepsilon|^2  +|\nabla\tilde{v}_\varepsilon^\eta-\nabla{ \tilde z\EEE}_\varepsilon|^2\big)\EEE\,\mathrm{d}y &\leq C\eta\EEE^2\sum_{ i\EEE \in \mathcal{L}_\varepsilon^\mathrm{good}}\int_{Q_{2{\hat\eps\EEE}
}(\hat{ i\EEE})} |\nabla { \tilde z\EEE}_\varepsilon|^2 \leq C\eta\EEE^2\int_{B_{1-\delta} \EEE \setminus { T\EEE}_\varepsilon}|\nabla { \tilde z \EEE}_\varepsilon|^2 \\&\leq C\eta\EEE^2\sup_{\varepsilon>0\EEE
\EEE}\big(\EEE\|e({ \tilde z\EEE}_\varepsilon)\|_{L^2(B_1)}^2+1\big)\EEE \leq C\eta\EEE^2\,.
\end{align*}
The bound for the second addend \EEE in the left hand side of \eqref{ineq:L2differencegood}, which involves the discrete gradients, follows from the corresponding bound of the gradients of the piecewise affine interpolations. Indeed, let $Q:=Q_{\hat\eps}(\hat i)\,,\  i\in \mathcal{L}^{\rm good}_\eps\EEE$. If $\mathcal{L}^d(Q \cap V^{\rm c, bl}_\eps) =0$, then  $\chi_{V_\eps^{\rm c,  bl}}|\overline{\nabla} v_\varepsilon^\eta -\overline{\nabla} \zeta_\varepsilon | = 0$ (a.e.) \EEE on $Q$.  On the other hand, if $\mathcal{L}^d(Q \cap V^{\rm c, bl}_\eps)>0$, then  \eqref{good_volume_bound_blow} and an argument similar to \eqref{energy_estimate_on_mildly_bad_cubes} (with $\tilde{v}_\varepsilon^\eta-{ \tilde \zeta}_\varepsilon$ instead of $\tilde u_\eps$)  imply 
$$\int_{Q} |\nabla\tilde{v}_\varepsilon^\eta-\nabla{ \tilde \zeta}_\varepsilon|^2  \le C\int_{Q}|\nabla\tilde{v}_\varepsilon^\eta-\nabla{ \tilde z}_\varepsilon|^2\,.      $$ 
Thus,  by applying  \eqref{eq: equivalence gradients} on each cube  and using the first part of \eqref{ineq:L2differencegood}  we arrive at the desired estimate.


\EEE Next, we verify \eqref{ineq:H1boundvepstilda}. \EEE Note first that \EEE by \eqref{blow_up_properties}(iv), \eqref{ineq:H1bound}, \eqref{ineq:gradientbound}, and \eqref{ineq:badsetestimate}, we obtain  
\begin{align*}
\lim_{\varepsilon\to 0}\L^{ d\EEE}\Big(B_{1-2\delta \EEE} \setminus \bigcup_{ i\EEE \in \mathcal{L}_\varepsilon^\mathrm{good}} Q_{\hat\eps\EEE}
(\hat{ i\EEE})\Big) \leq \lim_{\varepsilon\to 0}\L^{ d\EEE}\Big(\bigcup_{ i\EEE\in \mathcal{L}_\eps \EEE\setminus \mathcal{L}_\varepsilon^\mathrm{good}}Q_{\hat\eps\EEE}(\hat{ i\EEE})\Big) \leq \lim_{\varepsilon \to 0} C (\eta^{-{ d\EEE}}\L^{ d\EEE}(\omega_\varepsilon)+\L^{ d\EEE}({ T\EEE}_\varepsilon)) =0\,.
\end{align*}
This, together with \eqref{ineq:L2differencegood}  and \eqref{blow_up_properties}(iii), implies that $\tilde{v}_\varepsilon^\eta$ converges to the linear function $\nabla u( x_0\EEE)y$ in measure on $B_{1-2\delta}\EEE$  as $\eps\to 0$. This yields 
$$\limsup_{\eps\to 0}\L^{ d}(\AAA \tilde K\EEE_\varepsilon)\le C\delta\,, \quad \text{where} \quad \tilde K\EEE_\varepsilon:=\{ y\EEE \in B_{1-\delta}  \EEE \colon |\tilde{v}_\varepsilon^\eta(y)\EEE-\nabla u(x_0) y\EEE| >1 \EEE \} \,.$$ 
Setting for brevity \EEE $b_\varepsilon^\eta:=\fint_{ B_{1-\delta}} \tilde{v}^\eta_\varepsilon(y)\,\mathrm{d}y$, we can estimate  
\begin{align*}
|b_\eps^\eta|\L^{ d\EEE}(B_{1-\delta})&\leq  \int_{B_{1-\delta}\setminus \AAA\tilde K\EEE_\eps}|\tilde v_\eps^\eta(y)-\nabla u(x_0)y|\, \mathrm{d}y+|\nabla u(x_0)|\L^{ d\EEE}(B_1)+\int_{\AAA\tilde K\EEE_\eps}|\tilde v_\eps ^\eta(y)|\, \mathrm{d}y \\
&\leq  1+|\nabla u(x_0)|+\int_{B_{1-\delta}}| \tilde v_\eps ^\eta(y)-b_\eps^\eta|\, \mathrm{d}y+|b_\eps^\eta|\L^{ d}(\tilde K_\eps)\,.
\end{align*}
Therefore, by Poincar\'e's and H\"older's inequality for $\tilde v_\eps^\eta$ in $B_{1-\delta}$  we have
\begin{align*}
|b_\varepsilon^\eta|\L^{ d\EEE}(B_{1-\delta} \setminus \AAA\tilde K\EEE_\varepsilon)&\leq C\big(1+|\nabla u(x_0)|\big)+\|\tilde{v}_\varepsilon^\eta-b_\varepsilon^\eta \|_{L^1(B_{1-\delta})}\leq C(1 +\|\nabla \tilde{v}_\varepsilon^\eta\|_{L^2(B_{1-\delta})})\,,
\end{align*}
where $C>0$ depends also on $\nabla u(x_0)$. 
Hence,  for $\varepsilon>0$ and $\delta >0$ small enough, we get \EEE
$$|b_\varepsilon^\eta| \leq C(1+\|\nabla \tilde{v}_\varepsilon^\eta\|_{L^2(B_{1-\delta})})\,.$$
Moreover, by  \eqref{blow_up_properties}(v), \eqref{ineq:H1bound}, and \eqref{ineq:affineboundvepstilde}  we conclude \EEE
$$\int_{B_{1-\delta}} |\nabla \tilde{v}_\varepsilon^\eta|^2\,\mathrm{d}y\leq  C\eta^{-d}\|\nabla v_\varepsilon\|^2_{L^2(B_1)}\leq C\eta^{-d}(\|e(\tilde{ z\EEE}_\varepsilon)\|_{L^2(B_1)}+1)\EEE^2\leq C\eta^{-d}\,.$$
This  \EEE implies \eqref{ineq:H1boundvepstilda} for the piecewise affine interpolations $(\tilde v_{\eps}^{\eta})_{\eps>0}$ on $B_{1-\delta}$. We now apply \eqref{eq: equivalence gradients} in order to obtain \eqref{ineq:H1boundvepstilda} also for $\overline{\nabla}v_\varepsilon^\eta$ on the  slightly smaller ball $B_{1-2\delta}$.

Finally, we can now \EEE conclude this step by showing \eqref{eq:weakconvergencevtilda}. \EEE Indeed, by \eqref{ineq:H1boundvepstilda}, and having already established the measure convergence of $\tilde v_\eps^\eta$ towards $\nabla u(x_0)y$ in $\AAA B_{1-2\delta}\EEE$, we can infer that (up to a subsequence, not relabeled) it holds that  $ \tilde{v}_\varepsilon^\eta \rightharpoonup \nabla u( x_0\EEE)y\EEE$ weakly in $H^1(\AAA B_{1-2\delta}; \R^d\EEE)$.  By \eqref{eq: bernd} we can then also infer the weak convergence of $\overline{\nabla}v_\varepsilon^\eta \rightharpoonup \nabla u(x_0)Z$ in $L^2_{\rm loc}(\AAA B_{1-2\delta}; \R^{d\times 2^d}\EEE)$.\\
\end{step}
\begin{step}{5}(Weak convergence of $\chi_{\AAA V_\eps^{\rm{c, bl}}}\EEE\overline{e}(\zeta_\varepsilon)\EEE$) We are finally in the position to prove \eqref{identification_of_blow_up_limit}. To this end, let $\varphi \in C_c^\infty(\AAA B_{1-2\delta}\EEE;\mathbb{R}^{{d}\EEE\times {2^d}\EEE})$, extended by zero outside of $\AAA B_{1-2\delta}\EEE$, \EEE and write
\begin{align}\label{eq:decompsymmetricgradient}
\hspace{-0.4em}\int_{B_1}\hspace{-0.4em}\left( \chi_{V_\eps^{\rm c, bl \EEE}}\overline{e}( \zeta_\varepsilon)-e(u)(x_0)Z\EEE\right):\varphi
=\int_{B_1}\hspace{-0.4em}\chi_{V_\eps^{\rm c, bl \EEE}} \EEE \left( \overline{e}(\zeta_\varepsilon)-\bar{e}(v_\varepsilon^\eta)\EEE\right):\varphi
+\int_{B_1}\hspace{-0.4em}\left(\chi_{V_\eps^{\rm c, bl \EEE}} \EEE \overline{e}(v_\varepsilon^\eta)-e(u)(x_0)\EEE Z \right):\varphi
\,.
\end{align}
Now by Step 4,  recalling the notation $P$ below \eqref{infinitesimal_translation_rotation} for the orthogonal projection from $\mathbb{R}^{d\times 2^d}$ onto the space orthogonal to $\mathbb{R}_{\mathrm{s,t}}^{d\times 2^d},$ so that $P^\top=P\in \R^{(d\times 2^d)\times (d\times 2^d)}$,  we have that $\overline{e}(v_\varepsilon^\eta) = P(\overline\nabla v_\varepsilon^\eta)$ and  $e(u)Z = P(\nabla uZ)$. Moreover, as $x_0 \in W^0$, it is not restrictive to suppose that  $\chi_{V_\eps^{\rm c, bl}} \to \chi_{B_1}$ strongly in $L^2(B_1)$ by \ref{H1}  and  \eqref{conclusions_of_symmetric_Cauchy-Born}(ii). 
Thus, by the second part of \eqref{eq:weakconvergencevtilda} 
\EEE it holds that     
\begin{align*}
\lim_{\varepsilon\to 0} \int_{B_1}\EEE\left(\chi_{V_\eps^{\rm c, bl}} \overline{e}(v_\varepsilon^\eta)-e(u)(x_0)Z \right):\varphi\,\mathrm{d}y =\lim_{\varepsilon\to 0} \int_{\AAA B_{1-2\delta}\EEE}\left(\chi_{V_\eps^{\rm c, bl}} \EEE\overline{\nabla}v_\varepsilon^\eta-\nabla u(x_0)\EEE Z \right):P\varphi\,\mathrm{d}y =0\,,
\end{align*}
where we used that 
\begin{align*}
\Big|\int_{B_{1-2\delta}}\left(\chi_{V_\eps^{\rm c, bl}} \overline{\nabla}v_\varepsilon^\eta-\overline{\nabla}v_\eps^\eta\right):P\varphi\,\mathrm{d}y\Big|\leq\|\varphi\|_{L^{\infty}(B_1)}\sup_{\eps>0}\|\overline{\nabla}v_\varepsilon^\eta\|_{L^2(B_{1-2\delta})}\cdot\|\chi_{V_\eps^{\rm{c, bl}}}-\chi_{B_1}\|_{L^2(B_1)}\to 0\,.
\end{align*}
\EEE
On the other hand,   we get 
\begin{align*}
\left|\int_{B_1}\chi_{V_\eps^{\rm c, bl}}\left(\overline{e}(\zeta_\varepsilon)-\bar{e}(v_\varepsilon^\eta)\right):\varphi\right| &\leq \|\varphi\|_{L^\infty(B_1)}\sum_{ i\in \mathcal{L}_\varepsilon^\mathrm{good}}\int_{Q_{\hat\varepsilon}(\hat{ i})} \chi_{V_\eps^{\rm c, bl}} |\overline{\nabla} \zeta_\varepsilon - \overline{\nabla}v_\varepsilon^\eta| \\& \ \ \  +\|\varphi\|_{L^\infty(B_1)}\sum_{ i\in \mathcal{L}_\eps\setminus\mathcal{L}_\varepsilon^\mathrm{good}}\int_{Q_{\hat\varepsilon}(\hat{ i})\cap B_{1-2\delta}} \chi_{V_\eps^{\rm c, bl}}|\overline{e}(\zeta_\varepsilon) - \overline{e}(v_\varepsilon^\eta)|\,. \notag
\end{align*}
The first integral on the right hand side can be estimated via Hölder's inequality 
and \eqref{ineq:L2differencegood}, namely
\begin{align}\label{do not remove2}
\begin{split}
\sum_{ i\EEE \in \mathcal{L}_\varepsilon^\mathrm{good}}\int_{Q_{\hat\varepsilon\EEE}(\hat{ i\EEE})} \chi_{V_\eps^{\rm c,  bl}}\EEE |\overline{\nabla} \zeta\EEE_\varepsilon - \overline{\nabla}v_\varepsilon^\eta| 
\leq \L^{d}(B_1)^{\frac{1}{2}}\EEE\Big(\sum_{ i\EEE \in \mathcal{L}_\varepsilon^\mathrm{good}}\int_{Q_{\hat\varepsilon\EEE}(\hat{ i\EEE})} \chi_{V_\eps^{\rm c,  bl}}\EEE|\overline{\nabla}\zeta_\varepsilon-\overline{\nabla}v_\varepsilon^\eta|^2\EEE\Big)^{1/2}\leq C\eta\,.
\end{split}
\end{align}
\EEE For the second integral, 
we can similarly use H\"older's inequality, \EEE 
\EEE \eqref{ineq:badsetestimate}, \eqref{ineq:H1boundvepstilda}, \eqref{blow_up_properties}(vi), and the fact that the domain of integration is nonzero only for $x\in \AAA B_{1-2\delta}\EEE$, 
\EEE to estimate
\begin{align}\label{ineq:secondintegral}
\begin{split}
\sum_{ i \in  \mathcal{L}_\eps\setminus\mathcal{L}_\varepsilon^\mathrm{good}} \EEE &\int_{Q_{\hat\varepsilon\EEE}(\hat{ i\EEE})\cap \AAA B_{1-2\delta}\EEE}  \hspace{-0.2cm} \chi_{V_\eps^{\rm c,bl\EEE}}\EEE | \overline{e}( \zeta\EEE_\varepsilon) - \overline{e}(v_\varepsilon^\eta)|
\\ 
&\leq C(\hat\eps^d\EEE\#(\mathcal{L}_\eps \EEE \setminus  \mathcal{L}_\varepsilon^\mathrm{good}))^{\frac{1}{2}}\left(\|\chi_{V_\eps^{\rm c, bl\EEE}}\EEE\overline{e}( \zeta\EEE_\varepsilon)\|_{L^2(B_1)}+\|\overline{\nabla}v_\varepsilon^\eta\|_{L^2(\AAA B_{1-2\delta}\EEE)} \right) \\
&\leq C\Big(\frac{\sqrt{\L^{ d\EEE}(\omega_\varepsilon)}}{\eta^{{ d\EEE}/2}}+\sqrt{\L^{ d\EEE}({ T\EEE}_\eps)}\Big)\left(\|\chi_{V_\eps^{\rm c, bl\EEE}}\EEE\overline{e}( \zeta\EEE_\varepsilon)\|_{L^2(B_1)}+\|\overline{\nabla}v_\varepsilon^\eta\|_{L^2(\AAA B_{1-2\delta}\EEE)} \right) \\&\leq C\eta^{-d}\left(\L^{ d\EEE}(\omega_\varepsilon)^{1/2}+\L^{ d\EEE}({T}_\varepsilon\EEE)^{1/2}\right)\,.
\end{split}
\end{align}
Now recalling \eqref{ineq:H1bound}, \eqref{ineq:gradientbound}, and  \eqref{blow_up_properties}(iv), \eqref{eq:decompsymmetricgradient}--\eqref{ineq:secondintegral} imply for all $\varphi \in C^\infty_c(\AAA B_{1-2\delta}\EEE; \EEE \R^{d \times 2^d})$ 
$$\limsup_{\eps \to 0} \Big|\int_{B_1}( \chi_{V_\eps^{\rm c, bl\EEE}}\overline{e}( \zeta_\varepsilon)-e(u)(x_0)Z):\varphi  \Big| \le C\eta \Vert \varphi\Vert_{L^\infty(B_1)}.$$
Then, the desired property \eqref{identification_of_blow_up_limit} \EEE follows by \eqref{blow_up_properties}(vi), \EEE   by sending $\eta \to 0$, the arbitrariness of $\delta$, the boundedness of $\chi_{V_\eps^{\rm c, bl\EEE}}\overline{e}( \zeta_\varepsilon)-e(u)(x_0)Z$ in $L^2$, \EEE and the fact that the set $C_c^\infty(B_{1};\R^{d \times 2^d})$ is strongly dense in $L^2(B_1;\R^{d \times 2^d})$. 
	\end{step}
\end{proof}

We finish this section with the proof a geometric fact that we used in the proof of Lemma \ref{weak_limit_for_discrete_multies}.

\begin{lemma}\label{lemma: boundary-fp}
Let $s \in (0,1)$ and let $Q \subset \R^d$ be a cube. Let $P \subset Q$ be a set of finite perimeter with $\mathcal{L}^d(P) \le s \mathcal{L}^d(Q)$. Then, there exists a constant $C>0$ \EEE only depending on $s$ and $d$ such that
$$\mathcal{H}^{d-1}(\partial^*P \cap \partial Q) \le C \mathcal{H}^{d-1}(\partial^*P \cap {\rm int}(Q))\,.$$
\end{lemma}

\begin{proof}
Without loss of generality (after possibly a rigid motion and a rescaling) we can assume that $Q$ is \EEE the (open) unit cube, i.e., $Q= (-\frac{1}{2},\frac{1}{2})^d$. Thus, we suppose that $\mathcal{L}^d(P) \le s$. Moreover, we can assume without restriction that
\begin{align}\label{eq wr}
\mathcal{H}^{d-1}(\partial^*P \cap Q) \le \delta
\end{align}
for a small constant $\delta> 0$ depending only on $s$, to be specified along the proof. Indeed, otherwise the desired inequality clearly holds for a constant depending on $\delta$, as $\mathcal{H}^{d-1}(\partial^*P \cap \partial Q) \le 2d$.
	
We show the statement by induction over the dimension $d$. The statement for $d=2$ can be deduced from \cite[Lemma 4.6]{FriSol} as follows: suppose without restriction that $P$ is connected (more precisely, indecomposable, see e.g.\ \cite[Example~4.18]{Ambrosio-Fusco-Pallara:2000}), as otherwise we argue separately for all indecomposable components. By \cite[Lemma 4.6]{FriSol}, there exists $\delta_{ 0} \in (0,1)$ sufficiently small such that, if $\mathcal{H}^1(\partial^*P\cap Q)\leq \delta_0$, \EEE then  $$\mathcal{L}^2(P) > s \quad \text{or} \quad {\rm diam}(P) \le C_2\mathcal{H}^1(\partial^* P \cap Q)\,,$$ for an absolute constant $C_2 >0$, where  ${\rm diam}(P) := {\rm ess \, sup}\lbrace |x-y|\colon\, x,y \in P\rbrace$. (There, the proof has been performed for $s= \frac{1}{2}$, but it holds true for any $s \in (0,1)$, provided that $\delta_{0}$ is chosen sufficiently small depending on $s$.) \EEE Since $Q$ is a cube, we get $\mathcal{H}^{1}(\partial^*P \cap \partial Q)  \le C{\rm diam}(P)$, and since $\mathcal{L}^2(P) \le s $ by assumption, we deduce 
$$
\mathcal{H}^{1}(\partial^*P \cap \partial Q) \le  C\,{\rm diam}(P) \le C\mathcal{H}^1(\partial^* P \cap Q). $$
This concludes the proof for $d=2$. Let us now assume for the inductive hypothesis that the statement holds true for some $d \ge 2$, and let us prove it also for $d+1$. We will only show that
\begin{align}\label{eq wr2}
\mathcal{H}^{d-1}\big(\partial^*P \cap \partial Q \cap  \lbrace -1/2 < x_{d+1} < 1/2 \rbrace\big) \le C \mathcal{H}^{d-1}(\partial^*P \cap Q)\,,
\end{align}
where $x := (x',x_{d+1}) \in Q$, since the surface measure in the two remaining faces of $\partial Q$ can be controlled by repeating the argument after a rotation of the cube. We claim that $\mathcal{L}^{d+1}(P) \le s$ implies that \EEE for  $s' := \frac{s+1}{2} \in (0,1)$,  
\begin{align}\label{eq wr3}
\mathcal{L}^d(P \cap Q^d_t) \le s'
\end{align}
for $\mathcal{L}^1$-a.e.\ $t \in (-\frac{1}{2},\frac{1}{2})$, where $Q^{d}_{t}:=(-\tfrac{1}{2},\tfrac{1}{2})^{d}\times\{t\}$. In fact, we define $\delta:=\delta(s):= s' - s=\frac{1-s}{2}$, and suppose that \eqref{eq wr} holds. Then,  for a.e.\ $-\frac{1}{2} < t_1 < t_2 < \frac{1}{2}$, we have
$$\big|\mathcal{L}^d(P \cap Q^d_{t_2}) -  \mathcal{L}^d(P \cap Q^d_{t_1}) \big|  \le \delta\,,
$$
since otherwise by the coarea formula we find
\begin{align*}
\mathcal{H}^{d-1}\big( \partial^* P \cap \big((-\tfrac{1}{2},\tfrac{1}{2})^d \times [t_1,t_2]\big)  \big) \ge \big|\mathcal{L}^d(P \cap Q^d_{t_2}) -  \mathcal{L}^d(P \cap Q^d_{t_1}) \big|  > \delta\EEE\,,
\end{align*}
which contradicts our assumption \eqref{eq wr}. This observation along with $\mathcal{L}^{d+1}(P) \le s$ implies   for $\mathcal{L}^1$-a.e.\ $t \in (-\frac{1}{2},\frac{1}{2})$ that
\begin{align*}
\mathcal{L}^d(P \cap Q^d_t) \le s +\delta
= s'.
\end{align*} 
This shows \eqref{eq wr3}. By slicing properties of sets of finite perimeter we have that for a.e.\ $t\in (-\tfrac{1}{2},\frac{1}{2})$ the set  $P_t:=P\cap Q^{d}_{t}$ has finite perimeter. Thus, by the inductive hypothesis, we can find a constant $C>0$ depending on $d$ and $s'$ such that  
\begin{equation}\label{inductive_hypothesis}
\H^{d-1}(\partial^*P_t\cap \partial Q^{d}_{t})\leq C\H^{d-1}(\partial^*P_t\cap Q^{d}_{t})
\end{equation}
for a.e.\ $t \in (-\frac{1}{2},\frac{1}{2})$. By applying the coarea formula (cf.\ \cite[(18.25)]{maggi2012sets} for $g:=\chi_{Q}$ and slicing direction $e_{d+1}$), denoting by $\nu$ the generalized normal to $\partial^\ast P$, and using \eqref{inductive_hypothesis}, we can estimate
\begin{align*}
\int_{-1/2}^{1/2} \mathcal{H}^{d-1}(\partial^* P_t \cap Q^d_t) \, {\rm d}t & = \int_{-1/2}^{1/2}\mathrm{d}t\int_{{\partial^\ast P\cap\{x_{d+1}=t\}}}\chi_{Q}\, \mathrm{d}\H^{d-1}\\
& = \int_{\partial^\ast P}\chi_{Q}\sqrt{1-(\nu\cdot e_{d+1})^2}\, \mathrm{d}\H^d \le  \H^d(\partial^\ast P \cap Q)\,. 
\end{align*}
In a similar fashion, the coarea formula also implies that        
$$\mathcal{H}^{d\EEE}\big(\partial^*P \cap \partial Q \cap  \lbrace -1/2 < x_{d+1} < 1/2 \rbrace\big) = \int_{-1/2}^{1/2} \mathcal{H}^{d-1}(\partial^* P_t \cap \partial Q^d_t) \, {\rm d}t.$$
	These two estimates along with \eqref{inductive_hypothesis} conclude the proof of     \eqref{eq wr2}.
\end{proof}

\section{Lower Bound}\label{sec: 5}
In this section we return to the case $d=3$ and focus on the proof of the $\Gamma$-liminf inequality of Theorem \ref{Main G-convergence theorem}. In particular, we aim at proving the following result.  
\begin{proposition}\label{proposition:liminf} 
Let $(u_\varepsilon,E_\varepsilon)_{\varepsilon>0}$, where $E_\varepsilon\subset\mathbb{Z}_{\varepsilon}(\Omega)$ and $u_\varepsilon \in \mathcal{U}_\varepsilon(u_0,\partial_D\Omega,E_\eps)\EEE$, be such that $(u_\varepsilon,E_\varepsilon) \overset{\mathrm{d}}{\to} (u,E)$ as $\varepsilon\to 0$, in the sense of Definition~\ref{def:convergence}. Then,
\begin{align*}
\liminf_{\varepsilon \to 0}\mathcal{F}_\varepsilon(u_\varepsilon,E_\varepsilon) \geq \mathcal{F}_0(u,E)\,.
\end{align*} 
\end{proposition}
We split the proof of this proposition into two separate lemmata, which prove the $\Gamma$-liminf inequality for the surface part and the elastic part of the energy separately. It is not restrictive to assume that $\liminf_{\varepsilon \to 0}\mathcal{F}_\varepsilon(u_\varepsilon,E_\varepsilon)<+\infty$, since otherwise there is  nothing to prove. For the next lemmata, recall the definitions in \eqref{def: discrete elastic energy} and \eqref{eq: bdaypart}--\eqref{def:contsurf}.


\begin{lemma}\label{lemma:liminfsurf} Let $(u_\varepsilon,E_\varepsilon)_{\varepsilon>0}$ be such that $(u _\varepsilon,E_\varepsilon) \overset{\mathrm{d}}{\to} (u,E)$ as $\varepsilon\to 0$, in the sense of Definition~\ref{def:convergence}. Then,
\begin{align*}
\liminf_{\varepsilon \to 0}\left(F_\varepsilon^\mathrm{\mathrm{per}}(E_\varepsilon)+ F_\varepsilon^\mathrm{curv}(
E_\varepsilon)\right) \geq \mathcal{F}^\mathrm{surf}(u,E)+\mathcal{F}^{\rm \partial_D}(u,E)\,.
\end{align*} 
\end{lemma}
\begin{proof} Let $(u_\varepsilon,E_\varepsilon)_{\varepsilon>0}$ be a subsequence (not relabeled) that achieves the liminf. Adopting all the notations that have been introduced so far, the desired inequality follows directly from the outcomes of our compactness  result, i.e., Theorem~\ref{theorem:compactness} together with \cite[Remark 3.9]{KFZ:2021}. Indeed, we use \cite[Equation (3.28)]{KFZ:2021} along with \eqref{eq: bdaypart} and \eqref{def:contsurf}, to get $$\mathcal{F}^\mathrm{surf}(u,E)+\mathcal{F}^{\rm \partial_D}(u,E) \le  \liminf_{\varepsilon \to 0}  \mathcal{F}_{\mathrm{surf}}^{\varphi,\hat{\gamma}_{\varepsilon},q}(\hat{E}_{\varepsilon},U).$$ This, combined with \eqref{eq: fflat} concludes the proof. 
\end{proof}
\begin{lemma}\label{lemma:elastic liminf} Let $(u_\varepsilon,E_\varepsilon)_{\varepsilon>0}$, where $E_\varepsilon\subset\mathbb{Z}_{\varepsilon}(\Omega)$ and $u_\varepsilon \in \mathcal{U}_\varepsilon(u_0,\partial_D\Omega,E_\eps)\EEE$, be such that $(u_\varepsilon,E_\varepsilon) \overset{\mathrm{d}}{\to} (u,E)$ as $\varepsilon\to 0$, in the sense of Definition \ref{def:convergence}. Then,
\begin{align*}
\liminf_{\varepsilon \to 0} F_\varepsilon^{\mathrm{el}}(\mathrm{id}+\delta_\varepsilon u_\varepsilon,E_\varepsilon) \geq \mathcal{F}^\mathrm{el}(u,E)\,.
\end{align*}
\end{lemma}
\begin{proof}
The proof is a combination of the arguments for the $\Gamma$-liminf inequality in \cite{Schmidt:2009} and \cite{KFZ:2021}, respectively. \EEE 
By our compactness result, we can suppose that \eqref{disc_cont_est_combined}--\eqref{rate_kappa_epsilon} hold, up to passing to a subsequence. We can apply Theorem \ref{discrete_Cauchy_Born_any_d} for the sets $W_\eps$ introduced in \eqref{eq: tildeVdef} to get $V_\eps \supset W_\eps$ satisfying \eqref{conclusions_of_symmetric_Cauchy-Born}.  \EEE

As a preparation, we need that \eqref{eq:lsc0}(iii) also holds for the discrete gradients in place of $\nabla \tilde{u}_\eps$. To see this, we first observe that the fact that $\tilde{u}_\eps$ is affine on each $S \in \mathcal{S}_{\varepsilon,Q\EEE}$ (cf.\ \eqref{def:simplices}) \EEE implies that there exists an absolute constant $C>0$ such that 
$${\eps^3\EEE\# \hat\Q\EEE_\eps^{\kappa_\eps\EEE} \le C\sum_{Q \in   \hat{\Q}_{\eps, V_\eps^{\rm c}\EEE}} \mathcal{L}^3(Q \cap \lbrace |\EEE\nabla \tilde{u}_\eps|\EEE > \kappa_\eps\rbrace), \quad \text{where } \hat\Q\EEE_\eps^{\kappa_\eps\EEE} :=  \lbrace Q \in   \hat{\Q}_{\eps,{V_\eps^{\rm c}}\EEE} \EEE \colon \, \Vert \nabla \tilde{u}_\eps \Vert_{L^\infty(Q)} > \kappa_\eps \rbrace\,.}$$ 
For each $Q \in \hat{\Q}_{\eps, V_\eps^{\rm c}\EEE}\EEE$,  by \eqref{conclusions_of_symmetric_Cauchy-Born}{(i)}\EEE, the fact that $\tilde{u}_\eps$ is piecewise affine, and $\omega_u = \EEE \hat{\omega}_u\cap \Omega \EEE \subset W_\eps \EEE \subset V_\eps$ for all $\eps$, 
we deduce that
\begin{align*}
\mathcal{L}^3(Q \cap  \lbrace  |\nabla \tilde{u}_\eps| > \kappa_\eps \rbrace) \le 2 \mathcal{L}^3( (Q \setminus V_\eps) \cap  \lbrace  |\nabla \tilde{u}_\eps| > \kappa_\eps \rbrace)  \le 2 \mathcal{L}^3( (Q \setminus \hat{\omega}_u) \cap  \lbrace  |\nabla \tilde{u}_\eps| > \kappa_\eps \rbrace)\,.  
\end{align*}
By \eqref{eq: equivalence gradients}, \EEE 
there exists an absolute constant $C_*\ge 1$ such that for all $Q \in \hat{\Q}_{\eps, V_\eps^{\rm c}} \EEE \setminus \hat Q\EEE_\eps^{\kappa_\eps\EEE}$ it holds that 
\begin{equation*}
\big|\overline{\nabla}u_\eps|_{Q}\big|^2\leq C_*\fint_{Q}|\nabla\tilde u_\eps|^2\, \mathrm{d}x\leq C_*\kappa_\eps^2 \leq C_*^2\kappa_\eps^2\,. \EEE
\end{equation*}
Combining the previous estimates, we conclude
\begin{align*}
\mathcal{L}^3(\lbrace |\overline{\nabla} u_\eps| > C_*\kappa_\eps \rbrace \setminus V_\eps  ) \le \eps^{3\EEE} \# \hat\Q\EEE_\eps^{\kappa_\eps} \EEE \le C \mathcal{L}^3(   \lbrace  |\nabla \tilde{u}_\eps| > \kappa_\eps \rbrace \setminus \hat{\omega}_u). 
\end{align*}
This along with \eqref{eq:lsc0}(iii) shows that
\begin{equation}\label{small difference for_measures_new_kappa est}
\mathcal{L}^3(\lbrace |\overline{\nabla} u_\eps| > C_*\kappa_\eps \rbrace \setminus V_\eps  ) \to 0 \quad \text{as $\eps\to0$}\,. 
\end{equation}
Now, \EEE let us consider an arbitrary open subset $\tilde \Omega\subset \subset \Omega$, and $\varepsilon_0:=\varepsilon_0(\tilde\Omega,\Omega)\in (0,1)$ small enough so that 
\AAA for all $\varepsilon\in (0,\varepsilon_0]$ we have 
\begin{align}\label{eq: Vcube-def}
{\EEE(V_\eps^{\rm{c}}
)_{\rm{cub}}\EEE:=\bigcup_{Q\in  \hat\Q_{\eps, \tilde \Omega\setminus V_\eps  }} Q\subset\subset \Omega, \quad \text{and \ }  {\rm dist}\big((V_\eps^{\rm{c}})_{\rm{cub}}, E_\varepsilon\big) \ge 2\eps\,.}
\end{align}
\EEE Indeed, the second property follows from the last property in \eqref{disc_cont_est_combined}, 
the fact that  $V_\eps \supset W_\eps
\supset E_\eps^*  
\cap \Omega\EEE \supset \hat{E}_\eps 
\cap \Omega\EEE$ (see \eqref{eq: tildeVdef} and before \eqref{eq:lsc0}), \AAA and that $\eps/\eta_\eps\to 0$\,. \EEE  Thus, recalling the notation in \eqref{def_set_of_neighbours}, we have
\begin{equation*}
\mathcal{I}_\eps(i,E_\varepsilon) \EEE =\{1,\dots,8 \} \quad \forall\ i\in \mathbb{Z}_{\varepsilon}\big((V_\eps^{\rm{c}})_{\rm{cub}}\big)
\,.
\end{equation*}
Hence, by \eqref{discrete_gradient_def} 
and \EEE \eqref{definition_elastic_cell_energy}, for $\varepsilon \in (0,\varepsilon_0]$ we can estimate
\begin{align}\label{elastic_lower_bd_1}
F^{\mathrm{el}}_\varepsilon(\mathrm{id}+\delta_\varepsilon u_\varepsilon,E_\varepsilon)&\geq \delta_\varepsilon^{-2}\sum_{i\in \mathbb{Z}_{\varepsilon}((V_\eps^{\rm{c}})_{\rm{cub}})}\varepsilon^3W^{\mathrm{el}}_{\mathrm{bulk}}\big(Z+\delta_\varepsilon\overline{\nabla}u_\varepsilon(i)\big)\,.
\end{align}
As a result of the minimality and regularity assumptions (ii) and (iii) in the definition of the discrete elastic cell energy (cf.\ Subsection \hyperref[Assumptions on the elastic energy]{2.2}), we have that $W_{\mathrm{bulk}}^{\mathrm{el}}(Z)=0$, $DW_{\mathrm{bulk}}^{\mathrm{el}}(Z)=0$, and for every $F\in \mathbb{R}^{3\times 8}$ we can expand
\begin{equation}\label{expansion_elastic_energy_comp}
W_{\mathrm{bulk}}^{\mathrm{el}}(Z+F)=\frac{1}{2}Q_{\mathrm{bulk}}(F)+\Phi(F)\,,
\end{equation}
where we recall that $Q_{\mathrm{bulk}}:=D^2 W_{\mathrm{bulk}}^{\mathrm{el}}(Z)$ and $\Phi\colon\R^{3\times 8}\to  \R $ satisfies $|\Phi(F)|  \le C|F|^3$  for all $F \in \R^{3\times 8}$ with $|F| \le 1$. Again in view of 
assumption (iii) \EEE  for the elastic cell energy, and recalling the definition of $\bar{e}(u_\varepsilon)$ as the projection of $\overline{\nabla} u_\varepsilon$ on the orthogonal complement of $\R^{3\times 8}_{\mathrm{s,t}}$ (cf.~\eqref{infinitesimal_translation_rotation}), we have $Q_{\mathrm{bulk}}(\overline{\nabla}u_\varepsilon)=Q_{\mathrm{bulk}}(\bar{e}(u_\varepsilon))$. Let us also define $\vartheta_\eps = \chi_{ [0, \EEE C_*\kappa_\eps]}(|\overline{\nabla} u_\eps|)$ 
with $C_*$ from \eqref{small difference for_measures_new_kappa est}. \EEE Having identified $\overline{\nabla}u_\varepsilon$ with its piecewise constant interpolation, being equal to $\overline{\nabla}u_\varepsilon(i)$ on each $Q_\varepsilon(\hat{i})$, by 
\eqref{eq: Vcube-def}, \EEE \eqref{elastic_lower_bd_1}, \eqref{expansion_elastic_energy_comp},  and the fact that $\delta_\eps \kappa_\eps \to 0$ (see \eqref{rate_kappa_epsilon}), \EEE we infer that
\begin{align*}
\begin{split}
F^{\mathrm{el}}_\varepsilon(\mathrm{id}+\delta_\varepsilon u_\varepsilon,E_\varepsilon)&\geq \sum_{i\in \mathbb{Z}_{\varepsilon}((V_\eps^{\rm{c}})_{\rm{cub}})}\int_{Q_{\varepsilon}(\hat{i})} \vartheta_\eps \EEE\left(\frac{1}{2}Q_{\mathrm{bulk}}(\bar{e}(u_\varepsilon))+\delta_\varepsilon^{-2}\Phi(\delta_\eps\overline{\nabla}u_\varepsilon)\right)\, \mathrm{d}x\\
&=\int_{(V_\eps^{\rm{c}})_{\rm{cub}}} \vartheta_\eps \EEE \left(\frac{1}{2}Q_{\mathrm{bulk}}(\bar{e}(u_\varepsilon))+\delta_\varepsilon^{-2}\Phi(\delta_\eps\overline{\nabla}{u}_\varepsilon)\right)\, \mathrm{d}x\\
&\geq \frac{1}{2}\int_{\tilde \Omega} \vartheta_\eps Q_{\mathrm{bulk}}\big(\chi_{\Omega\setminus V_\eps}\bar{e}(u_\varepsilon)\big)\, \mathrm{d}x+\delta_{\varepsilon}^{-2}\int_{\tilde\Omega\setminus V_\eps}\vartheta_\eps \EEE\Phi(\delta_\eps\overline{\nabla}u_\varepsilon)\, \mathrm{d}x\\
&\geq \frac{1}{2}\int_{\tilde\Omega} \vartheta_\eps Q_{\mathrm{bulk}}\big(\chi_{\Omega\setminus V_\eps}\bar{e}(u_\varepsilon)\big)\, \mathrm{d}x-C\delta_{\varepsilon}\kappa_\eps^3\,.
\end{split}
\end{align*}
By \eqref{small difference for_measures_new_kappa est}, and $\chi_{\Omega\setminus V_\eps}\to\chi_{\Omega\setminus (E\cup\omega_u)}$ (cf.\ \eqref{eq: tildeVdef}, \eqref{eq:lsc0}(v), and \eqref{conclusions_of_symmetric_Cauchy-Born}{(ii)}), we have that $\vartheta_\eps\chi_{\Omega\setminus V_\eps}\to\chi_{\Omega\setminus (E\cup\omega_u)}$ boundedly in measure. This observation, together with \eqref{rate_kappa_epsilon}, the Cauchy-Born rule \EEE (i.e., the weak convergence property in \eqref{conclusions_of_symmetric_Cauchy-Born}{(iii)}), the convexity of $Q_{\mathrm{bulk}}$, and \EEE the fact that $u \equiv 0$ on $\omega_u$, yields $$\liminf_{\eps \to 0 } F^{\mathrm{el}}_\varepsilon(\mathrm{id}+\delta_\varepsilon u_\varepsilon,E_\varepsilon)  \ge  \frac{1}{2}\int_{\tilde\Omega\setminus E} Q_\mathrm{\mathrm{bulk}}(e(u)Z)\,\mathrm{d}x\,. $$
Recalling the definition of the elastic energy in \eqref{def:contel}, and  exhausting $\Omega$ with compactly supported open subsets $\tilde \Omega\subset\subset\Omega$, the statement follows. 
\end{proof}

\section{Upper Bound}\label{sec: 6}
This final section is devoted to the proof of the $\Gamma$-limsup inequality. 
\begin{proposition}\label{prop:limsup} Let $u=\chi_{\Omega\setminus E}u \in GSBD^2(\Omega)$ and $E\in \mathfrak{M}(\Omega)$ be a set of finite perimeter. Then, there exists $(u_\varepsilon,E_\varepsilon)_{\varepsilon>0}$, $E_\varepsilon\subset\mathbb{Z}_{\varepsilon}(\Omega)$ and $u_\varepsilon \in \mathcal{U}_\varepsilon(u_0,\partial_D\Omega,E_\eps)$, such that $(u_\varepsilon,E_\varepsilon) \overset{\mathrm{d}}{\to} (u,E)$ in the sense of Definition \ref{def:convergence} and
\begin{align}\label{prop ineq: limsup}
\limsup_{\varepsilon \to 0}\mathcal{F}_\varepsilon(u_\varepsilon,E_\varepsilon) \leq \mathcal{F}_0(u,E)\,.
\end{align} 
\end{proposition}
\begin{proof} 
We proceed in several steps. We first reduce to showing \eqref{prop ineq: limsup} for pairs of regular sets and displacement maps. Then, we replace these sets with polyhedral sets whose outer unit normal (at the points where this is defined) is a coordinate vector. Subsequently, we perform the approximation on the discrete level. The last step shows the convergence of the discrete energies to the respective continuum one.\\
\begin{step}{1}(Density of regular sets and displacements)
By a density result, 
we \EEE can assume that $\partial E\cap\Omega\in C^\infty
$ and that $u \in W^{1,\infty}(\Omega \setminus \overline E;\mathbb{R}^3)$. 
Indeed, \EEE by \cite[Theorem~2.2]{Crismale}, for each $E\in\M(\Omega)$ with $\mathcal{H}^{2}(\partial^* E \cap \Omega)<+\infty$ and each $u=\chi_{\Omega \setminus E}u\in GSBD^2(\Omega)$, there exists a sequence of  sets $(E_k)_{k\in \mathbb{N}}$ with $\partial E_k\cap\Omega\in C^{\infty}$,  $\chi_{E_k} \to \chi_E$ in $L^1(\Omega)$ and a sequence  $(u_k)_{k\in \mathbb{N}}$ with $u_k|_{\Omega\setminus \overline{E_k}} \in W^{1,\infty}(\Omega \setminus \overline{E_k};\R^3)$, $u_k|_{E_k} = 0$, and ${\rm tr}(u_k) = {\rm tr}(u_0)$ on $\partial_D \Omega \setminus \overline{E_k}$, such that $u_k \to u$ in measure on $\Omega$  and
\begin{align}\label{eq: limsupVito1}
\lim_{k \to \infty} \Big(\frac{1}{2}  \int_{\Omega \setminus \overline {E_k}}  Q_{\rm bulk}(e(u_k)Z)\, {\rm d}x + \int_{\partial E_k\cap(\Omega \cup\partial_D \Omega)}  \varphi(\nu_{E_k}) \, {\rm d}\mathcal{H}^{2}  \Big)  = \mathcal{F}_0(u,E)\,.
\end{align}
Strictly speaking, \cite[Theorem 2.2]{Crismale} only ensures that $\partial E_k
$ is Lipschitz and $u_k$ is in $H^1$, but in the proof it is shown that $\partial E_k
$ can be chosen of class $C^\infty$, see \cite[Proposition 5.4]{Crismale}, and $u_k$ can be approximated by smooth functions by standard approximation arguments. From now on, we therefore assume that $\partial E\cap \Omega \in C^\infty$ and that $u \in W^{1,\infty}(\Omega \setminus \overline E;\mathbb{R}^3)$.  
\end{step}\\
\begin{step}{2}(Reduction to polyhedral void sets with coordinate normals) By the regularity of $E$, the continuity of $\varphi$ (cf.\ \eqref{def:varphi}), and a covering argument, we can reduce to the case of $\partial E$ being polyhedral. Indeed, up to a localization argument,  we can consider a sequence $(E_k)_{k\in \mathbb{N}} \subset \subset E$ with polyhedral boundary and, for each $k\in\mathbb{N}$, extend $u$ to a function $u_k\in W^{1,\infty}(\Omega \setminus \overline E_k;\mathbb{R}^3)$ such that convergence in \eqref{eq: limsupVito1} holds. As ${\rm tr}(u) = {\rm tr}(u_0)$ on $\partial_D \Omega \setminus \overline{E}$ and $E_k \subset \subset \Omega$, this can be achieved in such a way that ${\rm tr}(u_k) = {\rm tr}(u_0)$ on $\partial_D \Omega$. Assuming now that $E \subset \subset \Omega$ has polyhedral boundary and ${\rm tr}(u) = {\rm tr}(u_0)$ on $\partial_D \Omega$,  
we \EEE prove that we can further restrict to the case of $\nu_E \in \{\pm e_1,\pm e_2,\pm e_3\}$. To this end, we fix $\sigma >0$ and define
\begin{align}\label{def:Esigma}
E_\sigma =  \bigcup_{i \in \sigma\mathbb{Z}^3,\ Q_\sigma(i) \subset E}Q_\sigma(i)\,.
\end{align}
 It is obvious that $\mathcal{H}^2$-a.e.\ it holds that  $\nu_{E_\sigma}\in \{\pm e_1,\pm e_2,\pm e_3\}$. First, it is elementary to check that $E_\sigma \subset E$ and $\chi\EEE_{E_\sigma} \to \chi\EEE_{E}$ in $L^1(\Omega)$  as $\sigma\to 0$, since 
\begin{align}\label{eq: inclusi}
E \setminus E_\sigma \subset \lbrace x \in E \colon \, {\rm dist}(x,\partial E) \le \sqrt{3}\sigma \rbrace\,.
\end{align}
 Therefore, it is  clear that
\begin{align}\label{eq: incluspartial}
\partial E_\sigma \subset (\partial E)_{2\sigma}\,.
\end{align}

We now fix an arbitrary face $S \subset \partial E$ with outer \EEE normal vector $\nu$. We show that  
\begin{align}\label{ineq:limsupcoordinatenormal}
\limsup_{\sigma \to 0} \int_{\partial E_\sigma \cap (S)_{2\sigma}} \varphi(\nu_{E_\sigma})\,\mathrm{d}\mathcal{H}^2 \leq  \int_S \varphi(\nu)\,\mathrm{d}\mathcal{H}^2 = \varphi(\nu)\mathcal{H}^2(S)\,. \EEE
\end{align}
Without restriction we assume that $\nu_k\geq 0$ for all $k=1,2,3$.  \BBB For convenience, we define $$T_{k}:=\{y\in \partial E_\sigma\cap(S)_{ 2\sigma}\colon \nu_{E_\sigma}(y)=e_k\}\,$$ 
and note that  by the definition of $E_\sigma$ in \eqref{def:Esigma}, due to the cubes at the boundary of the face,  we have 
\begin{align}\label{eq: smallli}
\mathcal{H}^2\big( (\partial E_\sigma\cap(S)_{ 2\sigma}) \setminus (T_1 \cup T_2 \cup T_3)\big) \le  C_S \sigma
\end{align}
for a constant $C_S>0$ depending only on $S$.   \EEE   
We claim that for all $k=1,2,3$ we have
\begin{align}\label{ineq:cardcubeslimsup}
\mathcal{H}^2(T_k)\leq \nu_k \mathcal{H}^2(S) + C_S\sigma\,.
\end{align}
Let us denote by $\Pi_k$ the orthogonal projection of $\R^3$ onto the subspace orthogonal to $e_k$. Fix  $y'\in \Pi_k((S)_{ 2\sigma})$. Then, for $\sigma>0$ sufficiently small,  by the definition of $E_\sigma$ in \eqref{def:Esigma} it follows that there exists at most one $t\in \R$ for which $y:=y'+te_k$ belongs to $\partial E_\sigma\cap   (S)_{ 2\sigma} $ and is such that $\nu_{E_\sigma}(y)=e_k$. Therefore, the coarea formula implies 
\begin{align*}
\mathcal{H}^2(\{y \in \partial E_\sigma \cap(S)_{2\sigma}\colon \nu_{E_\sigma}(y) =e_k\}) \leq \mathcal{H}^2(\Pi_k((S)_{ 2\sigma}))\leq \nu_k\mathcal{H}^2(S) +C_S\sigma\,.
\end{align*}
This shows  \eqref{ineq:cardcubeslimsup}. Now, due to \eqref{def:varphi}, by our assumption that $\nu_k\geq 0$ for all $k=1,2,3$, we have that $\varphi(\nu) =  \sum_{k=1}^3c_k\nu_k$, where $c_k=\sum_{\xi\in V, \xi_k>0}c_\xi$. By using \eqref{eq: smallli} and \eqref{ineq:cardcubeslimsup} we have  
\begin{align*}
\int_{\partial E_\sigma \cap(S)_{ 2\sigma} }\varphi(\nu_{E_\sigma})\,\mathrm{d}\mathcal{H}^2  \le  \sum_{k=1}^3 \int_{T_k} \varphi(e_k)\,\mathrm{d}\mathcal{H}^2 + C\sigma =\sum_{k=1}^3c_k\mathcal{H}^2(T_k)+ C\sigma\leq \varphi(\nu)\mathcal{H}^2(S) + C\sigma \,,
\end{align*}
where $C>0$ here depends on $\{c_{\xi}\}_{\xi\in V}$ and $S$. This shows \eqref{ineq:limsupcoordinatenormal}. Now, summing the above estimate over all faces $S$ of $\partial E$, \EEE and recalling \eqref{eq: incluspartial}, we find that $\limsup_{\sigma \to 0} \mathrm{Per}_\varphi(E_\sigma) \le \mathrm{Per}_\varphi(E)$. Similarly as before, we can also extend $u$ to a function $u_\sigma \in  W^{1,\infty}(\Omega \setminus \overline {E_\sigma};\mathbb{R}^3)$ with ${\rm tr}(u) = {\rm tr}(u_0)$ on $\partial_D \Omega$  without affecting the elastic energy as $\sigma \to 0$. Summarizing, from now on we can assume that $E \subset \subset \Omega$ is a polyhedral set and $\partial E$ has coordinate normals, and $u \in  W^{1,\infty}(\Omega \setminus \overline E;\mathbb{R}^3)$ with ${\rm tr}(u) = {\rm tr}(u_0)$ on $\partial_D \Omega$. \EEE
\end{step}\\
\begin{step}{3}(Construction of a recovery sequence) 
We are now in the position to construct the recovery sequence. We set 
\begin{align}\label{def:recoverysequence}
E_\varepsilon := \bigcup_{i \in \eta_\varepsilon\mathbb{Z}^3\cap E} \mathbb{Z}_\varepsilon(Q_{\eta_\varepsilon}(i))\cap \Omega\,.
\end{align}
We show that $\chi_{Q_\varepsilon(E_\varepsilon)\cap\Omega}\to \chi_E$ in $L^1(\Omega)$ as $\varepsilon \to 0$. Note that for $\eps>0$ small enough (and thus also $\eta_\varepsilon>0$ small accordingly), \EEE we have that
$${(Q_\varepsilon(E_\varepsilon)\cap\Omega)\triangle E=Q_\varepsilon(E_\varepsilon)\triangle E\EEE\subset (\partial E)_{\BBB \sqrt{3}(\eta_{\varepsilon}+\eps)/2\EEE}\,.}
$$ 
In fact, in this case, \EEE $x \in E \setminus 
Q_\varepsilon(E_\varepsilon)\EEE$ only if $i_x \notin E$, where $i_x \in \eta_\varepsilon\mathbb{Z}^3$ is such that $x\in Q_{\eta_\varepsilon+\eps\EEE}(i_x)$. Therefore, $\mathrm{dist}(x,\partial E) \leq \BBB\sqrt{3}(\eta_{\varepsilon}+\eps)/2\EEE$. Similarly, if $x\in Q_\varepsilon(E_\varepsilon)\EEE\setminus E$, then again $\mathrm{dist}(x,\partial E) \leq \BBB\sqrt{3}(\eta_{\varepsilon}+\eps)/2\EEE$.   
Hence, \EEE for $\varepsilon>0$ sufficiently small, we have that
\begin{align*}
\mathcal{L}^3((Q_\varepsilon(E_\varepsilon)\cap\Omega)\triangle E) \leq  \mathcal{L}^3((\partial E)_{\BBB\sqrt{3}(\eta_{\varepsilon}+\eps)/2\EEE}) \leq C\eta_\varepsilon\mathcal{H}^2(\partial E)\,, 
\end{align*}
where $C>0$ depends on $E$. Since $\eta_\varepsilon\to 0$ as $\varepsilon\to 0$, we deduce that $\chi_{Q_\varepsilon(E_\varepsilon)\cap \Omega}\to \chi_E$ in $L^1(\Omega)$. Moreover, we define a sequence of discrete displacements $u_\varepsilon \in \mathcal{U}_\varepsilon(u_0,\partial_D\Omega,E_\eps)$ via
$${u_\varepsilon(i):=\fint_{Q_\eps(i)} u(y) \, {\rm d}y \quad \text{for every \ } i\in \mathbb{Z}_\varepsilon(\Omega\EEE)\setminus E_\eps\EEE
\,,}$$
 $u_{\varepsilon}(i)=u_0(i)$ for every $i\in \mathbb{Z}_\varepsilon(U\setminus \Omega\EEE
)$, and $u_{\varepsilon}(i)=0$ for every $ i\in   \big( \mathbb{Z}_\varepsilon( \R^3 ) \setminus  \mathbb{Z}_\varepsilon(U\EEE
\big)  \cup E_\eps$. By the regularity of $u$ we can check that the piecewise constant interpolations $\bar{u}_\eps$ (being equal to $\bar{u}_{\varepsilon}(i)$ on each cube $Q_{\varepsilon}(\hat{i})$, see \eqref{discrete_gradient_def} and the comments below) are such that $\bar{u}_\eps \to u$ in measure on $\Omega$. We refer to \cite[Lemma 4.4]{Schmidt:2009} for details. Therefore, $(u_\varepsilon,E_\varepsilon) \overset{\mathrm{d}}{\to} (u,E)$ in the sense of Definition \ref{def:convergence} (with $\omega_u=\emptyset$ in this case).\EEE \end{step}\\ 
\begin{step}{4}(Energy convergence)
The upper bound for the elastic energy, namely  
$$\limsup_{\varepsilon\to 0} F_\varepsilon^{\mathrm{el}}(\mathrm{id}+\delta_\varepsilon u_\varepsilon, E_\varepsilon)\leq \frac{1}{2}\int_{\Omega\setminus E} Q_\mathrm{\mathrm{bulk}}(e(u)Z)\,\mathrm{d}x\,,$$
follows from the compatibility of $W^{\mathrm{el}}_{\mathrm{surf}}$ and $W^{\mathrm{el}}_{\mathrm{bulk}}$ (cf.\ assumption (vi) in the definition of the discrete elastic cell energy in Subsection \ref{Assumptions on the elastic energy}) and the assumed 
a priori \EEE Lipschitz bound on $u$, see \cite[Lemma 4.4]{Schmidt:2009}  and the final argument in the proof of \cite[Theorem 2.6]{Schmidt:2009}. 
In view of \eqref{eq: bdaypart}--\eqref{def:F}, by $E \subset \subset \Omega$, $u \in  W^{1,\infty}(\Omega \setminus \overline E;\mathbb{R}^3)$, and   ${\rm tr}(u) = {\rm tr}(u_0)$ on $\partial_D \Omega$, it therefore suffices to prove that
\begin{align}\label{ineq:limsupeta}
\limsup_{\varepsilon\to 0}\left( F_\varepsilon^\mathrm{per}(E_\varepsilon) +F_\varepsilon^\mathrm{curv}(E_\varepsilon) \right)\leq \mathrm{Per}_\varphi(E)\,.
\end{align}
We split the proof into two estimates. First, we show that
\begin{align}\label{ineq:limsupetasurf}
\limsup_{\varepsilon\to 0}F_\varepsilon^\mathrm{per}(E_\varepsilon) \leq \mathrm{Per}_{\varphi}(E) \,.
\end{align}
We prove this inequality locally.  Let $S \subset \partial Q_\eps(E_\varepsilon) \EEE$ be a face of a cube $Q_{\eta_\varepsilon}(i)+(-\frac{\eps}{2},-\frac{\eps}{2},-\frac{\eps}{2})$ and 
let $\nu \in \{\pm e_1, \pm e_2, \pm e_3\}$ be the outer unit normal vector to that face. (Note that by the definition of the set $E_\eps$ in \eqref{def:recoverysequence}, $Q_\eps(E_\eps)$ is also an $\eta_\eps$-scale cubic set, i.e., it consists of the union of cubes of sidelength $\eta_\eps$.)\EEE 
\ Then, 
\begin{align*}
\int_{S} \varphi(\nu) \,\mathrm{d}\mathcal{H}^2   =  \eta_\varepsilon^2 \sum_{\xi \in V,\ \nu \cdot\xi  >0} c_\xi\geq  \sum_{j \in \mathbb{Z}_\varepsilon(Q_{\eta_\varepsilon}(i))} \quad  \sum_{\xi \in V,\  \nu \cdot \xi   >0}\varepsilon^2 c_\xi (1-\chi_{E_\varepsilon}( \BBB j \EEE +\varepsilon\xi))\,.
\end{align*}
Then, in view of \eqref{def: discrete perimeter} and the fact that $Q_\eps(E_\varepsilon) \subset \subset \Omega$ for $\eps>0$ small (see \eqref{def:recoverysequence} and recall that $E \subset \subset \Omega$), \EEE  by taking the union of all faces of $\partial Q_\eps(E_\eps)$, we get \EEE
\begin{align*}
\limsup_{\varepsilon\to 0} F_\varepsilon^\mathrm{per}(E_\varepsilon) \leq \limsup_{\varepsilon\to 0}\mathrm{Per}_\varphi(Q_\eps(E_\varepsilon))\,.
\end{align*}
Hence, in order to obtain \eqref{ineq:limsupetasurf}, it suffices to prove that
\begin{align*}
\limsup_{\varepsilon\to 0} \mathrm{Per}_\varphi(Q_\varepsilon(E_\varepsilon))\leq \mathrm{Per}_\varphi(E)\,.
\end{align*}
This claim follows exactly as the proof in Step 2, by noting that for every $k=1,2,3$
$$|\mathcal{H}^2(\{x\in \partial E\colon \nu_E(x)=\pm e_k\})-\mathcal{H}^2(\{x\in \partial 
Q_{\varepsilon}(E_\eps) \EEE \colon \nu_{Q_{\varepsilon}(E)}(x)=\pm e_k\})|\to 0 \quad \text{as} \quad \varepsilon\to 0.$$
\EEE Therefore, \eqref{ineq:limsupetasurf} follows. It finally remains to prove that
\begin{align}\label{eq:limsupcurv}
\limsup_{\varepsilon \to 0} F_\varepsilon^\mathrm{curv}(E_\varepsilon) =0\,.
\end{align}
To this end, we first estimate the number of vertices $i \in \varepsilon\mathbb{Z}^3$ such that $E_\varepsilon$ is not flat in $Q_{\eta_{\varepsilon}}(i)$. Recall that  $\partial E$ is polyhedral, and let $\Sigma\subset \partial E$ be the union of edges such that for every $x\in \Sigma$ the tangent space $\mathrm{Tan}_{x}\partial E$ does not exist. Since $\nu_E \in \{\pm e_1, \pm e_2, \pm e_3\}$ 
$\mathcal{H}^2$-a.e., \EEE it holds that for all $i \in \varepsilon\mathbb{Z}^3 \setminus (\Sigma)_{\sqrt{3}\eta_\varepsilon}$ the set $\partial E_\varepsilon$ is flat in $Q_{\eta_\varepsilon}(i)$. Moreover, for $\eta_\varepsilon>0$ small enough, we can estimate
\begin{align}\label{ineq:cardestimatelimsupcurv}
\#\{i \in \varepsilon\mathbb{Z}^3 \colon \mathrm{dist}(i,\Sigma)\leq \sqrt{3}\eta_\varepsilon\} \leq \varepsilon^{-3}  \mathcal{L}^3((\Sigma)_{\sqrt{3}\eta_\varepsilon}) \leq  C\varepsilon^{-3}\eta_\varepsilon^2\mathcal{H}^1(\Sigma)\,.
\end{align}
Note that, due to \eqref{def:recoverysequence} and the structural assumption (iii) after the definition of the discrete curvature energy in \eqref{def: curvature energy}, we have that
\begin{align}\label{ineq:celllimsup}
W_{\varepsilon,\mathrm{cell}}^\mathrm{curv}(i,E_\varepsilon) \leq C\gamma_{\varepsilon} \eta_\varepsilon^{-1-q} \text{ for all such } i\,.
\end{align}
Therefore, using also assumption (i) below \eqref{def: curvature energy}, together with \eqref{ineq:cardestimatelimsupcurv} and \eqref{ineq:celllimsup}, we obtain
\begin{align*}
F_\varepsilon^\mathrm{curv}(E_\varepsilon) =\sum_{i \in 
\Z_\eps(U\EEE) \EEE} \varepsilon^3 W_{\varepsilon,\mathrm{cell}}^\mathrm{curv}(i,E_\varepsilon)  \leq C\varepsilon^3 \gamma_{\varepsilon}\eta_{\varepsilon}^{-1-q}\#\{i \in \varepsilon\mathbb{Z}^3 \colon \mathrm{dist}(i,\Sigma)\leq \sqrt{3}\eta_\varepsilon\}  \leq C\gamma_{\varepsilon}\eta_{\varepsilon}^{1-q}\mathcal{H}^1(\Sigma)\,.
\end{align*}
Due to \eqref{eq:rate gammadelta}, $\gamma_\varepsilon\eta_\varepsilon^{1-q}\to 0$ as $\varepsilon\to 0$, and hence we obtain \eqref{eq:limsupcurv}. Now \eqref{ineq:limsupetasurf} and \eqref{eq:limsupcurv} imply \eqref{ineq:limsupeta}. This concludes the proof.
\end{step}
\end{proof}

\section*{Acknowledgements} 
This work was supported by the DFG project FR 4083/3-1 and by the Deutsche Forschungsgemeinschaft (DFG, German Research Foundation) under Germany's Excellence Strategy EXC 2044 -390685587, Mathematics M\"unster: Dynamics--Geometry--Structure.
\EEE 

\typeout{References}

 \end{document}